\documentclass[11pt,pdftex]{amsart}
\usepackage{amsmath,amscd,latexsym}
\usepackage{amssymb}
\usepackage{amsbsy}
\usepackage{amsfonts}
\usepackage[all]{xy}
\usepackage{color}
\usepackage[dvipsnames]{xcolor}%textcolorの指定色を増やす
\usepackage{graphicx}
\usepackage{subcaption}
\usepackage[abs]{overpic}
\usepackage{multirow}
\usepackage{array}
\usepackage{paralist}
\usepackage{titletoc}
\usepackage{url}
\usepackage{comment}

\theoremstyle{plain}
\newtheorem{theorem}{Theorem}[section]
\newtheorem{proposition}[theorem]{Proposition}
\newtheorem{lemma}[theorem]{Lemma}
\newtheorem{corollary}[theorem]{Corollary}
\newtheorem{remark}[theorem]{Remark}
\newtheorem{definition}[theorem]{Definition}
\newtheorem{notation}[theorem]{Notation}
\newtheorem{definition-notation}[theorem]{Definition and Notation}

\newtheorem{main theorem}[theorem]{Main Theorem}

\newtheorem{claim}[theorem]{Claim}

\newcommand{\interior}{\operatorname{int}}

\newcommand{\closure}{\operatorname{cl}}
\newcommand{\length}{\operatorname{length}}

\newcommand{\Lk}{\operatorname{Lk}}

\newcommand{\ZZ}{\mathbb{Z}}
\newcommand{\QQ}{\mathbb{Q}}
\newcommand{\RR}{\mathbb{R}}
\newcommand{\CC}{\mathbb{C}}
\newcommand{\HH}{\mathbb{H}}
\newcommand{\EE}{\mathbb{E}}
\newcommand{\QQQ}{\hat{\mathbb{Q}}}
\newcommand{\RRR}{\hat{\mathbb{R}}}
\newcommand{\CCC}{\hat{\mathbb{C}}}
\newcommand{\HHH}{\bar{\mathbb{H}}}

\newcommand{\Diagram}{D}
\newcommand{\Poly}{\mathcal{P}}
\newcommand{\PPoly}{\operatorname{P}}
\newcommand{\tPoly}{\tilde{\mathcal{P}}}
\newcommand{\hPoly}{\hat{\mathcal{P}}}
\newcommand{\Cubing}{\mathcal{M}}
\newcommand{\tCubing}{\tilde{\mathcal{M}}}
\newcommand{\XCubing}{\mathcal{X}}
\newcommand{\tXCubing}{\tilde{\mathcal{X}}}

\newcommand{\Half}{\mathcal{B}}
\newcommand{\bHalf}{\bar{\mathcal{B}}}
\newcommand{\bSigma}{\bar{\Sigma}}
\newcommand{\Bplane}{H}
\newcommand{\Sbw}{\mathcal{S}_{bw}}
\newcommand{\tSbw}{\tilde{\mathcal{S}}_{bw}}
\newcommand{\Sw}{\mathcal{S}_{w}}
\newcommand{\tSw}{\tilde{\mathcal{S}}_{w}}
\newcommand{\Sb}{\mathcal{S}_{b}}
\newcommand{\tSb}{\tilde{\mathcal{S}}_{b}}
\newcommand{\BF}{\mbox{$\operatorname{BF}$}}
\newcommand{\crossing}{\mbox{$\mathcal{C}$}}
\newcommand{\Graph}{\mathcal{G}}
\newcommand{\DD}{\mathcal{D}}

\newcommand{\pu}{p_u}

\newcommand{\PSL}{\mbox{$\operatorname{PSL}$}}
\newcommand{\SL}{\mbox{$\operatorname{SL}$}}

\newcommand{\Isom}{\mbox{$\operatorname{Isom}$}}
\newcommand{\Fix}{\mbox{$\operatorname{Fix}$}}
\newcommand{\PFix}{\mbox{$\operatorname{PFix}$}}
\newcommand{\Aut}{\mbox{$\operatorname{Aut}$}}

\newcommand{\Stab}{\mbox{$\operatorname{Stab}$}}

\newcommand{\Ball}{\mbox{\boldmath$B$}^3}

\newcommand{\pNE}{\mbox{$\mathrm{NE}$}}
\newcommand{\pNW}{\mbox{$\mathrm{NW}$}}
\newcommand{\pSW}{\mbox{$\mathrm{SW}$}}
\newcommand{\pSE}{\mbox{$\mathrm{SE}$}}

\makeatletter
\renewcommand\subsection{\@startsection{subsection}{2}{0mm}
    {-10.5dd plus-8pt minus-4pt}{10.5dd}
    {\normalsize\upshape}}
\makeatother

\begin{document}
\title[Two-parabolic-generator subgroups]
{Two-parabolic-generator subgroups of hyperbolic $3$-manifold groups}
\author{Shunsuke Sakai}
\address{Gifu Higashi High School\\
4-17-1, noisshiki, Gifu City
500-8765, Japan}
\email{shunsuke463@gmail.com}

\author{Makoto Sakuma}
\address{Advanced Mathematical Institute\\
Osaka Metropolitan University\\
3-3-138, Sugimoto, Sumiyoshi, Osaka City
558-8585, Japan}
\address{Department of Mathematics\\
Faculty of Science\\
Hiroshima University\\
Higashi-Hiroshima, 739-8526, Japan}
\email{sakuma@hiroshima-u.ac.jp}
\subjclass[2010]{Primary 57M50, Secondary 57K10}
\keywords{two-bridge link, alternating link, parabolic transformation, cubed complex}

%\date{}

\begin{abstract}
We give a detailed account of Agol's theorem and his proof
concerning two-meridional-generator subgroups
of hyperbolic $2$-bridge link groups,
which is included in the slide of his talk 
at the Bolyai conference 2001. %held in Budapest in 2001.
We also give a generalization of the theorem
to two-parabolic-generator subgroups
of hyperbolic $3$-manifold groups,
which gives a refinement of a result due to Boileau-Weidmann.
\end{abstract}

\maketitle

\section{Introduction}
\label{sec:intro}

Adams proved in~\cite[Theorem~4.3]{Adams96} that 
the fundamental group of a finite volume hyperbolic $3$-manifold is
generated by two parabolic elements
if and only if the $3$-manifold is homeomorphic to the complement of a $2$-bridge link 
which is not a torus link.
Moreover, he also proved that the pair consists of meridians.
This refines
the result of Boileau-Zimmermann~\cite[Corollary~3.3]{Boileau-Zimmermann}
that a link in $S^3$ is a $2$-bridge link if and only if its link group is generated by two meridians.
Adams also proved that
(i) each hyperbolic $2$-bridge link group admits
only finitely many distinct parabolic generating pairs up to 
equivalence %conjugacy
\cite[Corollary~4.1]{Adams96} and (ii)
for the figure-eight knot group,
the upper and lower meridian pairs are the only parabolic generating pairs
up to equivalence~\cite[Corollary~4.6]{Adams96}.
Here, a {\it parabolic generating pair} of a non-elementary Kleinian group
$\Gamma$ is an unordered pair 
of two parabolic transformations 
that generates $\Gamma$.
Two parabolic generating pairs $\{\alpha,\beta\}$
and $\{\alpha',\beta'\}$ of $\Gamma$ are {\it equivalent}
if $\{\alpha',\beta'\}$ is equal to $\{\alpha^{\epsilon_1},\beta^{\epsilon_2} \}$
for some $\epsilon_1, \epsilon_2 \in \{\pm1\}$
up to simultaneous conjugation.

Agol~\cite{Agol2001} announced the following theorem
which generalizes and refines these results to all 
non-free Kleinian groups generated by two parabolic transformations.

\begin{theorem}[Agol~\cite{Agol2001}]
\label{thm:Agol}
Let $\Gamma$ be a non-free Kleinian group generated by two non-commuting  
parabolic elements.
Then one of the following holds.
\begin{enumerate}[\rm (1)]
\item
$\Gamma$ is conjugate to a hyperbolic $2$-bridge link group.
Moreover, every hyperbolic $2$-bridge link group 
has precisely two parabolic generating pairs up to equivalence.

\item
$\Gamma$ is conjugate to a Heckoid group.
Moreover, every Heckoid group has a unique parabolic generating pair up to equivalence.
\end{enumerate}
\end{theorem}

For an explicit description of the theorem, including 
the definition of a Heckoid group,
see Akiyoshi-Ohshika-Parker-Sakuma-Yoshida~\cite{AOPSY} 
and Aimi-Lee-Sakai-Sakuma~\cite{ALSS}
(cf. Lee-Sakuma \cite{Lee-Sakuma16}),
which give a full proof of the classification of 
non-free, two-parabolic-generator Kleinian groups
and an alternative proof of the classification
of parabolic generating pairs, respectively.
In the recent interesting articles~\cite{EMS} and~\cite{Parker-Tan} 
by Parker-Tan and
Elzenaar-Martin-Schillewaert, respectively,
we can find very beautiful pictures, produced by Yasushi Yamashita 
upon request of Caroline Series,
that nicely illustrate Theorem~\ref{thm:Agol} 
(see also Figure 0.2b in Akiyoshi-Sakuma-Wada-Yamashita~\cite{ASWY}).

The two parabolic generating pairs 
of a hyperbolic $2$-bridge link group 
in the second statement of Theorem~\ref{thm:Agol}(1)
are the {\it upper} and {\it lower meridian pairs} 
illustrated in 
Figure~\ref{fig:upper-lower-tunnel}
(cf.~Section~\ref{sec:results}). 
The assertion was obtained 
in~\cite{Agol2001}
as a consequence of the following more detailed result,
together with Adams' result~\cite[Theorem~4.3]{Adams96}
that every parabolic generating pair
of a hyperbolic $2$-bridge link group consists of meridians.

\begin{theorem}[Agol~\cite{Agol2001}]
\label{Theorem1-0}
Let $L\subset S^3$ be a hyperbolic $2$-bridge link.
Then any non-commuting meridian pair in the link group $G(L)$ which is not equivalent to the upper nor lower meridian pair
generates a free Kleinian group which is geometrically finite. 
\end{theorem}

\begin{figure}
  \centering
  \begin{overpic}[width=0.25\columnwidth]{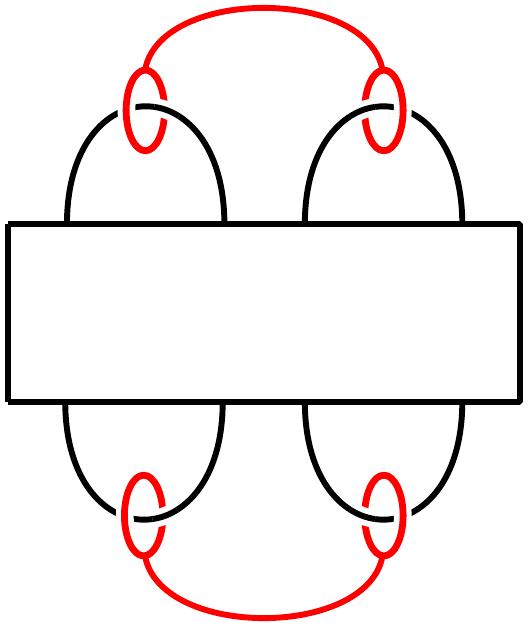}
    \put(46,121){\textcolor{red}{$\tau_+$}}
    \put(46,-6){\textcolor{red}{$\tau_-$}}
    \put(38,55){Braid}
  \end{overpic}
  \caption{The upper and lower meridian pairs of a $2$-bridge link group.
  The proper arcs $\tau_+$ and $\tau_-$ in the exterior $M(L)$ of a $2$-bridge link $L\subset S^3$ are the upper and lower tunnels, respectively.
  Each of the meridian pairs represented by $\tau_+$ and $\tau_-$ generates the link group $G(L)=\pi_1(S^3\setminus L)$.}
  \label{fig:upper-lower-tunnel}
\end{figure}

The main purpose of this paper is to give a detailed account of
Agol's beautiful proof of Theorem~\ref{Theorem1-0} included in the slide~\cite{Agol2001}.
A key ingredient of the proof is non-positively curved cubed decompositions 
of alternating link exteriors 
in which the checkerboard surfaces are hyperplanes (Proposition \ref{prop:cubing}). 
According to Rubinstein~\cite[p.3177]{Sakuma-Yokota},
such cubed decompositions were first 
found by Aitchison, though he did not publish the result.
They were rediscovered by D. Thurston~\cite{DThurston} 
and described in detail by Yokota~\cite{Yokota2011}
(cf.~\cite{ALR, Sakuma-Yokota, Sakai}).
The cubed decompositions play essential roles in the proofs of
(i) Proposition~\ref{prop:qf} which says that
the checkerboard surfaces for hyperbolic alternating links 
are quasi-fuchsian and 
(ii) Propositions~\ref{prop:disjoint-half-space}
and~\ref{prop:disjoint-open-disk}
concerning the disks bounded by the limit circles 
associated with checkerboard surfaces
in the ideal boundary $\CCC$ of 
the universal covering $\HH^3$ of the 
hyperbolic alternating link complement.
The proof of Theorem \ref{Theorem1-0} is
completed by applying Proposition~\ref{prop:ping-pong}
(a variant of Klein-Maskit combination theorem 
proved by using Maskit-Swarup~\cite{Maskit-Swarup})
to the action
of meridian pairs on $\CCC$ 
by using Proposition~\ref{prop:disjoint-open-disk}.
(See Figures~\ref{fig:omega-positive} and~\ref{fig:omega-zero-noncrossing},
which are copied from~\cite{Agol2001}.)
\medskip

%以下の部分の改善を下のtextcolor環境で提案
\begin{comment}
Building on Theorems~\ref{thm:Agol} and~\ref{Theorem1-0},
and by using (i) the covering theorem of Canary~\cite{Canary} 
together with the tameness theorem
established by Agol~\cite{Agol2004} 
and Calegari-Gabai~\cite{Calegari-Gabai}
(see also Soma~\cite{Soma} and Bowditch~\cite{Bowditch}),
and (ii) the result of 
Millichap and Worden~\cite{Millichap-Worden} concerning the commensurable classes of $2$-bridge link groups,
we also prove the following generalization of
Theorem~\ref{Theorem1-0}.
\end{comment}

Building on Theorems~\ref{thm:Agol} and~\ref{Theorem1-0},
we also prove the following generalization of Theorem~\ref{Theorem1-0}.

\begin{theorem}\label{thm:generalization} 
Let $X = \HH^3/G$ be an orientable, complete, hyperbolic $3$-manifold, 
$\{ \mu_1, \mu_2 \}$ a pair of non-commuting parabolic elements 
of $G$,
and $\Gamma=\langle \mu_1, \mu_2 \rangle$ the subgroup of $G$
generated by $\{ \mu_1, \mu_2 \}$.
  Then one of the following holds.
\begin{enumerate}
\item[{\rm (1)}] 
$\Gamma$ is a rank $2$ free group.
        
\item[{\rm (2)}] 
$\Gamma$ is equal to $G$, 
and it is a hyperbolic $2$-bridge link group.
Moreover, $\{ \mu_1, \mu_2 \}$ is equivalent to the upper or lower meridian pair.

\item[{\rm (3)}] 
$\Gamma$ is an index $2$ subgroup of $G$, 
where $\Gamma$ is 
the link group of 
a $2$-component hyperbolic $2$-bridge link,
and $G$ is the link group of 
a rational link in the projective $3$-space $P^3$.
Moreover,
$\{ \mu_1, \mu_2 \}$, as a subset of $\Gamma$, is 
equivalent to the upper or lower meridian pair
in the $2$-bridge link group, 
and $\{ \mu_1, \mu_2 \}$, as a subset of $G$, 
consists of meridians of the rational link.
\end{enumerate}
Moreover, if $X$ has finite volume, then
the conclusion (1) is replaced with the following 
finer conclusion.
\begin{enumerate}
\item[{\rm (1')}] 
$\Gamma$ is a rank $2$ free Kleinian group which is geometrically finite.
\end{enumerate}
\end{theorem}

See Definition~\ref{def:rational-link-P}
for the definition of a rational link in $P^3$,
and see
Remark~\ref{rem:statement-thm-general}
for a detailed description of the statement (3) in the above theorem.
This theorem gives a refinement of
the result by Boileau-Weidmann
\cite[Proposition 2]{Boileau-Weidmann}
concerning subgroups generated by two parabolic primitive elements
of the fundamental group of an orientable, complete, hyperbolic $3$-manifold of finite volume.
 The proof of Theorem~\ref{thm:generalization} is based on 
 (i) the result of Millichap-Worden~\cite{Millichap-Worden} concerning the commensurable classes of $2$-bridge link groups
 and (ii) the covering theorem of Canary~\cite{Canary}
  together with the tameness theorem established by Agol~\cite{Agol2004} and Calegari-Gabai~\cite{Calegari-Gabai} (see also Soma~\cite{Soma} and Bowditch~\cite{Bowditch}).

\medskip
This paper is organized as follows.
In Section~\ref{sec:results}, we reformulate
the main Theorem~\ref{Theorem1-0} into Theorem~\ref{Theorem1},
by using the correspondence between the meridian pairs up to equivalence and
the proper arcs in the link exterior up to proper homotopy.
We also state Theorem~\ref{Theorem1-g}
concerning general alternating links which 
is implicitly included in~\cite{Agol2001}.
In Section~\ref{sec:checkerboard-surface},
we recall the key fact that 
the checkerboard surfaces associated with prime alternating link diagrams
of hyperbolic alternating links are quasi-fuchsian
(Proposition~\ref{prop:qf}).
In Section~\ref{sec:meridian},
we describe the actions of meridians on
the ideal boundary
$\CCC$ of the hyperbolic space $\HH^3$,
and give a sufficient condition
for a meridian pair to generate a free Kleinian group
which is geometrically finite (Proposition~\ref{prop:ping-pong}).
The proposition is a basic tool for the proof of 
Theorems~\ref{Theorem1} and~\ref{Theorem1-g}. 
In Section~\ref{sec:cat0},
we quickly recall fundamental facts 
concerning non-positively curved spaces,
which is used in Sections~\ref{sec:NPC} and~\ref{sec:chekerboard-decomposition}.
In Section~\ref{sec:NPC},
we describe
non-positively curved cubed decompositions of 
alternating link exteriors (Proposition~\ref{prop:cubing}),
and study relative positions
of \lq\lq checkerboard hyperplanes'' and 
\lq\lq peripheral hyperplanes'',
the components of 
the inverse images of checkerboard surfaces  
and peripheral tori, respectively, 
in the universal cover $\tilde X$ of 
a hyperbolic alternating link complement $X$
(Proposition~\ref{pro:connected-intersection}).
In Section~\ref{sec:chekerboard-decomposition},
we review the ideal polyhedral decomposition of $X$
from the view point of the non-positively curved cubed decompositions.
Then 
we prove Proposition~\ref{prop:disjoint-half-space}
concerning relative positions of closed half-spaces in $\tilde X$
bounded by checkerboard hyperplanes.
In Section~\ref{sec:Butterflies-checkerboard},
we use Proposition~\ref{prop:disjoint-half-space} to prove 
the key proposition,
Proposition~\ref{prop:disjoint-open-disk},
concerning discs, 
in the ideal boundary $\CCC$
of $\tilde X=\HH^3$,
bounded by 
the limit circles of checkerboard hyperplanes.
In Section~\ref{sec:proof-maintheorem}, we prove
Theorems~\ref{Theorem1} and~\ref{Theorem1-g}
(and so  Theorem~\ref{Theorem1-0}),
by using Propositions~\ref{prop:ping-pong} and
\ref{prop:disjoint-open-disk}.
In Section~\ref{sec:generalization},
we prove Theorem~\ref{thm:generalization}
after introducing and studying rational links in $P^3$.

\medskip
{\bf Acknowledgement.}
The second author would like to thank Ian Agol 
for sending the slide of his talk~\cite{Agol2001},
encouraging him (and any of his collaborators) to write up the proof,
and describing key ideas of the proof.
He would also like to thank Michel Boileau for enlightening conversations.
Both authors would like to thank Iain Aitchison
for teaching them the key ideas of polyhedral decompositions 
and cubical decompositions of alternating link exteriors
and for numerous invaluable discussions.
They would also like to thank Takuya Katayama and Ian Leary 
for invaluable discussions and informations concerning 
Proposition~\ref{prop:cubing2X}. 
Their thanks also go to 
Hirotaka Akiyoshi, Tetsuya Ito, 
Yuya Koda, Yohei Komori, Hideki Miyachi,
Makoto Ozawa,
and Yuta Taniguchi
for enlightening conversations and encouragements.
The second author is supported by JSPS KAKENHI Grant Number JP20K03614
and by Osaka Central Advanced Mathematical Institute 
(MEXT Joint Usage/Research Center on Mathematics and Theoretical Physics JPMXP0619217849).

\section{Reformulation of Theorem~\ref{Theorem1-0}}
\label{sec:results}

Let $L$ be a link in $S^3$,
$X=X(L):=S^3\setminus L$ the {\it link complement},
and $M=M(L):=S^3\setminus \interior N(L)$, the {\it link exterior},
where $N=N(L)$ is a regular neighborhood of $L$.
The {\it link group} $G=G(L)$ of $L$ is the fundamental group $\pi_1(M)=\pi_1(X)$.
A {\it meridian} of $L$ is an element $\mu$ of $G$
which is represented by a based loop freely homotopic to
a {\it meridional loop} in $\partial N$,
i.e., a simple loop that bounds an essential disk in $N$.

A {\it meridian pair} is an unordered pair $\{\mu_1,\mu_2\}$ of meridians of $L$.
Two meridian pairs $\{\mu_1,\mu_2\}$ and $\{\mu_1',\mu_2'\}$ are {\it equivalent}
if $\{\mu_1',\mu_2'\}$ is equal to 
$\{g\mu_1^{\varepsilon_1}g^{-1}, g\mu_2^{\varepsilon_2}g^{-1}\}$
for some $\varepsilon_1, \varepsilon_2\in\{\pm 1\}$
and $g\in G$.

Note that there is a bijective correspondence between
the set of meridian pairs of $L$ up to equivalence
and the set of proper paths in $M$
up to proper homotopy (cf.~\cite{Adams96},~\cite[Section 2]{Lee-Sakuma16}
and Lemma~\ref{lem:meridian-parabolic2}(2)).
Here a {\it proper path} in $M$
is a path (a continuous image of a closed interval)
which intersects $\partial M$ precisely at the endpoints.
Two proper paths in $M$ are {\it properly homotopic} in $M$
if they are homotopic keeping the condition that 
the endpoints are contained in $\partial M$. 

Assume that $L$ is {\it hyperbolic}, i.e.,
the complement $X$ admits a complete hyperbolic structure of finite volume.
Then the meridian pair $\{\mu_1,\mu_2\}$ is {\it commuting}
(i.e., $\mu_1\mu_2=\mu_2\mu_1$)
if and only if 
the corresponding proper path is {\it inessential},
i.e., properly homotopic to an arc in $\partial M$
(cf.~Lemma~\ref{lem:non-commutative}).
In other words, $\{\mu_1,\mu_2\}$ is non-commuting
if and only if the proper path is {\it essential},
i.e., not inessential.
If $L$ is a $2$-bridge link and if the arc is properly homotopic to
the upper or lower tunnel of $L$,
then $\{\mu_1,\mu_2\}$ generates the link group $G$
(see Figure~\ref{fig:upper-lower-tunnel}). 
Thus Theorem~\ref{Theorem1-0} is reformulated as follows.

\begin{theorem}
\label{Theorem1}
Let $L\subset S^3$ be a hyperbolic $2$-bridge link.
Let $\gamma$ be an essential proper path in the link exterior $M(L)$, 
and 
let $\{\mu_1,\mu_2\}$ be the meridian pair 
in the link group $G(L)$ represented by  $\gamma$.
Assume that $\gamma$ is not 
properly homotopic to the upper nor lower tunnel of $L$.
Then $\{\mu_1,\mu_2\}$ generates a rank $2$ free Kleinian group which is geometrically finite. 
\end{theorem}

Agol's proof Theorem~\ref{Theorem1} in~\cite{Agol2001} 
actually includes a proof of 
the following result concerning
hyperbolic alternating links.

\begin{theorem}
\label{Theorem1-g}
Let $L\subset S^3$ be a hyperbolic alternating link
and $\Diagram$ a prime alternating diagram of $L$.
Let $\{\mu_1,\mu_2\}$ be a non-commuting meridian pair
and 
$\gamma$ an essential proper path in the link exterior $M(L)$
that represents the pair $\{\mu_1,\mu_2\}$.
If $\gamma$ is not properly homotopic to a crossing arc 
(with respect to the diagram $\Diagram$),
then $\{\mu_1,\mu_2\}$ generates a rank $2$ free Kleinian group which is geometrically finite.
\end{theorem}

\section{Checkerboard surfaces for alternating links}
\label{sec:checkerboard-surface}
{\it In the remainder of this paper,
$L\subset S^3$ denotes a hyperbolic alternating link and $\Diagram\subset S^2$ denotes
a prime alternating diagram of $L$, 
except in Sections~\ref{sec:NPC}
and~\ref{sec:chekerboard-decomposition},
where we assume only that $L$ is a prime alternating link.} 
Here a link diagram is {\it prime} if
(i) it contains at least one crossing and (ii)
for every simple loop $\alpha$ in the projection plane,
if $\alpha$ meets the diagram transversely in exactly two points,
then $\alpha$ bounds a disk that contains no crossings 
of the diagram.
It should be noted that a prime alternating diagram of a prime link
is connected (as a plane graph) and reduced (i.e., contains
no nugatory crossings.)

We pick two points $v_+$ and $v_-$ in $S^3$, identify $S^3 \setminus \{v_+,v_-\}$ with $S^2 \times \RR$ so that $\lim_{t\to\pm\infty}(x,t)=v_{\pm}$
for $x\in S^2$. 
The diagram $\Diagram$ is regarded as a $4$-valent graph in $S^2 \times \{0\}$, 
and we assume $L \subset \Diagram \times [-1,1]$.
For each crossing $c$ of $\Diagram$, 
we assume $L \cap \left(c \times [-1,1]\right) = c \times \{-1,1\}$. 
We call the point $c_+ := c \times 1$ 
(resp.~$c_- := c \times (-1)$) the {\it over} 
(resp.~{\it under}) {\it crossing point} of $L$ at $c$,
and call $c \times [-1,1]$
the \textit{crossing arc of $L$} at $c$.
The intersection of $c \times [-1,1]$ with the link exterior $M$
(resp.~the link complement $X$)
is called the {\it crossing arc in $M$} 
(resp.~the {\it open crossing arc in $X$})
at $c$.
We assume that the crossing arc $c \times [-1,1]$ is oriented so that $c_-$ and $c_+$, respectively, are the initial and terminal points.

We also assume that $L$ coincides with $\Diagram$
outside crossing balls,
regular neighborhoods in $S^3$ of the crossing
arcs at the crossings of $\Diagram$.
We  color the complementary regions of $\Diagram$ in $S^2$
alternatively black and white.
Then there is a compact, connected surface $S_b$ (resp.~$S_w$) 
bounded by $L$
that coincides with the black (resp.~white) regions outside the crossing balls
and intersects each crossing ball in a twisted rectangle:
it is called the {\it black} (resp.~{\it white}) {\it surface} for $L$.
It should be noted that $S_b$ and $S_w$ intersect transversely 
along the crossing arcs (cf.~Figure~\ref{fig:Aitchison2}(a)
in Section~\ref{sec:NPC}).
Moreover, there is a natural bijective correspondence 
between the components of
$(S_b\cup S_w)\setminus (S_b\cap S_w)$
and the regions of $\Diagram$.
We occasionally refer to each of $S_b$ and $S_w$ as
a {\it checkerboard surface}
and denote it by $S$.

For each checkerboard surface $S$,
we assume that $S$ intersects 
the regular neighborhood $N$ of $L$ in a collar neighborhood of $\partial S$
and so $S\cap M$ is properly embedded in $M$.
We refer to $S\cap M\subset M$ (resp.~$S\cap X\subset X$)
a {\it checkerboard  surface in $M$}
(resp.~an {\it open checkerboard  surface in $X$}),
and continue to denote it by $S$.

The following key proposition is implicitly included 
in the slide~\cite{Agol2001},
and its proof following Agol's suggestion is given by Adams
\cite[Theorem 1.9]{Adams07}.
The proof depends on the fact that
every hyperbolic alternating link complement admits  
a non-positively curved cubed decomposition
in which checkerboard surfaces are hyperplanes 
(see Section~\ref{sec:NPC}).
Except for the existence of such a decomposition,
essentially the same arguments had been given by
Aitchison-Rubinstein~\cite[Lemma and its proof in p.146]{Aitchson-Rubinstein_90a}
in a more general setting.
See
Futer-Kalfagianni-Purcell 
\cite[Theorem 1.6]{FKP2014}
for a generalization.

\begin{proposition}
\label{prop:qf}
Let $L\subset S^3$ be a hyperbolic alternating link,
and $S$ a checkerboard surface obtained from a prime alternating diagram
$\Diagram$ of $L$.
Then $S$ is quasi-fuchsian. 
\end{proposition}

To explain the meaning of the proposition,
let $\pu:\tilde X\to X$ be the universal covering,
and identify the link group $G=\pi_1(X)$ with the covering transformation group 
$\Aut(\tilde X)$.
Since $L$ is hyperbolic,
$\tilde X$ is identified with the hyperbolic space $\HH^3$
and $G=\Aut(\tilde X)$ is regarded as a Kleinian group.
Then $S$ being {\it quasi-fuchsian} means that 
$\pi_1(S)$ injects into $\pi_1(X)=G$ and 
the Kleinian group $\pi_1(S) < G < \PSL(2,\CC)$
satisfies the following condition:
if $S$ is orientable then $\pi_1(S)$ is a quasi-fuchsian group (cf.~\cite[p.120, Definition]{Matsuzaki-Taniguchi}), and
if $S$ is non-orientable then the index $2$ subgroup of $\pi_1(S)$
corresponding to the orientation double cover is a quasi-fuchsian group.  
Since
the action of a quasi-fuchsian group on 
the $3$-ball $\HHH^3=\HH^3\cup \CCC$ is topologically conjugate to the action of a fuchsian group
(see~\cite[Theorem 5.31]{Matsuzaki-Taniguchi}),
we obtain the following corollary.

\begin{corollary}
\label{cor:qf}
Let $L\subset S^3$ be a hyperbolic alternating link,
and $S$ a checkerboard surface obtained from a prime alternating diagram
$\Diagram$ of $L$.
Let $\Sigma$ be a component of the inverse image $\pu^{-1}(S) \subset \tilde X=\HH^3$.
Then $\Sigma$ is an open disk properly embedded in $\HH^3$,
and it divides $\HH^3$ into two half-spaces, $B^-$ and $B^+$,
which satisfy the following conditions.
\begin{enumerate}
\item[\rm(1)]
$\HH^3=B^-\cup B^+$ and $\Sigma=B^-\cap B^+$.
\item[\rm(2)]
The closure $\bSigma$ of $\Sigma$ in $\HHH^3=\HH^3\cup\CCC$
is a disk properly embedded in $\HHH^3$, and 
$(\HHH^3,\bSigma)$ is homeomorphic to the standard ball pair
$(B^3,B^2)$, where $B^3$ is the unit $3$-ball in $\RR^3$
and $B^2$ is the intersection of $B^3$ with the $x$-$y$ plane.
\item[\rm(3)]
The closures $\bar B^{\pm}$ of $B^{\pm}$ in $\HHH^3$
are $3$-balls, such that
\[
\HHH^3=\bar B^-\cup\bar B^+, \quad \bSigma=\bar B^-\cap\bar B^+.
\]
\item[\rm(4)]
$\partial\bSigma$ is a circle in $\CCC$ which divides
$\CCC$ into two disks $\Delta^-:=\bar B^-\cap \CCC$ and 
$\Delta^+:=\bar B^+\cap \CCC$,
such that
$\CCC=\Delta^-\cup \Delta^+$ and 
$\partial\bSigma=\Delta^-\cap \Delta^+$.
\end{enumerate}
\end{corollary}

We call $\Sigma\subset\HH^3$ and
$\bSigma\subset\HHH^3$, respectively,
a {\it checkerboard plane} and
a {\it checkerboard disk}.
The {\it color} of $\Sigma$ (or $\bSigma$) is
defined to be black or white according
the color of 
the corresponding checkerboard surface $S$.
We call each of $B^{\pm}$ a {\it checkerboard half-space}
bounded by $\Sigma$.

\section{The action of meridian pairs on the ideal boundary of the hyperbolic space}
\label{sec:meridian}

In the reminder of the paper, 
we assume for convenience that the hyperbolic alternating link $L$ is {\it oriented},
and we use the terminology \lq\lq meridian'' and \lq\lq meridian pair''
in the following restricted sense:
A {\it meridian} of $L$ is an element of
the link group $G$
which is represented by an oriented closed path freely 
homotopic to
a meridional loop in $\partial N$ that has linking number $+1$ with $L$.
A {\it meridian pair} is an unordered pair $\{\mu_1,\mu_2\}$ of 
meridians of $L$ in the restricted sense.
Then two meridian pairs are equivalent
in the sense defined in Section~\ref{sec:results}
if and only if they are simultaneously conjugate. 
Of course, this does not affect the contents of 
Theorems~\ref{Theorem1} and~\ref{Theorem1-g}.

Recall that $\tilde X$ is identified with the hyperbolic space $\HH^3$
and $G=\Aut(\tilde X)$ is identified with a Kleinian group.
Thus a meridian $\mu\in G < \PSL(2,\CC)$ is parabolic,
and its action on $\HHH^3=\HH^3\cup \CCC$
has a unique fixed point,  which lies in $\CCC$.
The point is called the {\it parabolic fixed point} of $\mu$ and 
denoted by $\Fix(\mu)$.

Let $\PFix(G)\subset\CCC$ be the set of the parabolic 
fixed points of $G$,
i.e., the set of the parabolic fixed points of 
the parabolic elements of $G$.
For each $p\in \PFix(G)$,
the stabilizer $\Stab_G(p)$ of $p$ in $G$ is a rank $2$ free abelian group
which belongs to the conjugacy class of the fundamental group
of a component of $\partial M$.
Since every component of $\partial M$ contains a unique
meridian loop (which has linking number $+1$ with $L$) up to isotopy,
$\Stab_G(p)$ contains a unique meridian $\mu_p$ of the oriented link $L$.
We call $\mu_p$ the {\it meridian of $L$ at 
the parabolic fixed point $p$}.

\begin{lemma}
\label{lem:meridian-parabolic}
The maps $\mu\mapsto \Fix(\mu)$ and $p \mapsto \mu_p$, respectively,
determine the following bijective correspondence and its inverse:
\[
\mbox{\{meridians of $L$}\} \to \PFix(G).
\]
\end{lemma}

The following lemma is easily proved.

\begin{lemma}
\label{lem:non-commutative}
Let $\{\mu_1,\mu_2\}$ be a meridian pair represented by
a proper path $\gamma$ in the link exterior $M$.
Then the following conditions are equivalent.
\begin{enumerate}
\item[\rm(1)]
$\{\mu_1,\mu_2\}$ is commuting.
\item[\rm(2)]
$\Fix(\mu_1)=\Fix(\mu_2)$.
\item[\rm(3)]
$\gamma$ is inessential.
\end{enumerate}
\end{lemma}

We now describe the action of the meridian $\mu_p$ on $\HHH^3$.
To this end, we assume that 
$X\setminus M$ consists of open cusp neighborhoods, and therefore
$\tilde M:=\pu^{-1}(M)$
is a submanifold of $\tilde X=\HH^3$ bounded by 
disjoint horospheres $\{\Bplane_p\}_{p\in \PFix(G)}$.
Note that the Euclidean torus $\Bplane_p/\Stab_G(p)$ is a component of
$\partial M$ and every component of $\partial M$ is of this form.

Now, let $S\subset X$ be an open checkerboard surface for $L$.
We may assume that $S$ intersects transversely each component of $\partial M$ 
in a closed Euclidean geodesic.
For each $p\in \PFix(G)$, $\pu^{-1}(S)\cap \Bplane_p$ is a disjoint union of Euclidean lines 
$\{\ell_j\}_{j\in\ZZ}=\{\ell_j(p)\}_{j\in\ZZ}$ 
such that $\mu_p(\ell_j)=\ell_{j+1}$.
Let $\Sigma_j=\Sigma_j(p)$ be the checkerboard plane 
that is the component of $\pu^{-1}(S)$ such that
$\ell_j \subset \Sigma_j \cap \Bplane_p$.

\begin{lemma}
\label{lem:distinct-plane}
Under the above setting,
$\Sigma_j \cap \Bplane_p=\ell_j$
for each $p\in\PFix(G)$ and $j\in\ZZ$.
In other words, 
the checkerboard planes $\Sigma_j$ ($j\in \ZZ$) are all different.
\end{lemma}

\begin{proof}
Suppose to the contrary that $\Sigma_j=\Sigma_{j'}$
for some distinct integers $j$ and $j'$.
Let $\check\Sigma$ be the intersection of $\Sigma_j=\Sigma_{j'}$ and $\tilde M$.
Then $\check\Sigma$ is properly embedded in $\tilde M$,
and the image $\check S:=\pu(\check\Sigma)=S\cap M$ is a checkerboard surface in $M$.
Let $\tilde\alpha$ be a path in $\check\Sigma$
joining the boundary components $\ell_j$ and $\ell_{j'}$ of 
$\check\Sigma$.
Since $\tilde M$ is simply connected,
$\tilde\alpha$ is homotopic rel endpoints to a path in $H_p\subset \partial \tilde M$.
Thus the path $\alpha:=\pu\circ\tilde\alpha$ in $\check S$ is homotopic rel endpoints 
to a path in $\partial M$ inside $M$.
On the other hand, since $\ell_j\ne \ell_{j'}$,
$\alpha$ is not homotopic rel endpoints to 
$\partial \check S$ in $\check S$.
This contradicts \cite[Theorem 11.31]{Purcell}
which says that $\check S$ is $\pi_1$-essential, in particular, 
boundary $\pi_1$-injective (see \cite[Definition 11.30]{Purcell}).
\end{proof}

\begin{remark}
\label{rem:incompressible}
{\rm
See Proposition~\ref{pro:connected-intersection}(2)
for a direct geometric proof of the above lemma.
The $\pi_1$-essentiality of checkerboard surfaces
associated with prime alternating diagrams
had been proved by Aumann~\cite{Aumann}
(cf. Menasco-Thistlethwaite~\cite[Proposition 2.3]{Menasco-Thistlethwaite}).
See  
Ozawa~\cite[Theorem~3]{Ozawa1} and
\cite[Theorem~2.8]{Ozawa2} for generalizations.
}
\end{remark}

\begin{lemma}
\label{lem:balls}
{\rm (1)} There are checkerboard half-spaces 
$B_j^{\pm}=B_j^{\pm}(p)$ ($j\in\ZZ$) 
bounded by $\Sigma_j$
which satisfy the following conditions:
\begin{enumerate}
\item[\rm(a)] 
$\HHH^3=\bar B_j^-\cup \bar B_j^+$ and 
$\bSigma_j=\bar B_j^- \cap \bar B_j^+$.
\item[\rm(b)]
$\bar B_j^-\subset \bar B_{j+1}^{-}$ and
$\bar B_j^+\supset \bar B_{j+1}^{+}$. 
\item[\rm(c)]
$\mu_p(\bar B_j^{\pm})=\bar B_{j+1}^{\pm}$.
\end{enumerate}

{\rm (2)} Set $\Delta_j^{\pm}=\Delta_j^{\pm}(p):=\bar B_j^{\pm}(p)\cap \CCC$.
Then $\Delta_j^{\pm}$ are disks in $\CCC$
which satisfy the following conditions.
\begin{enumerate}
\item[\rm(a)] 
$\CCC=\Delta_j^-\cup \Delta_j^+$ and 
$\partial \bSigma_j=\Delta_j^- \cap \Delta_j^+$.
\item[\rm(b)]
$\Delta_j^-\subset \Delta_{j+1}^{-}$ and
$\Delta_j^+\supset \Delta_{j+1}^{+}$.
\item[\rm(c)]
$\mu_p(\Delta_j^{\pm})=\Delta_{j+1}^{\pm}$.
\end{enumerate}
\end{lemma}

In the above lemma, the symbol $\pm$ is used in the following way:
for example, $\mu_p(\bar B_j^{\pm})=\bar B_{j+1}^{\pm}$ means that
$\mu_p(\bar B_j^{\epsilon})=\bar B_{j+1}^{\epsilon}$ for each $\epsilon\in\{-,+\}$.
We apply this convention throughout the paper.

\begin{proof}
By Lemma~\ref{lem:distinct-plane}, $H_p\cap\Sigma_j$ 
is equal to the line $\ell_j$.
Observe that the line $\ell_j$
divides $H_p$ into two closed half-spaces $H_{p,j}^-$ and $H_{p,j}^+$,
where $\ell_{j\pm1}\subset H_{p,j}^{\pm}$
(see Figure~\ref{fig:Butterfly}(a)).
By Corollary~\ref{cor:qf},
$(\HHH^3,\bSigma_j)$ is a standard ball pair
and there are checkerboard half-spaces   
$B_j^{\epsilon}$ 
($\epsilon\in\{-,+\}$) bounded by $\Sigma_j$ 
which satisfy the condition (1-a),
such that $H_{p,j}^{\epsilon}\subset \bar B_j^{\epsilon}$.  
Since $H_{p,j}^{-}\subset H_{p,j+1}^{-}$ and 
$H_{p,j}^{+}\supset H_{p,j+1}^{+}$,
the condition (1-b) are satisfied.
Since $\mu_p(H_{p,j}^{\pm})=H_{p,j+1}^{\pm}$,
the condition (1-c) is also satisfied,
completing the proof of (1).

The assertion (2) follows from (1) and
the fact that 
$(\HHH^3,\bSigma_j)$ is a standard ball pair.
\end{proof}

\begin{figure}
  \centering
  \begin{subfigure}{0.3\columnwidth}
    \begin{overpic}[width=\columnwidth]{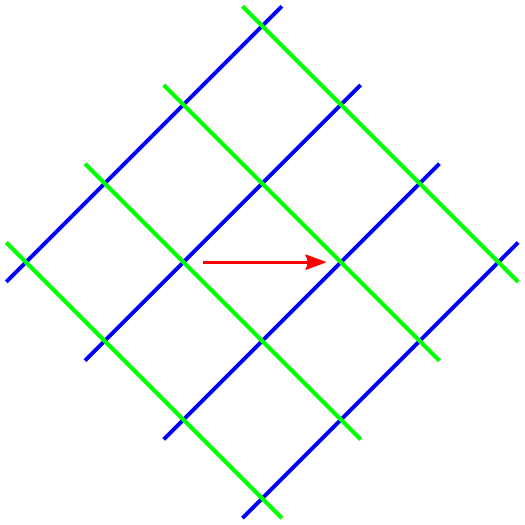}
      \put(56,52){\textcolor{red}{$\mu_p$}}
      \put(12,29){\textcolor{blue}{$\ell_j$}}
      \put(26,12){\textcolor{blue}{$\ell_{j+1}$}}
    \end{overpic}
    \subcaption{}
    \label{fig:Butterfly_a}
  \end{subfigure}
  \hspace{10mm}
  \begin{subfigure}{0.28\columnwidth}
    \begin{overpic}[width=\columnwidth]{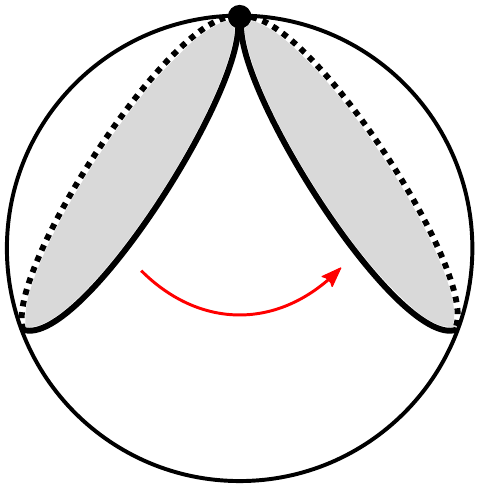}
      \put(-5,90){$\Delta_j^-$}
      \put(101,90){$\Delta_{j+1}^+$}
      \put(25,70){$\bSigma_j$}
      \put(74,70){\scriptsize$\bSigma_{j+1}$}
      \put(53,116){$p$}
      \put(52,33){\textcolor{red}{$\mu_p$}}
    \end{overpic}
    \subcaption{}
    \label{fig:Butterfly_b}
  \end{subfigure}
  \caption{
  (a) The action of $\mu_p$ on 
  $(H_p, H_p\cap\pu^{-1}(S_b),H_p\cap\pu^{-1}(S_w))$. 
  (b) A rough model of the action of $\mu_p$
  on $(\HHH^3,\{\bSigma_j\}_j)$.
  This figure is not precise. In fact,
  $\Delta_j^-\cap \Delta_{j+1}^+=\partial\Delta_j^-\cap \partial\Delta_{j+1}^+$
  is generically strictly bigger than $\{p\}$
  (cf. Remark~\ref{rem:intersection-limit-circle}).
  }
  \label{fig:Butterfly}
\end{figure}
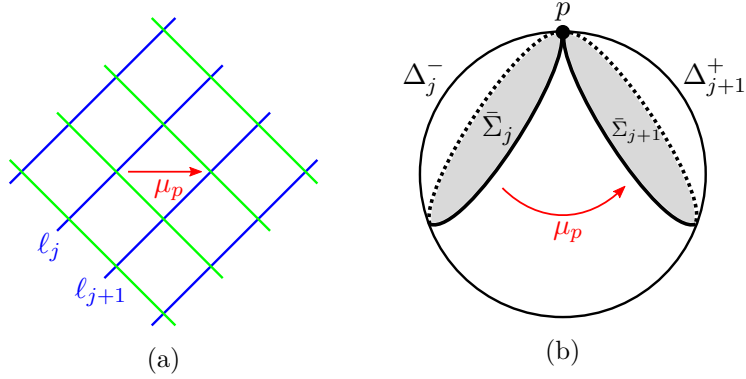

\begin{definition}
\label{def:butterfly}
{\rm
Under the above setting, 
a {\it butterfly $\BF(p)$ at $p\in \PFix(G)$} 
is a pair of disks $\{\Delta_j^-, \Delta_{j+1}^+\}=\{\Delta_j^-(p), \Delta_{j+1}^+(p)\}$ in $\CCC$
for some $j\in\ZZ$.
The {\it color} of the butterfly is defined to be black or white
according to the color of the checkerboard surface $S$.
The underlying space $|\BF(p)|$ of $\BF(p)$
is defined by $|\BF(p)|:=\Delta_j^-\cup \Delta_{j+1}^+\subset \CCC$
(see Figure~\ref{fig:Butterfly_b}).

It should be noted that a butterfly $\BF(p)$ is determined by
the parabolic fixed point $p$, the color
(equivalently, the choice of a checkerboard surface $S$), and the choice of 
a component $\Sigma_j$ of $\pu^{-1}(S)$ such that $p\in \bSigma_j$.
}
\end{definition}

\begin{lemma}
\label{lem:butterfly}
For a butterfly $\BF(p)=\{\Delta_j^-, \Delta_{j+1}^+\}$ at $p\in \PFix(G)$,
the following hold.
\begin{enumerate}
\item[{\rm (1)}]
$\Delta_j^-$ and $\Delta_{j+1}^+$ are disks in $\CCC$ which have disjoint interiors.
\item[{\rm (2)}]
$\mu_p(\CCC\setminus \interior \Delta_j^-)=\Delta_{j+1}^+$.
\end{enumerate}
\end{lemma}

\begin{proof}
(1) 
By Lemma~\ref{lem:balls}(2-a, b),
$\interior \Delta_j^-\subset
\interior \Delta_{j+1}^- = \CCC\setminus \Delta_{j+1}^+$.
Hence 
\[\interior \Delta_j^- \cap \interior \Delta_{j+1}^+
\subset (\interior \Delta_j^-) \cap \Delta_{j+1}^+
=\emptyset.\]

(2) 
By Lemma~\ref{lem:balls}(2-c),
$\mu_p(\CCC\setminus \interior \Delta_j^-)
=\mu_p(\Delta_j^+)
=\Delta_{j+1}^+$.
\end{proof}

\begin{remark}
\label{rem:intersection-limit-circle}
{\rm
The parabolic fixed point $p$ is contained 
the intersection 
$\Delta_j^- \cap \Delta_{j+1}^+=
\partial\Delta_j^-\cap \partial\Delta_{j+1}^+$.
However, in general, the intersection is 
strictly bigger than the singleton $\{p\}$;
it is generically a Cantor set
(cf.~\cite[Theorem 3.13]{Matsuzaki-Taniguchi}).
We hope to give a more detailed description of the intersection
in a subsequent paper.
}
\end{remark}

Now, let $\{\mu_1,\mu_2\}$ be a non-commuting meridian pair,
and set $p_i:=\Fix(\mu_i)\in \PFix(G)$ ($i=1,2$).
Note that $p_1\ne p_2$ and $\mu_i=\mu_{p_i}$  ($i=1,2$)
by Lemmas~\ref{lem:meridian-parabolic} and~\ref{lem:non-commutative}.
Then the following lemma follows immediately from Lemma~\ref{lem:meridian-parabolic}. 

\begin{lemma}
\label{lem:meridian-parabolic2}
{\rm (1)} The correspondence $\{\mu_1,\mu_2\} \mapsto \{p_1, p_2\}$
gives a bijective correspondence
from the set of the non-commuting meridian pairs of $L$
up to equivalence 
to the set of the unordered pairs of distinct points in $\PFix(G)$
up to the action of $G$.

{\rm (2)} Let $\{\mu_1,\mu_2\}$ and $\{p_1, p_2\}$ be as in the above, and
let $\gamma$ be a proper path in $M$
that represents the pair
$\{\mu_1,\mu_2\}$.
Then $\gamma$ lifts to a proper path $\tilde \gamma$ in the universal cover 
$\tilde M \subset \tilde X=\HH^3$
that joins the horospheres $\Bplane_{p_1}$ and $\Bplane_{p_2}$.
Conversely, if $\gamma$ is a proper path in $M$ 
which is the image of a proper path $\tilde\gamma$ in $\tilde M$ 
joining $\Bplane_{p_1}$ and $\Bplane_{p_2}$,
then $\gamma$ represents the pair $\{\mu_1,\mu_2\}$.
\end{lemma}

\begin{notation}
\label{notation:butterfly}
{\rm
Under the above setting,
when we consider two butterflies $\BF(p_1)$ and $\BF(p_2)$ simultaneously,
we denote the butterfly $\BF(p_i)$ by $\{\Delta_i^-,\Delta_i^+\}$ 
for $i=1,2$,
where $\Delta_i^-$ and $\Delta_i^+$ correspond to
$\Delta_j^-(p_i)$ and $\Delta_{j+1}^+(p_i)$, respectively, %for some $j\in\ZZ$
in Definition~\ref{def:butterfly}.
Thus $\mu_i(\CCC\setminus \interior \Delta_i^-)=\Delta_i^+$ ($i=1,2$).
In this sense, we regard $\BF(p_i)$ as the ordered pair
$(\Delta_i^-,\Delta_i^+)$ of closed disks in $\CCC$,
though we continue to denote it by $\{\Delta_i^-,\Delta_i^+\}$.
}
\end{notation}

The proof of Theorem~\ref{Theorem1}
is based on the following proposition.

\begin{proposition}
\label{prop:ping-pong}
Let $L\subset S^3$ be a hyperbolic alternating link,
$\{\mu_1,\mu_2\}$ a non-commuting meridian pair in the link group $G(L)$,
and $\{p_1,p_2\}$ the corresponding pair
of parabolic fixed points. 
Then the subgroup $\Gamma=\langle \mu_1,\mu_2\rangle$
generated by $\{\mu_1,\mu_2\}$ is a rank $2$ free Kleinian group which is geometrically finite,
provided that there are butterflies 
$\BF(p_i)=\{\Delta_i^-,\Delta_i^+\}$ at $p_i$ ($i=1,2$) satisfying the following conditions.
\begin{enumerate}
\item[{\rm (i)}]
The underlying spaces $|\BF(p_1)|$ and $|\BF(p_2)|$ have disjoint interiors in $\CCC$, equivalently,
the four disks 
$\Delta_1^-$, $\Delta_1^+$, $\Delta_2^-$ and $\Delta_2^+$
have disjoint interiors.
\item[{\rm (ii)}]
The complementary open set
$O:=\CCC\setminus (|\BF(p_1)|\cup |\BF(p_2)|)$
is non-empty.
\end{enumerate}
\end{proposition}

\begin{proof}
By a standard ping-pong argument
(see~\cite[Chapter 4]{MSW} for a nice exposition with beautiful illustrations),
we have $w(O)\cap O=\emptyset$
for any non-trivial reduced word $w$
in $\{\mu_1,\mu_2\}$.
Hence, the subgroup $\Gamma=\langle \mu_1,\mu_2\rangle$ of $G(L)$
is a rank $2$ free group
and it has a non-empty domain of discontinuity.
Since a two-parabolic-generator Kleinian group
which has a non-empty domain of discontinuity
is geometrically finite by 
Maskit-Swarup~\cite[Theorem 1]{Maskit-Swarup},
$\Gamma$ is geometrically finite.
\end{proof}

\begin{remark}
{\rm
Though~\cite{Agol2001} appeals to the Klein-Maskit combination theorem,
we could not verify that the conditions in~\cite[Theorem C.2]{Maskit}
is satisfied in the setting of Proposition~\ref{prop:ping-pong}.
This is the reason why we use 
the result of Maskit and Swarup~\cite{Maskit-Swarup}.
We thank Yohei Komori and Hideki Miyachi
for suggesting this idea to us.
}
\end{remark}

\section{Basic facts concerning non-positively curved spaces}
\label{sec:cat0}

In this section, we recall fundamental facts 
concerning non-positively curved spaces, basically
 following Bridson-Haefliger~\cite{BH}.

Let $X=(X,d)$ be a metric space.
In this paper, we mean by
a {\it geodesic} in $X$ an isometric embedding $g:J\to X$
where $J$ is a connected subset of $\RR$.
If $J$ is the whole $\RR$ (resp.~a closed interval),
$g$ is called a {\it geodesic line} (resp.~a {\it geodesic segment}).
We do not distinguish between a geodesic and its image.
$X$ is said to be a {\it geodesic space}
if every pair of points can be joined by a geodesic in $X$.
For points $a$ and $b$ in a geodesic space $X$,
we denote by $[a,b]$ a geodesic segment joining $a$ and $b$.
The symbols $(a,b)$, $[a,b)$ and $(a,b]$
represent open or half-open geodesic segments, respectively.  
Then the distance $d(a,b)$ is equal to the length, $\length([a,b])$,
of the geodesic segment $[a,b]$. 
(See~\cite[Definition I.1.18]{BH} for the definition of the length of a curve.)

A geodesic space $X$ is called a {\it CAT(0) space}
if any geodesic triangle is thinner than a comparison triangle in 
the Euclidean plane $\EE^2$,
that is, the distance between any points on a geodesic triangle is
less than or equal to the corresponding points on a comparison triangle.
A CAT(0) space is 
{\it uniquely geodesic}, i.e., 
for every pair of points, there is a unique geodesic joining them (\cite[Proposition~II.1.4(1)]{BH}).
A geodesic space $X$ is said to be {\it non-positively curved}
if it is locally a CAT(0) space 
(cf.~\cite[Definitions II.1.1 and II.1.2]{BH}).
The following (special case of) the Cartan-Hadamard theorem 
is fundamental.

\begin{proposition}
\label{prop:CH0}
{\rm \cite[Special case of Theorem II.4.1(2)]{BH}}
Let $X$ be a complete, connected, metric space.
If $X$ is non-positively curved, then
the universal covering $\tilde X$ (with the induced length metric)
is a CAT(0) space.
\end{proposition}

A subset $W$ of a uniquely geodesic space $X$ is said to be {\it convex}
if, for any distinct points $a$ and $b$ in $W$,
the geodesic segment $[a,b]$ is contained in $W$.
For a closed convex set $W$ of a complete CAT(0) space $X$,
let $\pi_W: X\to W$ be the \textit{projection},
namely 
$\pi_W(x)$ for every $x\in X$ is the unique point in $W$ 
such that 
$d(x,\pi_W(x))= d(x,W):=\inf \{d(x,y) \mid y \in W\}$ 
(see~\cite[Proposition II.2.4]{BH}).
For points $x\in X\setminus W$ and $w\in W$
define
\[
\angle_w(x,W):=
\inf \left\{ \angle_w(x,y) \mid y \in W \setminus \{w\} \right\},
\]
where $\angle_w(x,y)$ is the Alexandrov angle
$\angle_w([w,x],[w,y])$
between the geodesic segments $[w,x]$ and $[w,y]$ at $w$ 
(see~\cite[Definition I.1.12 and Notation II.3.2]{BH}).

\begin{remark}
\label{rem:localness-angle}
{\rm
The angle $\angle_w(x,W)$ is determined by the local shape of $(X,W)$ around $w$
in the following sense.
For any neighborhood $U$ of $w$,
 and for any $x'\in (w,x]\cap U$, we have
 \[
\angle_w(x,W)=\angle_w(x',W\cap U):=
\inf \left\{ \angle_w(x',y') \mid y' \in (W\cap U) \setminus \{w\} \right\},
\]
because for any $x'\in (w,x]$, $y\in W \setminus \{w\}$ 
and $y'\in (w,y]$,
we have $\angle_w(x,y)=\angle_w(x',y')$.
}
\end{remark}

\begin{lemma}\label{lem:GR2b}
Let $X$ be a complete CAT(0) space and $W$ a closed convex subset of $X$.
Then for any $x\in X$ and $w\in W$, we have $w=\pi_W(x)$
if and only if $\angle_w(x,W)\ge \frac{\pi}{2}$.
\end{lemma}

\begin{proof}
The only if part is nothing other than 
\cite[Proposition II.2.4(3)]{BH}.
To see the if part, suppose that  the inequality
$\angle_w(x,W)\ge \frac{\pi}{2}$ holds, and
assume to the contrary that $w$ is
different from the point $w_0:=\pi_W(x)$.
Let $\Delta(\bar x, \bar w, \bar w_0)\subset \EE^2$
be the comparison triangle of the
geodesic triangle $\Delta(x,w,w_0)\subset X$. 
Then $\angle_{\bar w}(\bar x,\bar w_0)\ge \angle_w(x,w_0)\ge \angle_w(x,W)\ge \frac{\pi}{2}$
by~\cite[Propositions II.1.7(4)]{BH} and the assumption.
We also have 
$\angle_{\bar w_0}(\bar x,\bar w)\ge \angle_{w_0}(x,w)\ge \frac{\pi}{2}$
by~\cite[Propositions II.1.7(4) and II.2.4(3)]{BH}. 
Thus the Euclidean triangle $\Delta(\bar x, \bar w, \bar w_0)$
has two angles $\ge \frac{\pi}{2}$, a contradiction.
\end{proof}

Let $W_1$ and $W_2$ be closed convex subsets of 
a complete CAT(0) space $X$. 
The distance $d(W_1, W_2)$ between $W_1$ and $W_2$ is defined by
 \[ d(W_1, W_2) := \inf \{d(x_1,x_2) \mid x_1 \in W_1, x_2 \in W_2\}.\]
For a pair of distinct points $(x_1,x_2)\in W_1\times W_2$,
the geodesic segment $[x_1,x_2]$ is a {\it shortest path} 
 between $W_1$ and $W_2$ if $d(W_1,W_2)=\length([x_1,x_2])$.
The geodesic segment $[x_1,x_2]$ is a 
 \textit{common perpendicular} to $W_1$ and $W_2$
 if $\angle_{x_1}(x_2, W_1)\ge \frac{\pi}{2}$
 and $\angle_{x_2}(x_1, W_2)\ge \frac{\pi}{2}$.

\begin{lemma}\label{lem:GR2} 
  Let $X$ be a complete CAT(0) space, and let $W_1$ and $W_2$ be 
  closed convex subsets of $X$.
  Then, for a pair of distinct points $(x_1,x_2)\in W_1\times W_2$,
  the geodesic segment $[x_1,x_2]$ is a shortest path 
  between $W_1$ and $W_2$
  if and only if it is a common perpendicular to $W_1$ and $W_2$.
  In particular, if a common perpendicular to $W_1$ and $W_2$ exists,
  then $d(W_1,W_2)>0$ and so $W_1\cap W_2=\emptyset$.
\end{lemma} 

\begin{proof}
Assume that $[x_1,x_2]$ is a common perpendicular to $W_1$ and $W_2$.
Consider the projection 
$\pi:=\pi_{[x_1,x_2]}$ from $X$ to the closed convex set $[x_1,x_2]$.
Then for any pair of points $(y_1,y_2)\in W_1\times W_2$,
we see $x_1=\pi(y_1)$ by the if part of Lemma~\ref{lem:GR2b},
because $\angle_{x_1}(y_1,[x_1,x_2])=\angle_{x_1}(x_2,y_1)\ge \angle_{x_1}(x_2,W_1)\ge \frac{\pi}{2}$.
Similarly $x_2=\pi(y_2)$.
Since the projection does not increase distances 
by~\cite[Proposition II.2.4(4)]{BH},
we have 
\[
d(x_1,x_2)=d(\pi(y_1),\pi(y_2))\le d(y_1,y_2).
\]
Hence $\length([x_1,x_2])=d(x_1,x_2)= d(W_1,W_2)$.
This completes the proof of the if part.
The only if part immediately 
follows from (the if part of) Lemma~\ref{lem:GR2b}.
\end{proof}

A {\it cubed complex} is  
a metric space $X=(X,d)$ obtained from a disjoint union of unit cubes
$\hat X=\bigsqcup_{\lambda\in\Lambda} (I^{n_\lambda}\times\{\lambda\})$
by gluing their faces through isometries.
To be precise, it is an $M_{\kappa}$-polyhedral complex with $\kappa=0$
in the sense of \cite[Definition I.7.37]{BH}
that is made up of Euclidean unit cubes, i.e.,
the set $\mathrm{Shapes}(X)$ in the definition 
consists of Euclidean unit cubes.
(See \cite[Example (I.7.40)(4)]{BH}.)
The metric $d$ on $X$ is the length metric induced from the Euclidean metrics of the unit cubes
(see \cite[I.7.38]{BH} for a precise definition).
We recall the following basic fact
(cf.~\cite[Theorem in p.97 or I.7.33]{BH}).

\begin{proposition}
\label{prop:complete-metric}
Every finite dimensional cubed complex 
is a complete geodesic space.
\end{proposition}

When we need to consider the combinatorial structure of the cubed complex $X$
in addition to its metric,
we denote it by using the corresponding calligraphic letter 
$\mathcal{X}$ and regard the metric space $X$
as the underlying space $|\mathcal{X}|$ of $\mathcal{X}$.
Otherwise, we do not distinguish symbolically 
among $X$, $\mathcal{X}$ and $|\mathcal{X}|$,
and use a symbol which we think fit to the setting.
We also call $\mathcal{X}$ 
a {\it cubed decomposition}
of the metric space $X$.

For a point $x\in X=|\mathcal{X}|$,
two non-trivial geodesics issuing from $x$
are said to define the {\it same direction}
if the Alexandrov angle between them is zero.
This determines an equivalence relation on the set of 
non-trivial geodesics issuing from $x$,
and the Alexandrov angle induces a metric 
on the set of the equivalence classes.
The resulting metric space is called the
{\it space of directions} at $x$ and denoted 
$S_x(X)$ (see ~\cite[Definition II.3.18]{BH}).

Suppose $x$ is a vertex $v$ of the cubed complex $\XCubing$.
Then the space $S_v(\XCubing)$ is obtained by gluing the spaces
%$\{S_{v_{\lambda}}(I^{n_{\lambda}}\times\{\lambda\})\}_{v_{\lambda}\in p^{-1}(v)}$.
$\{S_{v_{\lambda}}(I^{n_{\lambda}}\times\{\lambda\})\}$,
where $\lambda$ runs over the elements of the suffix set $\Lambda$
such that $(v_{\lambda},\lambda)\in I^{n_\lambda}\times\{\lambda\}
\subset \hat X$
is mapped to $v$ by the projection $\hat X \to X$. 
Here $S_{v_{\lambda}}(I^{n_{\lambda}}\times\{\lambda\})$ is the space   
of directions in the cube $I^{n_{\lambda}}\times\{\lambda\}$ at the vertex $v_{\lambda}$;
so it is an {\it all-right spherical simplex},
a geodesic simplex in the unit sphere $S^{n_{\lambda}-1}$  
all of whose edges have length $\pi/2$.
Hence $S_v(\XCubing)$ has a structure of a finite dimensional {\it all-right spherical complex}, 
namely an $M_{\kappa}$-polyhedral complex with $\kappa=1$
in the sense of~\cite[Definition I.7.37]{BH}
which is made up of all-right spherical simplices.
This complex is called the {\it geometric link} of $v$ in $\XCubing$,
and is denoted by $\Lk(v,\XCubing)$ (see~\cite[(I.7.38)]{BH}).
It is endowed with the length metric $d_{\Lk(v,\XCubing)}$
induced from the spherical metrics of the all-right spherical simplices.
Then the following holds
(cf.~\cite[the second sentence in p.191]{BH}).

\begin{lemma}
\label{lem:distances-link}
The metric $d_{S_v(\XCubing)}$ on $S_v(\XCubing)=\Lk(v,\XCubing)$
determined by the Alexandrov angle
is equal to
the metric $d_{\Lk(v,\XCubing)}^{\pi}$ defined by
\[
d_{\Lk(v,\XCubing)}^{\pi}(g_1,g_2):=\min\{d_{\Lk(v,\XCubing)}(g_1,g_2),\pi\}.
\]
\end{lemma}

\begin{proof}
By~\cite[Theorem I.7.39]{BH},
there is a natural isometry $f$
from the open ball $B_X(v,\epsilon)=\{x\in X \ | \ d(v,x)<\epsilon\}$,
for some $\epsilon>0$, 
onto the open ball of the same radius $\epsilon$ about the cone point of
the Euclidean cone  $C_0(\Lk(v,\XCubing))$ over the metric space 
$\Lk(v,\XCubing)$.
(See~\cite[Definition I.5.6]{BH} for the definition of the Euclidean cone
($\kappa$-cone with $\kappa=0$) and its cone point.)
The metric $d_{\Lk(v,\XCubing)}^{\pi}$ is recovered from the metric 
of the open ball about the cone point of
the Euclidean cone  $C_0(\Lk(v,\XCubing))$ (see~\cite[Remark I.5.7]{BH}),
whereas the metric $d_{S_v(\XCubing)}$ is determined by 
the metric on $B_X(v,\epsilon)$.
Hence, by the naturality of the isometry $f$,
we obtain the desired result.
\end{proof}

We recall the following well-known 
criterion for a cubed complex 
to be non-positively curved \cite[Theorem II.5.20]{BH},
where a {\it flag complex} is a {\it simplicial} complex
in which every finite set of vertices that is pairwise joined by an edge
spans a simplex.

\begin{proposition}[Gromov's flag criterion]
\label{prop:Gromov-LC}
A finite dimensional cubed complex $\XCubing$ is non-positively curved
if and only if the geometric link of each vertex is a flag complex.
\end{proposition}

\section{Non-positively curved cubed decompositions of alternating link exteriors}
\label{sec:NPC}

The proof of Proposition \ref{prop:disjoint-half-space}, as well as 
that of Proposition~\ref{prop:qf} given by 
\cite{Adams07, Agol2001, Aitchson-Rubinstein_90a},
is based on non-positively curved cubed decompositions of 
prime alternating link exteriors 
in which the checkerboard surfaces are {\it hyperplanes}, 
i.e., consist of midsquares of the cubes.
Here a {\it midsquare} of a cube $I^3$ is 
a square properly embedded in $I^3$
which is parallel to a face of $\partial I^3$
and passes through the center $(\frac{1}{2},\frac{1}{2},\frac{1}{2})$.
(See \cite[Definition 2.2]{Haglund-Wise}
for a precise definition of a hyperplane.)
%\cite[Section 2.4]{Sageev} or \cite[Definition 4.5]{Farley}
%for a precise definition
%of a hyperplane.)  
In this section, we quickly describe the cubed decompositions
following the construction by D. Thurston~\cite{DThurston} 
and the detailed description by Yokota~\cite{Yokota2011}
(cf.~\cite{ALR, Sakuma-Yokota, Sakai}).

For each crossing of a prime alternating diagram $\Diagram$
of a prime alternating link $L$, consider an octahedron 
that contains the corresponding crossing arc as a vertical central axis
(see Figure~\ref{fig:Aitchison}).
Truncating each octahedron at its top and bottom vertices and splitting along 
the horizontal square containing the remaining four vertices,
we obtain a pair of cubes in the link exterior $M$,
each of which intersects $\partial M$ along the top or bottom face
and intersects the checkerboard surfaces in the vertical midsquares
(see Figure~\ref{fig:Aitchison2}).
We can expand the cubes in $M$ so as to obtain the desired cubed decomposition of $M$
(see Figure~\ref{fig:Aitchison} and its caption).
Thus we obtain the following proposition.

\begin{figure}
  \centering
  \begin{tabular}{c}
    \begin{subfigure}{0.45\columnwidth}
      \begin{overpic}[width=\columnwidth]{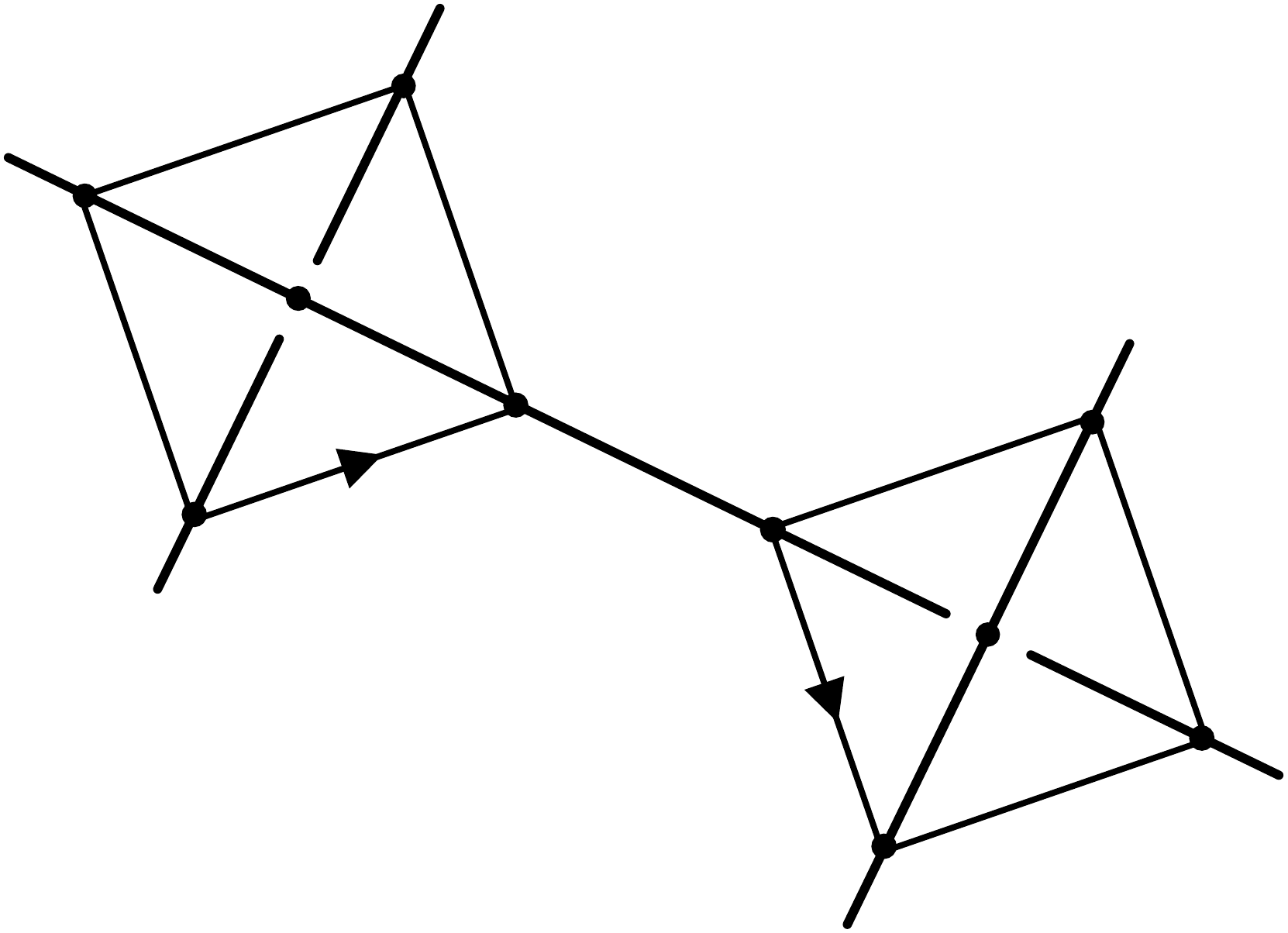}
        \put(41,76){$c$}
        \put(135.5,27){$c'$}
        \put(70,33){$R$}
%        \put(-15,110){$(1)$}
      \end{overpic}
      \vspace{-2mm}
      %\subcaption{}
    \end{subfigure}\\
    \begin{subfigure}{0.7\columnwidth}
      \begin{overpic}[width=\columnwidth]{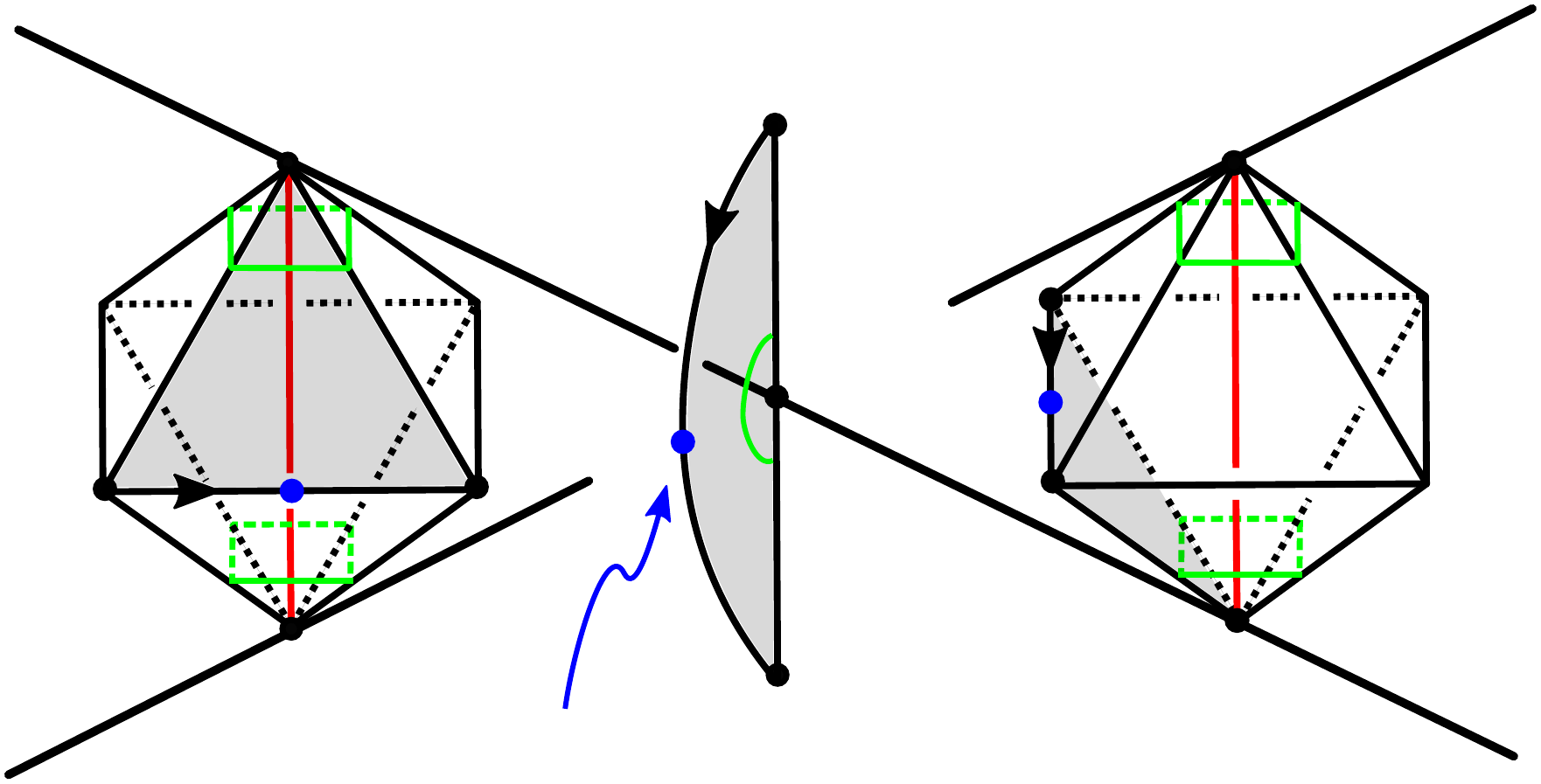}
        \put(4,55){$v_+$}
        \put(90,55){$v_-$}
        \put(91,4){\textcolor{blue}{$m(R)$}}
        \put(136,125){$v_+$}
        \put(136,10){$v_-$}
        \put(174,81){$v_+$}
        \put(174,56){$v_-$}
%        \put(-5,130){$(2)$}
      \end{overpic}
      %\subcaption{}
    \end{subfigure}
  \end{tabular}
  \caption{Local picture of the cubed complex $\Cubing$. The partially truncated octahedra in $M$ are expanded so that they cover the whole $M$.
  The shaded faces of the octahedra at the crossings $c$ and $c^\prime$ are identified with the central bow-shaped face.
In particular, the pair of the horizontal arrowed edges of the octahedra are identified with the arrowed edge of the bow-shaped face joining
monotonically  the top vertex $v_+$ and the bottom vertex $v_-$, passing through the center $m(R)$ of the region $R$.}
  \label{fig:Aitchison}
\end{figure}

\begin{figure}
  \centering
  \begin{subfigure}{0.24\columnwidth}
    \centering
    \begin{overpic}[width=\columnwidth]{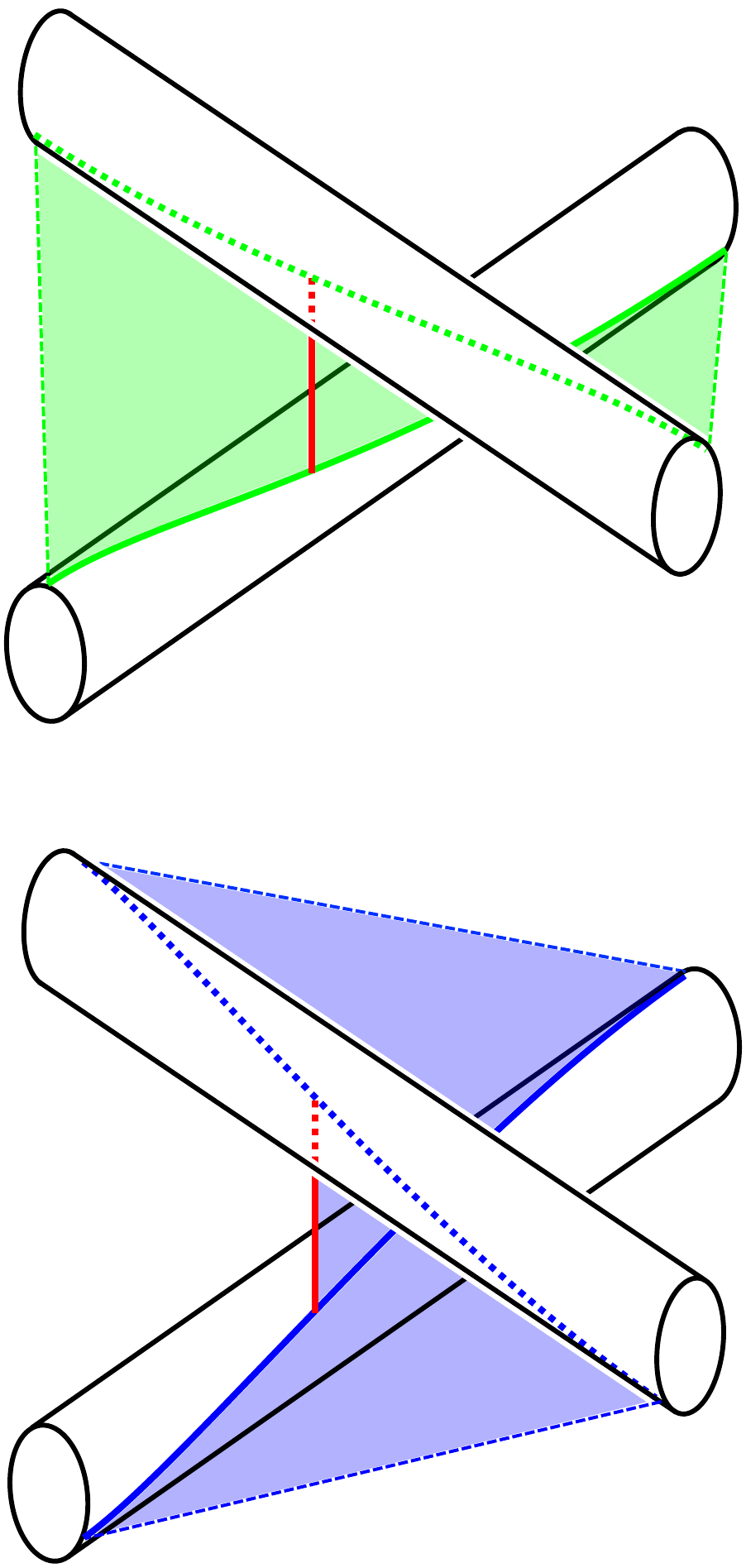}
      \put(94,151){\textcolor{Green}{$S_w$}}
      \put(10,151){\textcolor{Green}{$R_2$}}
      \put(47,3){\textcolor{blue}{$R_1$}}
      \put(47,87){\textcolor{blue}{$S_b$}}
    \end{overpic}
    \subcaption{}
    \label{fig:Aitchison2_c}
  \end{subfigure}
  \hspace{8mm}
  \begin{subfigure}{0.35\columnwidth}
    \centering
    \begin{overpic}[width=\columnwidth]{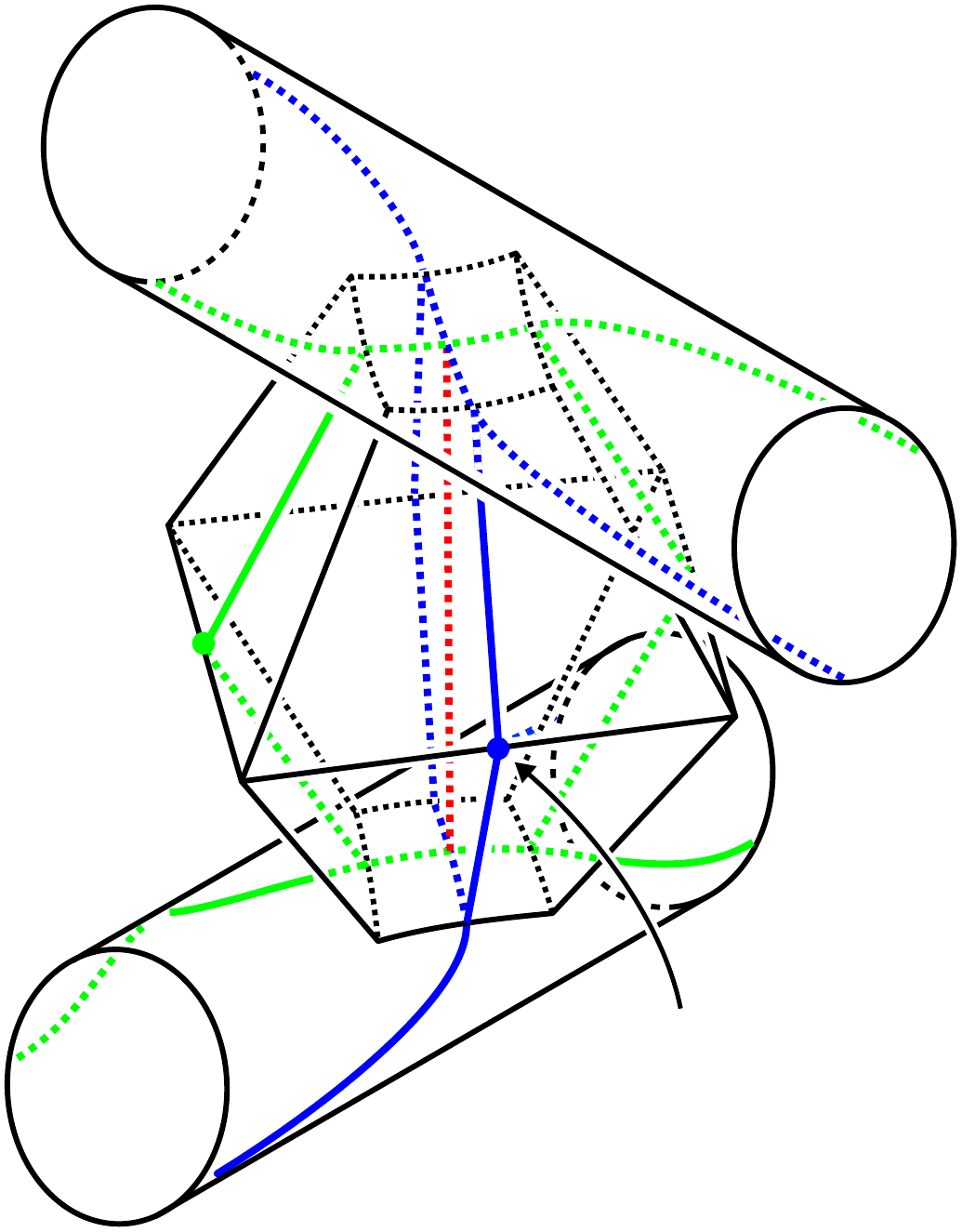}
      \put(10,100){$v_-$}
      \put(-3,78){\textcolor{Green}{$m(R_2)$}}
      \put(21,64){$v_+$}
      \put(113,71){$v_-$}
      \put(84,21){\textcolor{blue}{$m(R_1)$}}
      \put(-23,23){\textcolor{Green}{$\partial S_w$}}
      \put(28,-3){\textcolor{blue}{$\partial S_b$}}
    \end{overpic}
  \vspace{2mm}
  \subcaption{}
  \label{fig:Aitchison2_b}
  \end{subfigure}
  \hspace{5mm}
  \begin{subfigure}{0.15\columnwidth}
    \centering
    \begin{overpic}[width=\columnwidth]{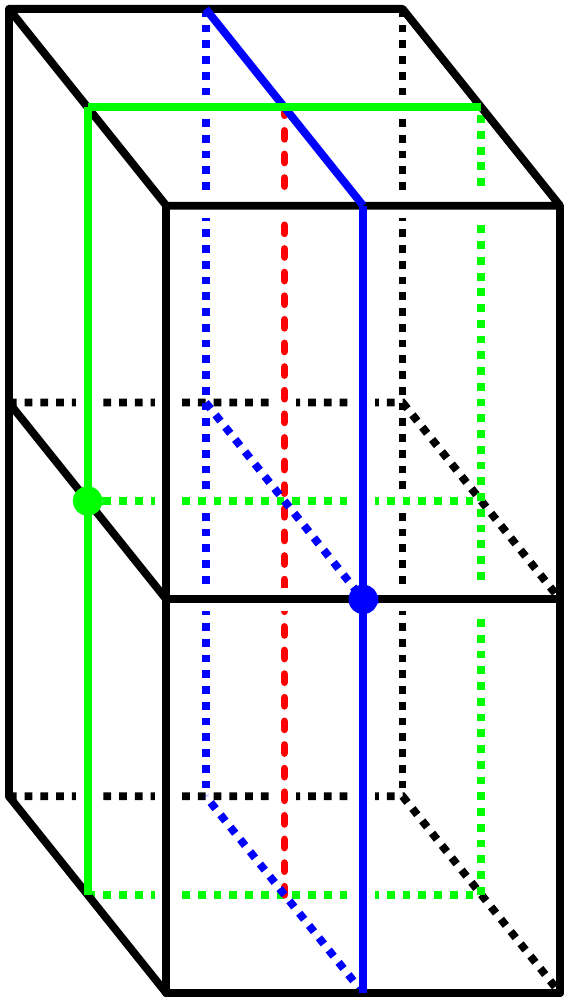}
      \put(32,-10){\textcolor{blue}{$S_b$}}
      \put(-5,1){\textcolor{Green}{$S_w$}}
%      \put(-22,38){\large$\cong$}
    \end{overpic}
    \vspace{13mm}
    \subcaption{}
    \label{fig:Aitchison2_a}
  \end{subfigure}
  \caption{
  At each crossing, 
  (a) $S_b$ and $S_w$ intersect transversely along the crossing arc,
 (b) each of $S_b$ and $S_w$ intersects 
 the partially truncated octahedron in a vertical middle plate containing 
 the crossing arc, and 
 (c) each of the middle plate determines a pair of midsquares in the pair of cubes, and the two pairs of midsquares intersect orthogonally
 along the vertical axes of the cubes.
  }
  \label{fig:Aitchison2}
\end{figure}

\begin{proposition}
\label{prop:cubing}
Let $L$ be a prime alternating link and $\Diagram$ 
a prime alternating diagram of $L$.
Then there is a complete, non-positively curved, cubed complex $\Cubing$ 
whose underlying space is the exterior $M$ of $L$,
which satisfies the following conditions.
\begin{enumerate}
\item[\rm(1)]
Each cube $I^3$ intersects $\partial M$ in the top face $I^2\times\{1\}$
or the bottom face $I^2\times\{0\}$.
\item[\rm(2)]
There are hyperplanes $\Sb$ and $\Sw$ in $\Cubing$
that represent the isotopy classes of 
the black and white surfaces, respectively,
and satisfy the following conditions.
\begin{enumerate}
\item[\rm(a)]
Each of $\Sb$ and $\Sw$ intersects each cube
in one of the two vertical midsquares
$\{\frac{1}{2}\}\times I^2$ and $I\times \{\frac{1}{2}\}\times I$.
\item[\rm(b)]
$\Sb$ and $\Sw$ intersects 
\lq\lq orthogonally'' along $\crossing:=\Sb\cap \Sw$,
the disjoint union of geodesic segments representing crossing arcs.
\end{enumerate}
\item[\rm(3)]
$\Cubing$ has precisely two inner vertices $v_+$ and $v_-$.
\item[\rm(4)]
There is a bijective correspondence between the inner edges 
(edges contained in $\interior\Cubing$) 
of $\Cubing$
and the regions of $\Diagram$:
the inner edge $e(R)$ corresponding to a region $R$
is a monotone path joining $v_+$ with $v_-$
that intersects $\Sb\cup\Sw$ \lq\lq orthogonally'' 
at a single point $m(R)$ which is
contained in 
the component of
$(\Sb\cup\Sw)\setminus\crossing$ corresponding to $R$.
We call $m(R)$ the center of $R$. 

\item[\rm(5)]
For each $\epsilon\in\{+,-\}$,
the geometric link $\Lk(v_{\epsilon},\Cubing)$ is 
the all-right spherical complex whose combinatorial structure is
obtained from the cell decomposition of $S^2$
determined by the dual graph $\Diagram^*$ of $\Diagram$,
by subdividing each region of $\Diagram^*$ as follows.
Each region of $\Diagram^*$ contains a unique vertex, say $c$, of $D$.
Subdivide the region by taking the join of $c$ and 
the edge cycle of $\Diagram^*$ forming the boundary of the region
(see Figure~\ref{fig:A-Link}).
Here the vertex $m^*(R)$ of $\Diagram^*\subset \Lk(v_{\pm},\Cubing)$ 
dual to the region $R$
corresponds to the direction at $v_{\epsilon}$ 
determined by the geodesic $[v_{\epsilon},m(R)]$.
\end{enumerate}
\end{proposition}

\begin{figure}
  \centering
  \includegraphics[width=0.35\columnwidth]{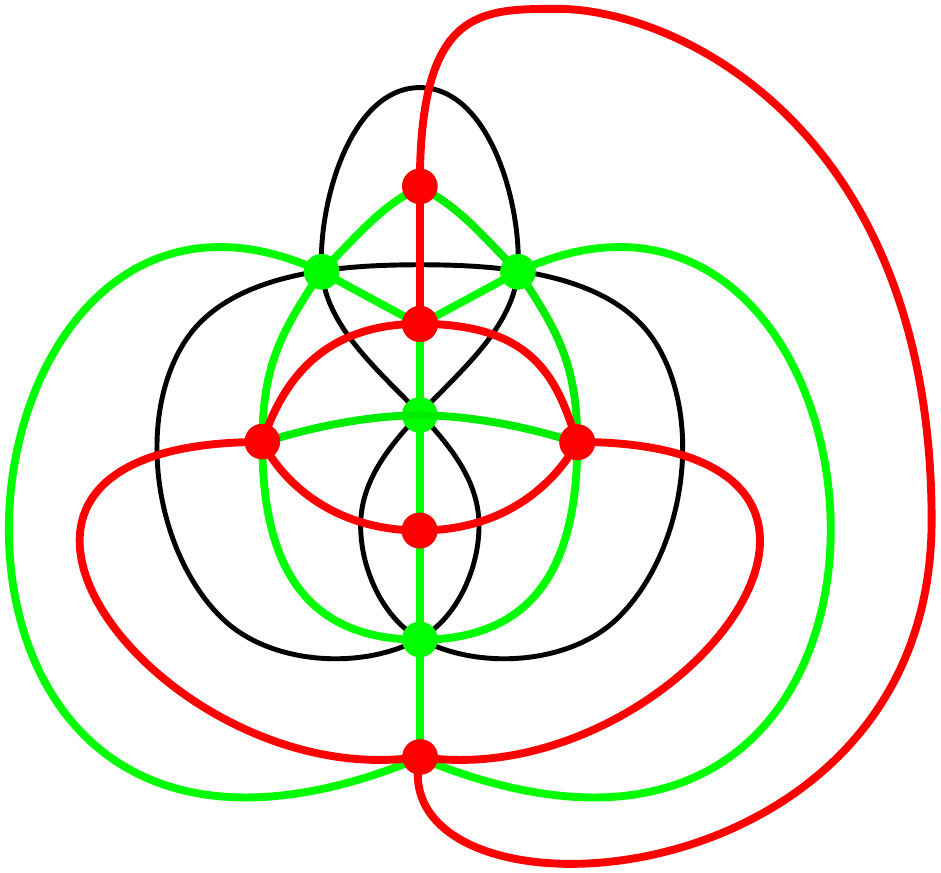}
 \caption{The geometric link $\Lk(v_{\epsilon},\Cubing)$:
 The union of the thick red graph and the thick green graph
 forms 
 the $1$-skeleton of $\Lk(v_{\epsilon},\Cubing)$.}
  \label{fig:A-Link}
\end{figure}

\begin{remark}
\label{remark:cubing}
{\rm
(1) In the statement (2-b), the adjective 
\lq\lq orthogonally'' means that 
every interior point of $\crossing$ 
has a neighborhood $U$ in $\Cubing$
such that the triple $(U, U\cap\Sb, U\cap\Sw)$
is isometric to a neighborhood of the origin
in $(\RR^3, 0\times \RR^2, \RR\times 0\times \RR)$.

(2) In the statement (4), 
the adjective 
\lq\lq orthogonally'' means that
there is a CAT(0) neighborhood $U$ of $m(R)$ in $\Cubing$,
such that for any point $x\in (e(R)\setminus \{m(R)\})\cap U$
and $y\in ((\Sb\cup\Sw)\setminus \{m(R)\})\cap U$,
we have $\angle_{m(R)}(x,y)=\pi/2$.

(3) For each component $T$ of $\partial M$, the restriction $\Cubing|_T$ of $\Cubing$ 
to $T$
gives a cubed decomposition of $T$,
whose $1$-skeleton is the union of two longitudes $\ell_b$ and $\ell_w$,
where $\ell_b$ and $\ell_w$ are parallel to $S_b\cap T$ and $S_w\cap T$, respectively.
In particular, for each square $I^2$ of $\Cubing|_T$,
exactly one of the two diagonals of the square projects to a meridional loop (cf.~Figure~\ref{fig:Butterfly}(a)).

(4) The assertion (5) implies the following key fact.
For distinct regions $R_1$ and $R_2$ of $\Diagram$,
the distance $d_{\Lk(v_{\epsilon},\Cubing)}(m^*(R_1),m^*(R_2))$
is $\pi/2$ or $\ge \pi$
according to whether $R_1$ and $R_2$ are adjacent or not.
By Lemma~\ref{lem:distances-link},
this implies that the Alexandrov angle
$\angle_{v_{\epsilon}}(m(R_1),m(R_2))$ is equal to $\pi/2$ or $\pi$
according to whether $R_1$ and $R_2$ are adjacent or not.
This fact plays a key role in the proof of Proposition~\ref{prop:disjoint-half-space}.
}
\end{remark}

Let $\XCubing$ be the cubed complex obtained 
by attaching the cubed complex $\bigcup_{n=1}^{\infty} \partial\Cubing \times [n,n+1]$
to $\Cubing$ along $\partial\Cubing$.
The the underlying space is the link complement $X$, and 
we have the following key proposition.

\begin{proposition}
\label{prop:nonpoitively-curved-cubing}
The cubed complexes $\Cubing$ and $\XCubing$ are complete and non-positively curved.
\end{proposition}

\begin{proof}
By Proposition~\ref{prop:complete-metric},
$\Cubing$ and $\XCubing$ are complete.
From the description of $\Lk(v_{\pm},\Cubing)$
($=\Lk(v_{\pm},\XCubing)$)
given by Proposition~\ref{prop:cubing}(5),
we can check that they are flag complexes 
as in \cite[Proof of Proposition 3.3]{Sakuma-Yokota}.
For any other vertex $v$ of $\Cubing$ and $\XCubing$, 
we can easily see that 
the geometric link of $v$ in $\Cubing$ and $\XCubing$, respectively,
is a flag complex.
Hence, $\Cubing$ and $\XCubing$ are non-positively curved
by Gromov's flag criterion (Proposition~\ref{prop:Gromov-LC}). 
\end{proof}

%By attaching the cubed complex $\bigcup_{n=1}^{\infty} \partial\Cubing \times [n,n+1]$
%to $\Cubing$ along $\partial\Cubing$,
%we obtain a complete, non-positively curved, cubed complex $\XCubing$
%whose underlying space is the link complement $X$.
%(The completeness follows from 
%Proposition~\ref{prop:complete-metric}, and 
%the non-positively curvedness follows from
%Proposition~\ref{prop:Gromov-LC}(Gromov's flag criterion) and Proposition~\ref{prop:cubing}(5).) 
%As in Proposition~\ref{prop:cubing}(2),
%the open checkerboard surfaces  $S_b$ and $S_w$ in $X$
%are isotopic to hyperplanes in $\XCubing$,
%which we also denote by $\Sb$ and $\Sw$, respectively.
%Then $\Sb$ and $\Sw$ intersects 
%orthogonally along $\crossing:=\Sb\cap \Sw$,
%the disjoint union of geodesic lines representing 
%open crossing arcs.

Let $\tXCubing$ (resp.~$\tCubing$) be the cubed decomposition
of the universal covering space $\tilde X$ (resp.~$\tilde M$)
obtained by pulling back the cubed decompositions $\XCubing$ of $X$
(resp.~$\Cubing$ of $M$)
through the covering projection $\pu:\tilde X\to X$.
Then $\tXCubing$ and $\tCubing$ are complete CAT(0) cubed complexes
by Proposition~\ref{prop:nonpoitively-curved-cubing}
and the Cartan-Hadanard theorem (Proposition~\ref{prop:CH0}).

As in Proposition~\ref{prop:cubing}(2),
the open checkerboard surfaces  $S_b$ and $S_w$ in $X$
are isotopic to hyperplanes in $\XCubing$,
which we also denote by $\Sb$ and $\Sw$, respectively.
Then $\Sb$ and $\Sw$ intersects 
orthogonally along $\crossing:=\Sb\cap \Sw$,
the disjoint union of geodesic lines representing 
open crossing arcs.
Set $\tSb:=\pu^{-1}(\Sb)$ and $\tSw:=\pu^{-1}(\Sw)$.
Then every component $\Sigma$ of $\tSb$ (resp.~$\tSw$) is a hyperplane in $\tXCubing$,
and it is regarded as the universal covering of $\Sb$ (resp.~$\Sw$):
we call $\Sigma$ a {\it checkerboard hyperplane} in $\tXCubing$.
Of course, a checkerboard hyperplane is a checkerboard plane
defined in Section~\ref{sec:checkerboard-surface}.
%Note also that  
%$\tSb$ and $\tSw$ (and so all checkerboard hyperplanes) 
%are regarded as subcomplexes
%of the cubical subdivision $\tXCubing'$ of $\tXCubing$,
%the cubed complex
%obtained from $\tXCubing$ by subdividing each cube $I^3$ into $8$ cubes
%by cutting along the three midsquares.

\begin{proposition}
\label{prop:cubing2X}
Every checkerboard hyperplane
$\Sigma$ is convex in the CAT(0) space $\tXCubing$.
Moreover, $\Sigma$ divides $\tXCubing$ into two closed convex subspaces,
namely, there are convex subspaces 
$\Half$ and $\Half^{\mathfrak{c}}$ of $\tXCubing$
such that $\tXCubing=\Half\cup\Half^{\mathfrak{c}}$
and $\Sigma=\Half\cap\Half^{\mathfrak{c}}$.
\end{proposition} 

\begin{proof}
This follows from
Farley's result \cite[Theorem 4.4]{Farley},
motivated by Sageev's
combinatorial study of hyperplanes in \cite{Sageev}.
See \cite[Section 4.3]{SS} for another proof.
\end{proof}

We call each of the subspaces
$\Half$ and $\Half^{\mathfrak{c}}$ of $\tXCubing$ in the above proposition
a {\it checkerboard hyper-half-space} bounded by 
the checkerboard hyperplane $\Sigma$. 
Of course, every checkerboard hyper-half-space is a
checkerboard half-space defined in Section~\ref{sec:checkerboard-surface}.

By a {\it peripheral plane}, we mean a component of 
$\partial\tCubing \subset \tXCubing$.
Then we have the following proposition, 
which is easily proved   
by using \cite[Theorem 1.1]{SS}.

\begin{proposition}
\label{prop:cubing2XE}
Under the above setting, every peripheral plane $\Bplane\subset \partial\tCubing$
is convex in the CAT(0) space $\tXCubing$.
\end{proposition} 

Propositions~\ref{prop:cubing2X} and~\ref{prop:cubing2XE} imply the following proposition.

\begin{proposition}
\label{pro:connected-intersection}
{\rm (1)} Let $\Sigma_1$ and $\Sigma_2$ be distinct checkerboard hyperplanes
in $\tXCubing$.
Then one of the following holds.
\begin{enumerate}
\item[{\rm (a)}] 
$\Sigma_1\cap\Sigma_2=\emptyset$.
\item[{\rm (b)}]
$\Sigma_1\cap\Sigma_2$ is a geodesic line. 
Moreover, $\Sigma_1$ and $\Sigma_2$ intersect orthogonally along $\Sigma_1\cap\Sigma_2$,
in the sense defined in Remark~\ref{remark:cubing}(1).
Furthermore, $\Sigma_1\cap\Sigma_2$ divides each of $\Sigma_1$ and $\Sigma_2$
into two convex subspaces.
\end{enumerate}

{\rm (2)} Let $\Sigma$ be a checkerboard hyperplane
and $\Bplane$ a peripheral plane in $\tXCubing$.
Then one of the following holds.
\begin{enumerate}
\item[{\rm (a)}] 
$\Sigma\cap\Bplane=\emptyset$.
\item[{\rm (b)}]
$\Sigma\cap\Bplane$ is a geodesic line. 
Moreover, $\Sigma$ and $\Bplane$ intersect orthogonally along $\Sigma\cap\Bplane$,
in the sense defined in Remark~\ref{remark:cubing}(1).
Furthermore, $\Sigma\cap\Bplane$ divides each of 
$\Sigma$ and $\Bplane$
into two convex subspaces.
\end{enumerate}
\end{proposition}

\begin{proof}
The assertions except for the orthogonalities are consequences of
Propositions~\ref{prop:cubing2X},~\ref{prop:cubing2XE}
and the fact that the intersection of two convex sets is again convex.
The orthogonality of $\Sigma_1$ and $\Sigma_2$ in (1-b)
follows from the fact that 
$\Sigma_1$ and $\Sigma_2$ are hyperplanes,
and so their relative positions are as explained in
Proposition~\ref{prop:cubing}(2) and 
illustrated in
Figure~\ref{fig:Aitchison2}(c).
The orthogonality of $\Sigma$ and $H$ in (2-b)
follows similarly from 
Proposition~\ref{prop:cubing}(1,2) and 
Figure~\ref{fig:Aitchison2}(c).
The additional assertions in (1-b) and (2-b)
are proved by a (much simpler) argument 
similar to the proof of Proposition~\ref{prop:cubing2X}.
\end{proof}

The following technical corollary is used in Section~\ref{sec:proof-maintheorem}.

\begin{corollary}
\label{cor:no-branching}
Let $\Sigma_1$ and $\Sigma_2$ be distinct checkerboard hyperplanes
such that $\ell:=\Sigma_1\cap\Sigma_2$ is a geodesic line.
Then the following hold.
\begin{enumerate}
\item[{\rm (1)}]
If $\ell'$ is a geodesic in $\tXCubing$ such that 
$\ell\cap \ell'\ne \emptyset$,
then either $\ell'\subset \ell$ or $\ell\cap\ell'$ is a singleton.
\item[{\rm (2)}]
If $\ell'$ is a geodesic in $\Sigma_i$ ($i=1$ or $2$)
such that 
$\ell\cap \ell'$ is a singleton $\{y\}$ in $\interior \ell'$,
then $y$ is a transversal intersection point of $\ell$ and $\ell'$, and
the two components of $\ell'\setminus\{y\}$ are 
contained in distinct components of $\Sigma_i\setminus\ell$.

\item[{\rm (3)}]
Let $\Bplane$ be a peripheral hyperplane in $\tXCubing$,
such that $\ell\cap \Bplane\ne\emptyset$.
Then $\ell\cap \Bplane$ consists of a single point, $w$,
and $\pi_{\Bplane}^{-1}(w)=\ell$,
where $\pi_{\Bplane}:\tXCubing \to \Bplane$ is the projection.
\end{enumerate}
\end{corollary}

\begin{proof}
(1) 
Since $\ell\cap\ell'$ is a convex subset of $\ell$,
$\ell\cap\ell'$ is either a singleton or a non-degenerate geodesic
(a geodesic strictly bigger than a singleton).
If $\ell\cap\ell'$ is a non-degenerate geodesic,
then $\ell'$ must be contained in the geodesic line $\ell$,
because every point in $\ell$ has a Euclidean neighborhood 
in $\tXCubing$
by Proposition~\ref{pro:connected-intersection}(1-b)
and Remark~\ref{remark:cubing}(1),
and because, in the Euclidean space, every geodesic has no branching
(see Figure~\ref{fig:branching}(a)).

(2) This follows from 
the fact that the point $y\in\ell\subset \Sigma_i$ has a Euclidean neighborhood in $\Sigma_i$
(by Proposition~\ref{pro:connected-intersection}(1-b)
and Remark~\ref{remark:cubing}(1))
and the fact that 
every geodesic has no branching in the Euclidean plane
(see Figure~\ref{fig:branching}(b)).

(3) It follows from Proposition~\ref{prop:cubing}(1,2)
that $\ell$ intersects $\Bplane$ orthogonally at a single point, $w$,
and that $\ell\subset \pi_{\Bplane}^{-1}(w)$.
To see the converse inclusion, pick a point $z$ $(\ne w)$ of $\pi_{\Bplane}^{-1}(w)$.
Then the geodesic segment of $[z,w]$ is orthogonal to $\Bplane$ at $w$
(cf.~Lemma~\ref{lem:GR2b}).
Since $w$ has a Euclidean neighborhood in $\tXCubing$ 
by Proposition~\ref{pro:connected-intersection}(2-b), 
this implies that
a small neighborhood of $w$ in $[z,w]$ is contained in $\ell$.
Hence $[z,w]\subset \ell$ by the assertion (1).
Thus $z\in \ell$ and so $\pi_{\Bplane}^{-1}(w)=\ell$.
\end{proof}

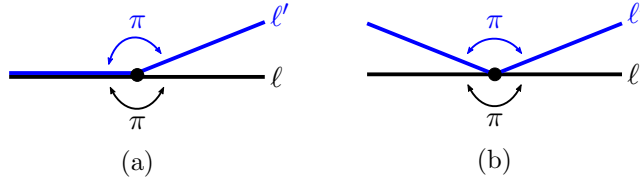
\begin{figure}
  \centering
  \begin{subfigure}{0.25\columnwidth}
    \begin{overpic}[width=\columnwidth]{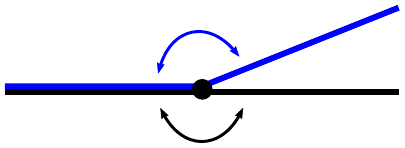}
      \put(46.5,32){\textcolor{blue}{$\pi$}}
      \put(100,33){\textcolor{blue}{$\ell^{\prime}$}}
      \put(46.5,-6){$\pi$}
      \put(100,10){$\ell$}
    \end{overpic}
    \vspace{-2mm}
    \subcaption{}
    \label{fig:branching_a}
  \end{subfigure}
  \hspace{10mm}
  \begin{subfigure}{0.25\columnwidth}
    \begin{overpic}[width=\columnwidth]{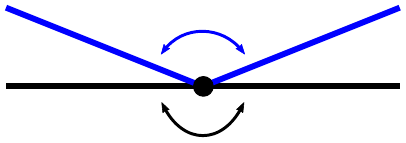}
      \put(46.5,31){\textcolor{blue}{$\pi$}}
      \put(100,33){\textcolor{blue}{$\ell^{\prime}$}}
      \put(46.5,-6){$\pi$}
      \put(100,10){$\ell$}
    \end{overpic}
    \vspace{-2mm}
    \subcaption{}
    \label{fig:branching_b}
  \end{subfigure}
  \caption{Branching of geodesics. 
  Though branching of geodesic can occur in CAT$(0)$ spaces,
  it never occurs in Euclidean spaces.
  }
  \label{fig:branching}
\end{figure}

The following lemma
is used in Sections~\ref{sec:Butterflies-checkerboard} and~\ref{sec:proof-maintheorem}.

\begin{lemma}
\label{lemma:three-components}
Let $\Sigma_1$ and $\Sigma_2$ be disjoint checkerboard hyperplanes in $\tXCubing$.
Then $\Sigma_1\cup\Sigma_2$ divides $\tXCubing$ into three convex subspaces.
To be precise,
there are three closed convex subspaces 
$\Half_1$, $\Half_{1,2}$ and $\Half_2$,
such that 
\[
\tXCubing = \Half_1\cup \Half_{1,2}\cup \Half_2, \quad
\Half_1\cap \Half_{1,2}=\Sigma_1,\quad
\Half_{1,2}\cap \Half_2=\Sigma_2, \quad 
\Half_1\cap \Half_2=\emptyset.
\]
\end{lemma}

\begin{proof}
By Proposition~\ref{prop:cubing2X},
there are closed convex subspaces 
$\Half_i$ and $\Half_i^{\mathfrak{c}}$ such that
$\tXCubing=\Half_i\cup\Half^{\mathfrak{c}}_i$
and $\Sigma_i=\Half_i\cap\Half_i^{\mathfrak{c}}$
($i=1,2$).
Since $\Sigma_1\subset \tXCubing \setminus \Sigma_2
=\interior \Half_2 \sqcup \interior \Half_2^{\mathfrak{c}}$,
we may assume 
$\Sigma_1\subset \interior \Half_2^{\mathfrak{c}}$ and $\Sigma_1\cap \Half_2=\emptyset$.
Similarly, we may assume
$\Sigma_2\subset \interior \Half_1^{\mathfrak{c}}$ and $\Sigma_2\cap \Half_1=\emptyset$. 
Then $\Half_1\cap\Half_2$ is disjoint from
$\Sigma_1\cup\Sigma_2=\partial\Half_1\cup \partial\Half_2$,
and therefore
$\Half_1\cap\Half_2=\interior \Half_1\cap \interior\Half_2$.
Hence $\Half_1\cap\Half_2$ is a closed, open, proper subset of 
$\tXCubing$.
Since $\tXCubing$ is connected, this implies 
$\Half_1\cap\Half_2=\emptyset$.
Thus, by setting $\Half_{1,2}:=\Half_1^{\mathfrak{c}}\cap\Half_2^{\mathfrak{c}}$,
we obtain the desired result.
\end{proof}

\section{Decompositions of alternating link complements
into checkerboard ideal polyhedra}
\label{sec:chekerboard-decomposition}

We recall the (topological) ideal polyhedral decomposition
of the complement $X$ of a prime alternating link $L$
associated with its prime alternating diagram $\Diagram$, 
due to Thurston~\cite{Thurston}, Menasco~\cite{Menasco},
Takahashi \cite{Takahashi} and others,
following the description by
Aitchison-Rubinstein~\cite{Aitchson-Rubinstein_90}
(see also~\cite[Theorem 11.6]{Purcell}).

Regard the prime alternating diagram $\Diagram$
as a $4$-valent graph on the boundary of the $3$-ball $B^3$.
Then $(B^3,\Diagram)$ is regarded as a (topological) polyhedron
(cf.~\cite[Definition 1.1]{Purcell}).
By removing the vertices from  $(B^3,\Diagram)$,
we obtain a (topological) ideal polyhedron,
which we denote by $\PPoly(\Diagram)$. 
Each region $R$ of $\Diagram$ determines the (ideal) face 
$\check R:=R\setminus\{\mbox{vertices}\}$ of $\PPoly(\Diagram)$, 
and each edge
$e$ of $\Diagram$ determines the (ideal) edge $\check e:=\interior e$
of $\PPoly(\Diagram)$.
Prepare two disjoint copies $\PPoly_+(\Diagram)$ and $\PPoly_-(\Diagram)$
of $\PPoly(\Diagram)$,
and glue them together by the following \lq\lq gear rule'':
For each region $R$ of $\Diagram$,
the face $\check R$ of $\PPoly_+(\Diagram)$ is identified with
the face $\check R$ of $\PPoly_-(\Diagram)$
through rotation by one edge 
in the clockwise or anti-clockwise direction
according to whether $R$ is black or white
(see Figure~\ref{fig:CB1}).
(Here, 
we employ the convention that
the twisted bands in the black (resp.~white) surface are 
left-handed (resp.~right-handed)
as in Figure~\ref{fig:Aitchison2}(a).)

\begin{proposition}
\label{prop:polyhedral-decomposition}
Under the above setting, the resulting space is
naturally homeomorphic to the complement $X$ of $L$.
Moreover, the following hold.
\begin{enumerate}
\item[\rm(1)]
The image in $X$ of an edge of $\PPoly_{\pm}(\Diagram)$
is an open crossing arc.
Moreover, the inverse image of each crossing arc in each of 
$\PPoly_{\pm}(\Diagram)$ consists of two edges.
\item[\rm(2)]
If $R$ is a black region of $\Diagram$, then 
the image in $X$ of the face $\check R$ of $\PPoly_{\pm}(\Diagram)$
is equal to the closure of the component
of $S_b\setminus(S_b\cap S_w)$
corresponding to $R$.
Parallel assertions also hold when $R$ is a white region.
\item[\rm(3)]
For each $\epsilon\in\{+,-\}$,
the image of $\PPoly_{\epsilon}(\Diagram)$ in $X$
is equal to the closure of the component of
$X\setminus(S_b\cup S_w)$ containing the point $v_{\epsilon}$.
\end{enumerate}
\end{proposition}

\begin{figure}
  \centering
  \begin{overpic}[width=0.7\columnwidth]{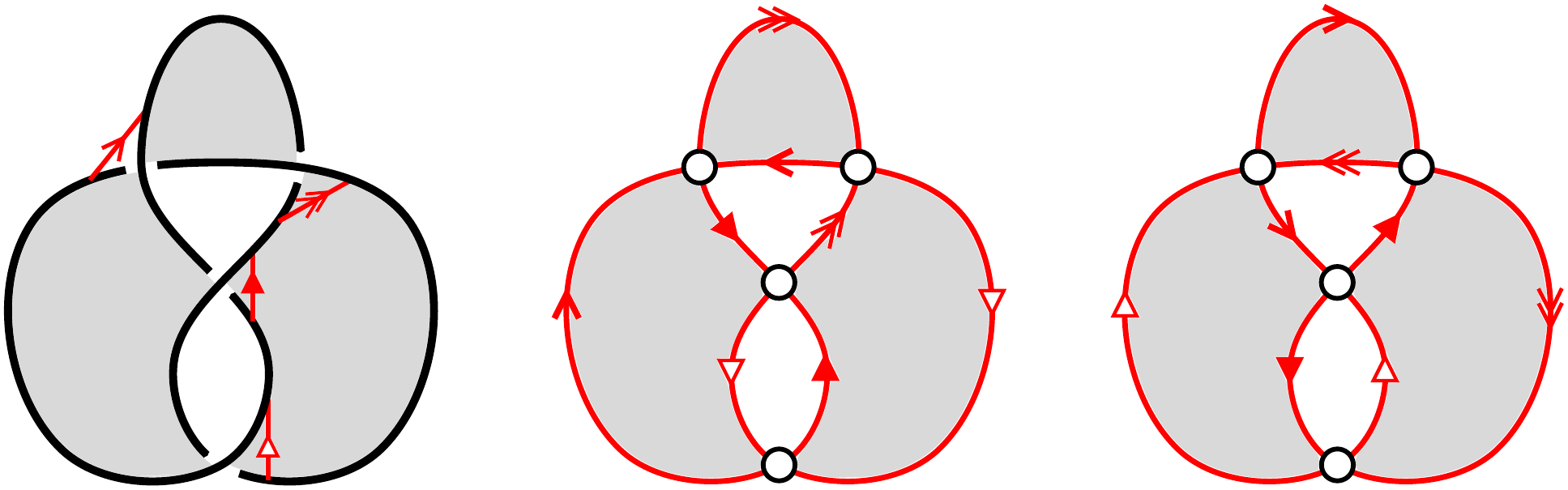}
    \put(33,-15){$D_L$}
    \put(128,-15){$\PPoly_+(\Diagram)$}
    \put(228,-15){$\PPoly_-(\Diagram)$}
  \end{overpic}
  \vspace{4mm}
  \caption{The decomposition of the figure-eight knot complement
  into a pair of checkerboard ideal polyhedra:
  The shaded regions are black regions and the unshaded regions are white regions.}
  \label{fig:CB1}
\end{figure}

Though the natural maps from $\PPoly_{\pm}(\Diagram)$ to $X$
are not injective on the $1$-skeletons,
their lifts to the universal cover $\tilde X$ are
homeomorphisms onto their images, 
each of which is equal to the closure
of a component of $\tilde X\setminus \pu^{-1}(S_b\cup S_w)$.
Thus we obtain a tessellation of $\tilde X$ by copies of
$\PPoly_{+}(\Diagram)$
and $\PPoly_{-}(\Diagram)$,
where the wall is $\pu^{-1}(S_b\cup S_w)$,
the union of all checkerboard planes.

\medskip

By working in the non-positively curved cubed decomposition
$\XCubing$ of $X$
and the CAT(0) cubed decomposition $\tXCubing$ of $\tilde X$,
we can refine the above topological picture
into the geometric picture explained below.

Recall that the open checkerboard surfaces $S_b$ and $S_w$ in $X$
are isotopic to the hyperplanes $\Sb$ and $\Sw$
in the non-positively curved cubed complex $\XCubing$. 
They intersect orthogonally along $\crossing=\Sb\cap \Sw$,
the disjoint union of geodesic lines representing open crossing arcs. 
The union $\Sbw:=\Sb\cup\Sw$ cuts $\XCubing$ into two connected components.
We denote 
by $\Poly_+$ and $\Poly_-$, the closures
of the components of $\XCubing\setminus \Sbw$ containing the 
vertices $v_+$ and $v_-$, respectively.
Then $\Poly_{\pm}$ is naturally homeomorphic to
the image of $\PPoly_{\pm}(\Diagram)$ in $X$. 

In the universal cover $\tXCubing$,
both $\tSb=\pu^{-1}(\Sb)$ and $\tSw=\pu^{-1}(\Sw)$ are disjoint unions of checkerboard hyperplanes,
and they intersect orthogonally along $\tilde\crossing:=\tSb\cap\tSw$.
The union $\tSbw=\tSb\cup\tSw$ 
of all checkerboard hyperplanes divides $\tXCubing$
into infinitely many \lq\lq right-angled, cubed, ideal polyhedra'',
and we obtain the following proposition.

\begin{proposition}
\label{prop:checkerboard-polyhedron}
Let $\tPoly_{\epsilon}$ be the closure of a component of 
$\tXCubing\setminus \tSbw$ 
which projects to $\Poly_{\epsilon}\subset \XCubing$
($\epsilon\in\{+,-\}$).
Then $\tPoly_{\epsilon}$ admits a natural structure
of a (topological) ideal polyhedron
with respect to which
there is an isomorphism $\varphi_{\epsilon}:\PPoly(\Diagram)\to \tPoly_{\epsilon}$
satisfying the following conditions.
\begin{enumerate}
\item[\rm(1)]
For each region $R$ of $\Diagram$,
there is a checkerboard hyperplane, 
$\Sigma_R=\Sigma_R(\tPoly_{\epsilon})$,
satisfying the following conditions.
\begin{enumerate}[\rm (a)]
\item 
$\varphi_{\epsilon}(\check R)=\partial \tPoly_{\epsilon}\cap \Sigma_R$.
\item
If $R$ is a black region, 
then $\pu(\Sigma_R)=\Sb$,
and the restriction of the universal covering projection 
$\pu|_{\Sigma_R}:\Sigma_R \to\Sb$ to the face 
$\varphi_{\epsilon}(\check R)$ is a homeomorphism
onto the closure of the component of $\Sb\setminus\crossing$
corresponding to $R$.
Parallel assertions also hold when $R$ is a white region.
\end{enumerate}
\item[\rm(2)]
For each region $R$ of $\Diagram$,
let $\Half^{\mathfrak{c}}_{R}=\Half^{\mathfrak{c}}_{R}(\tPoly_{\epsilon})$ be the checkerboard hyper-half-space bounded by
$\Sigma_R$ that contains $\tPoly_{\epsilon}$.
Then $\tPoly_{\epsilon}=\bigcap_R \Half^{\mathfrak{c}}_{R}$,
where $R$ runs over the regions of $\Diagram$.
In particular, $\tPoly_{\epsilon}$ is convex in the CAT(0) space 
$\tXCubing$.
\item[\rm(3)]
Let $e$ be an edge of $\Diagram$ and let $R_1$ and $R_2$ be 
the regions of $\Diagram$
sharing $e$.
Then the two faces $\varphi_{\epsilon}(\check R_1)$ and
$\varphi_{\epsilon}(\check R_2)$ intersect orthogonally along
the edge $\varphi_{\epsilon}(\check e)$.
The edge $\varphi_{\epsilon}(\check e)$ projects 
to a geodesic line in $\XCubing$ representing an open crossing arc.

\end{enumerate}

Moreover, $\varphi_+$ and $\varphi_-$ are related as explained below.
Note that,
for each region $R$ of $\Diagram$, 
$(\pu\circ\varphi_{\pm})|_{\check R}$ are homeomorphisms 
with the same image,
and so the composition 
$(\pu\circ\varphi_{-})|_{\check R} \circ 
{(\pu\circ\varphi_{+})|_{\check R}}^{-1}$
is a well-defined automorphism of 
the ideal polygon $\check R$.
This automorphism is 
a rotation by one edge 
in the clockwise or anti-clockwise direction
according to whether $R$ is black or white.
\end{proposition}

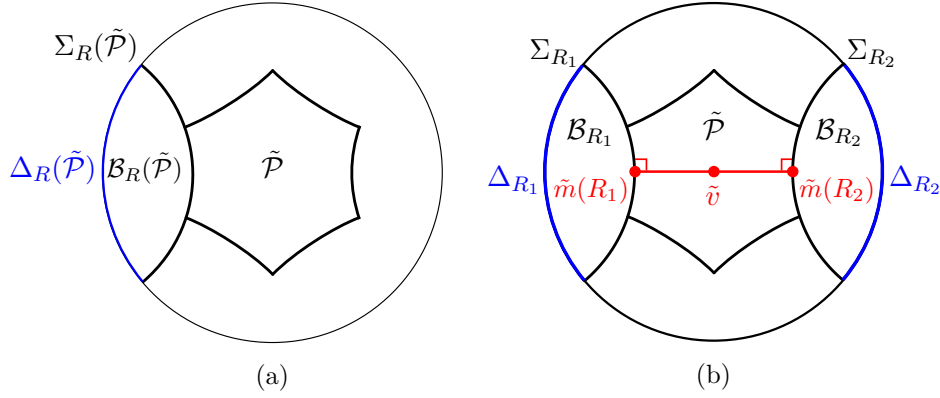
\begin{figure}
  \centering
  \begin{tabular}{c}
    \begin{subfigure}{0.33\columnwidth}
      \begin{overpic}[width=\columnwidth]{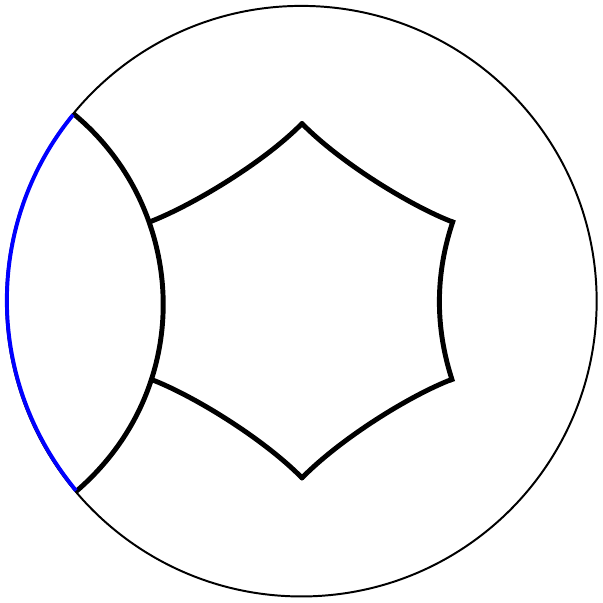}
        \put(-34,64){\textcolor{blue}{$\Delta_R(\tPoly)$}}
        \put(-17,110){$\Sigma_R(\tPoly)$}
        \put(3,64){\small $\Half_R(\tPoly)$}
        \put(61.5,64){$\tPoly$}
      \end{overpic}
      \subcaption{}
      \label{fig:disjoint-halfspace_a}
    \end{subfigure}
    \hspace{10mm}
    \begin{subfigure}{0.33\columnwidth}
      \begin{overpic}[width=\columnwidth]{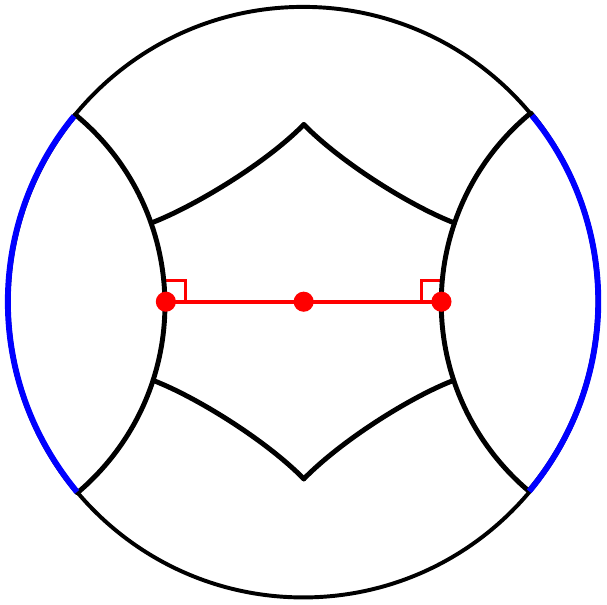}
        \put(-20,60){\textcolor{blue}{$\Delta_{R_1}$}}
        \put(132,60){\textcolor{blue}{$\Delta_{R_2}$}}
        \put(-4,108){$\Sigma_{R_1}$}
        \put(116,108){$\Sigma_{R_2}$}
        \put(10,78){$\Half_{R_1}$}
        \put(104,78){$\Half_{R_2}$}
        \put(5,56){\small \textcolor{red}{$\tilde{m}(R_1)$}}
        \put(98,56){\small \textcolor{red}{$\tilde{m}(R_2)$}}
        \put(61.5,78){$\tPoly$}
        \put(62.5,53){\textcolor{red}{$\tilde{v}$}}
      \end{overpic}
      \subcaption{}
      \label{fig:disjoint-halfspace_b}
    \end{subfigure}
  \end{tabular}
  \caption{
    (a) $\Half_{R}(\tPoly)$ is the checkerboard hyper-half-space in $\tXCubing$ bounded by 
    the hyperplane $\Sigma_R(\tPoly)$ which is {\it disjoint} from 
    $\interior\tPoly$. 
     (This $2$-dimensional figure does not reflect the fact that $\tPoly$ is an {\it ideal} polyhedron.)
    (b) If $R_1$ and $R_2$ are not adjacent, then 
    $\angle_{\tilde v}(\tilde m(R_1),\tilde m(R_2))=\pi$,
    and hence $[\tilde m(R_1),\tilde v]\cup [\tilde v, \tilde m(R_2)]$ 
    is a common perpendicular to  $\Sigma_{R_1}$ and $\Sigma_{R_2}$.
    This implies $\Half_{R_1}(\tPoly)\cap \Half_{R_2}(\tPoly)=\emptyset$.
     }
  \label{fig:disjoint-halfspace}
\end{figure}

The proposition is obtained by looking 
Proposition~\ref{prop:polyhedral-decomposition}
in the setting of Proposition~\ref{prop:cubing}.
The convexity of $\tPoly_{\epsilon}$ in (2) 
is a consequence of Proposition~\ref{prop:cubing2X}.

\begin{definition-notation}
\label{def:checkerboard-polyhedron}
{\rm
We call $\tPoly_{\epsilon}$ a {\it checkerboard ideal polyhedron} in $\tXCubing$.
The unique point $\tilde v_{\epsilon}\in \pu^{-1}(v_{\epsilon})$
contained in $\tPoly_{\epsilon}$ is called 
the {\it center} of $\tPoly_{\epsilon}$.

When we do not mind the sign $\epsilon$,
we drop it from the symbols, such as
$\tPoly_{\epsilon}$ 
and $\varphi_{\epsilon}$.
For a fixed checkerboard ideal polyhedron $\tPoly$
and for a region $R$ of $\Diagram$, 
we use the following terminology and notation.
\begin{enumerate}
\item[\rm(1)]
The face $\varphi(\check R)$ of $\tPoly$
is called the {\it face $R$ of $\tPoly$}.
\item[\rm(2)]
The {\it center} $\tilde m(R)$ of the face $R$ of $\tPoly$
is defined as follows.
By Proposition~\ref{prop:checkerboard-polyhedron}(1-b), 
$\pu$ determines a homeomorphism 
from $\varphi(\check R)$ to the closure 
of the component of $\Sbw\setminus\crossing$
containing the center $m(R)$.
Then $\tilde m(R)\in \varphi(\check R)$ is the 
inverse image of $m(R)$.
\item[\rm(3)]
$\Sigma_R=\Sigma_R(\tPoly)$ denotes the checkerboard hyperplane in $\tXCubing$
containing the face $R$ of $\tPoly$.
\item[\rm(4)]
$\Half_R=\Half_R(\tPoly)$ and 
$\Half_R^{\mathfrak{c}}=\Half_R^{\mathfrak{c}}(\tPoly)$ denote
the checkerboard hyper-half-spaces in $\tXCubing$
bounded by $\Sigma_R(\tPoly)$, such that
$\tPoly\subset \Half_R^{\mathfrak{c}}(\tPoly)$ and 
$\Half_R(\tPoly)\cap \Half_R^{\mathfrak{c}}(\tPoly)=\Sigma_R(\tPoly)$
(see Figure~\ref{fig:disjoint-halfspace}(a)).
\end{enumerate}
}
\end{definition-notation}

Then we have the following proposition, which plays a key role in the proof of Theorem~\ref{Theorem1-0}.

\begin{proposition}
\label{prop:disjoint-half-space}
Let $\tPoly \subset \tXCubing$ be a checkerboard ideal polyhedron, and
let $R_1$ and $R_2$ be distinct regions of $\Diagram$.
Then $\Half_{R_1}=\Half_{R_1}(\tPoly)$ and $\Half_{R_1}=\Half_{R_2}(\tPoly)$
are disjoint if and only if
$R_1$ and $R_2$ are not adjacent.
\end{proposition}

\begin{proof}
Let $\tilde v$ be the center of $\tPoly$,
and let $\tilde m(R_i)$ be the center of the face $R_i$ of $\tPoly$ ($i=1,2$).
Then the geodesic segment $[\tilde v, \tilde m(R_i)]$ is perpendicular to 
$\Sigma_{R_i}=\Sigma_{R_i}(\tPoly)$ by 
Proposition~\ref{prop:cubing}(4) 
(cf.~Remark~\ref{remark:cubing}(2))
and Remark~\ref{rem:localness-angle}.
Note that there is a natural isomorphism
$\Lk(\tilde v,\tXCubing)\cong \Lk(v,\Cubing)$,
where $v=\pu(\tilde v)$.
Thus, by Remark~\ref{remark:cubing}(4), 
$\angle_{\tilde v}(\tilde m(R_1),\tilde m(R_2))$
is equal to $\pi/2$ or $\pi$
according to whether $R_1$ and $R_2$ are adjacent or not.
Hence, if $R_1$ and $R_2$ are not adjacent,
then $[\tilde m(R_1),\tilde v]\cup [\tilde v, \tilde m(R_2)]$ 
is a geodesic
which is perpendicular to the checkerboard hyperplanes 
$\Sigma_{R_1}$ and $\Sigma_{R_2}$
at their endpoints.
(In fact, it is a local geodesic 
by~\cite[Remark I.5.7 and Theorem I.7.39]{BH} 
and so it is a geodesic by~\cite[Proposition II.1.4(2)]{BH}.)
Hence, by Lemma~\ref{lem:GR2}, it is a shortest path between
the hyperplanes, and  
in particular, $\Sigma_{R_1}$ and $\Sigma_{R_2}$ are disjoint.
By Lemma~\ref{lemma:three-components},
$\Sigma_{R_1}\cup \Sigma_{R_2}$ divides $\tXCubing$ into
three closed convex subspaces 
$\Half_1$, $\Half_{1,2}$ and $\Half_2$,
that satisfy the condition in the lemma.
Since $\tPoly$ intersects both $\Sigma_{R_1}$ and $\Sigma_{R_2}$,
we have $\tPoly\subset \Half_{1,2}$.
This implies that 
$\Half_i=\Half_{R_i}$ ($i=1,2$).
Hence $\Half_{R_1}$ and $\Half_{R_2}$ are disjoint
(see Figure~\ref{fig:disjoint-halfspace}(b)).

On the other hand, if the regions $R_1$ and $R_2$ are adjacent
in $D$, 
then the faces $R_1$ and $R_2$ of $\tPoly$ are adjacent.
Thus
$\Sigma_{R_1}\cap \Sigma_{R_2}$
is a geodesic line 
(cf.~Proposition~\ref{pro:connected-intersection}(1))
and hence $\Half_{R_1}$ and $\Half_{R_2}$ are not disjoint.
\end{proof}

\begin{figure}
  \centering
  \begin{tabular}{c}
    \begin{subfigure}{0.28\columnwidth}
      \begin{overpic}[width=\columnwidth]{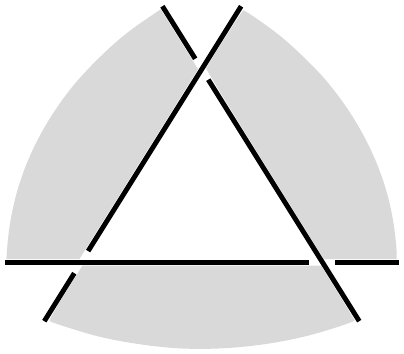}
        \put(80,52){$R_{b,1}$}
        \put(13,52){$R_{b,2}$}
        \put(47,9){$R_{b,3}$}
        \put(49,39){$R_{w}$}
      \end{overpic}
      \vspace{1mm}
      \subcaption{}
      \label{fig:adjacent-faces_a}
    \end{subfigure}
    \hspace{15mm}
    \begin{subfigure}{0.33\columnwidth}
      \begin{overpic}[width=\columnwidth]{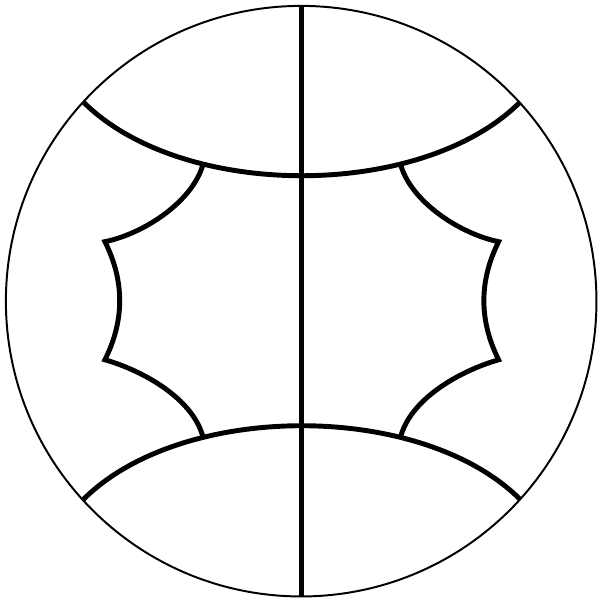}
        \put(-31,112){$\Sigma_{R_{b,i}}(\tPoly_+)$}
        \put(114,112){$\Sigma_{R_{b,i+1}}(\tPoly_-)$}
        \put(-30,12){$\Sigma_{R_{b,j}}(\tPoly_+)$}
        \put(114,12){$\Sigma_{R_{b,j+1}}(\tPoly_-)$}
        \put(14,133){$\Sigma_{R_{w}}(\tPoly_+) = \Sigma_{R_{w}}(\tPoly_-)$}
        \put(40,60){$\tPoly_+$}
        \put(78,60){$\tPoly_-$}
      \end{overpic}
      \subcaption{}
      \label{fig:adjacent-faces_b}
    \end{subfigure}
  \end{tabular}
  \caption{The checkerboard polyhedra $\tPoly_{+}$ and $\tPoly_{-}$ 
share a face that is contained in the checkerboard hyperplane
$\Sigma_{R_w}(\tPoly_{+})=\Sigma_{R_w}(\tPoly_{-})$.
Then $\Sigma_{R_{b,i}}(\tPoly_+)=\Sigma_{R_{b,i+1}}(\tPoly_-)$.
  }
  \label{fig:adjacent-faces}
\end{figure}

At the end of this section, 
we note the following observation (see Figure~\ref{fig:adjacent-faces}),
which is used in Section~\ref{sec:proof-maintheorem}.

\begin{lemma}
\label{lem:common-plane}
Let $R_w$ be a white region of $\Diagram$
and $R_{b,i}$ ($1\le i\le n$) be the black regions of $\Diagram$
which are adjacent to $R_w$ and which are arranged around $R_w$
in this cyclic order with respect to the anti-clockwise orientation of 
$\partial R_w$.
Let $\tPoly_{\pm} \subset \tXCubing$ be the checkerboard ideal polyhedra,
such that $\tPoly_{+}\cap \tPoly_{-}$ is 
the face $R_w$ of both $\tPoly_{+}$ and $\tPoly_{-}$.
Then we have $\Sigma_{R_{b,i}}(\tPoly_+)=\Sigma_{R_{b,i+1}}(\tPoly_-)$
and $\Half_{R_{b,i}}(\tPoly_+)=\Half_{R_{b,i+1}}(\tPoly_-)$,
where the index $i$ is considered with modulo $n$.
When the colors black and white are interchanged,
similar assertion holds.

\end{lemma} 

\begin{proof}
For $\epsilon\in \{+,-\}$,
let $\varphi_{\epsilon}$ be the isomorphism from 
$\PPoly(\Diagram)$ to the checkerboard ideal polyhedron $\tPoly_{\epsilon}$.
Then $\varphi_{+}(\check R_w)=\varphi_{-}(\check R_w)$ by the assumption.
Consider the edge $e_i:=R_w \cap R_{b,i}$ of $\Diagram$.
Then, by 
the last assertion of Proposition~\ref{prop:checkerboard-polyhedron},
we see that $\varphi_{+}(\check e_i)=\varphi_{-}(\check e_{i+1})$ and that
it is a common edge of the faces 
$\varphi_{+}(\check R_{b,i})$
and $\varphi_{-}(\check R_{b,i+1})$.
Since $\tPoly_{\pm}$ are right-angled cubed polyhedra
(cf.~Proposition ~\ref{prop:checkerboard-polyhedron}(3)),
this implies that the two faces are contained in a single checkerboard hyperplane,
which is equal to $\Sigma_{R_{b,i}}(\tPoly_+)=\Sigma_{R_{b,i+1}}(\tPoly_-)$.
Since $\tPoly_{+}$ and $\tPoly_{-}$ share the common face
$\varphi_{+}(\check R_w)=\varphi_{-}(\check R_w)$,
we also have $\Half_{R_{b,i}}(\tPoly_+)=\Half_{R_{b,i+1}}(\tPoly_-)$.
\end{proof}

%\begin{figure}
%  \centering
%  \begin{tabular}{c}
%    \begin{subfigure}{0.28\columnwidth}
%      \begin{overpic}[width=\columnwidth]{adjacent-faces_a.pdf}
%        \put(80,52){$R_{b,1}$}
%        \put(13,52){$R_{b,2}$}
%        \put(47,9){$R_{b,3}$}
%        \put(49,39){$R_{w}$}
%      \end{overpic}
%      \vspace{1mm}
%      \subcaption{}
%      \label{fig:adjacent-faces_a}
%    \end{subfigure}
%    \hspace{15mm}
%    \begin{subfigure}{0.33\columnwidth}
%      \begin{overpic}[width=\columnwidth]{adjacent-faces_b.pdf}
%        \put(-31,112){$\Sigma_{R_{b,i}}(\tPoly_+)$}
%        \put(114,112){$\Sigma_{R_{b,i+1}}(\tPoly_-)$}
%        \put(-30,12){$\Sigma_{R_{b,j}}(\tPoly_+)$}
%        \put(114,12){$\Sigma_{R_{b,j+1}}(\tPoly_-)$}
%        \put(14,133){$\Sigma_{R_{w}}(\tPoly_+) = \Sigma_{R_{w}}(\tPoly_-)$}
%        \put(40,60){$\tPoly_+$}
%        \put(78,60){$\tPoly_-$}
%      \end{overpic}
%      \subcaption{}
%      \label{fig:adjacent-faces_b}
%    \end{subfigure}
%  \end{tabular}
%  \caption{The checkerboard polyhedra $\tPoly_{+}$ and $\tPoly_{-}$ 
%share a face that is contained in the checkerboard hyperplane
%$\Sigma_{R_w}(\tPoly_{+})=\Sigma_{R_w}(\tPoly_{-})$.
%Then $\Sigma_{R_{b,i}}(\tPoly_+)=\Sigma_{R_{b,i+1}}(\tPoly_-)$.
%  }
%  \label{fig:adjacent-faces}
%\end{figure}

\section{Butterflies and checkerboard ideal polyhedra}
\label{sec:Butterflies-checkerboard}

The complement $X$ of a hyperbolic alternating link $L$
with a prescribed prime alternating diagram $\Diagram$
admits two distinct geometric structures given as:
\begin{itemize}
\item[-]
the complete hyperbolic manifold $\HH^3/G$, and
\item[-]
the underlying space of the non-positively curved cubed complex $\XCubing$ 
that is constructed from a prime alternating diagram $\Diagram$ of $L$.
\end{itemize}
We fix homeomorphisms 
\[
(X,M)\cong(\HH^3/G,(\HH^3\setminus\mathcal{Q})/G)\cong 
(|\XCubing|,|\Cubing|),
\]
and identify the relevant spaces through the homeomorphisms.
Here $\mathcal{Q}$ is the disjoint union of 
the open horoballs bounded by the horospheres $\{H_p\}_{p\in\PFix(G)}$
introduced in Section~\ref{sec:meridian}
(the paragraph after Lemma~\ref{lem:non-commutative}).
This identification induces the following $G$-equivariant identifications
of the universal covering spaces
\[
(\tilde X,\tilde M)=(\HH^3,\HH^3\setminus\mathcal{Q})=
(|\tXCubing|,|\tCubing|).
\]
In particular,
each horosphere $\Bplane_p\subset \HH^3$ is regarded as a peripheral plane 
contained in $\partial\tCubing$ in the CAT(0) space $\tXCubing$;
so we call it the {\it peripheral plane centered at $p$}.

We also assume that the quasi-fuchsian checkerboard surfaces $S_b$ and $S_w$
in the hyperbolic manifold $X=\HH^3/G$
(cf.~Sections~\ref{sec:checkerboard-surface} and~\ref{sec:meridian})
are the hyperplanes $\Sb$ and $\Sw$, respectively, in 
the non-positively curved cubed complex $\XCubing$ 
(cf.~
Sections~\ref{sec:NPC} and~\ref{sec:chekerboard-decomposition}).
Thus each checkerboard plane $\Sigma\subset \HH^3$
is a checkerboard hyperplane in the CAT(0) cubed complex $\tXCubing$.

For a checkerboard ideal polyhedron 
$\tPoly\subset \tXCubing=\HH^3$,
let $\hPoly$ be 
the closure of $\tPoly$ in 
$\HHH^3=\HH^3\cup \CCC$.
Then the isomorphism $\varphi:\PPoly(\Diagram)\to \tPoly$
(between topological ideal polyhedra)
extends to an isomorphism
$\hat\varphi:(B^3,\Diagram)\to \hPoly$
(between topological polyhedra),
because $\tilde X\setminus \tilde M$ is identified with
the disjoint family of open horoballs $\mathcal{Q}$ centered at
points in $\PFix(G)$. 
For each vertex $c$ of $\Diagram$,
the ideal point $p:=\hat\varphi(c)$ belongs to $\PFix(G)$,
and we call $p$ the {\it ideal vertex of $\tPoly$ corresponding to $c$}.
We also call $c$ the {\it vertex of $\Diagram$ 
corresponding to the ideal vertex $p$ of $\tPoly$}.

We introduce the following notation for objects in 
the closure $\HHH^3$ of the hyperbolic space,
building on Definition and Notation~\ref{def:checkerboard-polyhedron}
for objects in the CAT(0) cubed complex $\tXCubing$
(cf.~Figure~\ref{fig:disjoint-halfspace}(a)).

\begin{notation}
\label{notation:idealpolyhedron2}
{\rm
Let $\tPoly\subset \tXCubing$ be a checkerboard ideal polyhedron,
and $R$ a region of the diagram $\Diagram$.
\begin{enumerate}
\item
$\bSigma_R=\bSigma_R(\tPoly)$ denotes the checkerboard disk properly embedded in $\HHH^3$
obtained as the closure of $\Sigma_R(\tPoly)\subset \tXCubing=\HH^3$.
\item
$\bHalf_R=\bHalf_R(\tPoly)$ and 
$\bHalf^{\mathfrak{c}}_R=\bHalf^{\mathfrak{c}}_R(\tPoly)$ denote the $3$-balls in $\HHH^3$
obtained as the closures of the checkerboard hyper-half-spaces
$\Half_R(\tPoly)$ and $\Half_R^{\mathfrak{c}}(\tPoly)$.
Note that $\tPoly\subset \bHalf^{\mathfrak{c}}_R(\tPoly)$ and 
$\bHalf_R(\tPoly)\cap \bHalf^{\mathfrak{c}}_R(\tPoly)=\bSigma_R(\tPoly)$.
\item 
$\Delta_R(\tPoly)$ denotes the disk in $\CCC$ defined by
$\Delta_R(\tPoly):=\bHalf_R(\tPoly)\cap\CCC$. 
\end{enumerate}
}
\end{notation}

Then we have the following lemma.

\begin{lemma}
\label{lemma:Ideal-Halfspace}
Let $\tPoly_1$ and $\tPoly_2$ be checkerboard ideal polyhedra,
and let $R_1$ and $R_2$ be regions of $D$.
If the checkerboard hyper-half-spaces 
$\Half_{R_1}(\tPoly_1)$ and $\Half_{R_2}(\tPoly_2)$ in $\HH^3$
are disjoint, then
the two disks $\Delta_{R_1}(\tPoly_1)$ and $\Delta_{R_2}(\tPoly_2)$
have disjoint interiors in $\CCC$.
\end{lemma}

\begin{proof}
By Corollary~\ref{cor:qf},
the pair
$(\HHH^3,\bHalf_{R_i})$, with
$\Half_{R_i}=\Half_{R_i}(\tPoly_i)$,
is homeomorphic to the standard pair
$(B^3,B^3_+)$ of the unit $3$-ball $B^3$ in $\RR^3$ and
the closed upper half-ball
$B^3_+=\{(x,y,z)\in B^3 \ | \ z\ge 0\}$
($i=1,2$).
Thus a point $x\in \CCC$ belongs to the interior of 
$\Delta_{R_i}:=\Delta_{R_i}(\tPoly_i)$
if and only if there is a neighborhood $U$
of $x$ in $\HHH^3$ such that $U\cap\HH^3 \subset \Half_{R_i}$
($i=1,2$).
So $x$ belongs to 
$\interior\Delta_{R_1}\cap \interior\Delta_{R_2}$
if and only if 
there is a neighborhood $U$
of $x$ in $\HHH^3$ such that $U\cap\HH^3 \subset \Half_{R_1}\cap \Half_{R_2}$.
Hence, if $\Half_{R_1}\cap \Half_{R_2}= \emptyset$
then $\interior\Delta_{R_1}\cap \interior\Delta_{R_2}= \emptyset$.
\end{proof}

Proposition~\ref{prop:disjoint-half-space} together with 
Lemma~\ref{lemma:Ideal-Halfspace} implies 
the following proposition, which plays a key role 
in the proof of Theorem~\ref{Theorem1}.

\begin{proposition}
\label{prop:disjoint-open-disk}
Let $\tPoly\subset \tXCubing$ be a checkerboard ideal polyhedron, and
let $R_1$ and $R_2$ be distinct regions of $\Diagram$.
If $R_1$ and $R_2$ are not adjacent, then
$\Delta_{R_1}(\tPoly)$ and $\Delta_{R_2}(\tPoly)$
have disjoint interiors in $\CCC$.
\end{proposition}

\begin{remark}
\label{rem:disjoint-open-disk}
{\rm
In Lemma~\ref{lemma:Ideal-Halfspace} and 
Proposition~\ref{prop:disjoint-open-disk},
the converses also hold.
Actually,
Proposition~\ref{prop:disjoint-open-disk} reflects only a small
part of a very interesting statement in Agol's slide~\cite{Agol2001},
which we read as follows.
Aitchison and Rubinstein (cf.~\cite{Aitchson-Rubinstein_90a}) 
studied patterns of the intersections of the limit circles
$\{\partial_{\infty}\Sigma\}$
of the checkerboard hyperplanes 
in the ideal boundary 
$\partial_{\infty}\tXCubing$ of the CAT(0) space $\tXCubing$:
Put a circle around each region of $\Diagram$,
then the limit circles
$\{\partial_{\infty}(\Sigma_{R}(\tPoly))\}_R$
\lq\lq have this intersection pattern'' in $\partial_{\infty}\tXCubing$
(see Figure~\ref{fig:intersection-pattern}).
We hope to give more detailed interpretation of 
this statement
in a subsequent paper.
}
\end{remark}

\begin{figure}
  \centering
  \begin{overpic}[width=0.7\columnwidth]{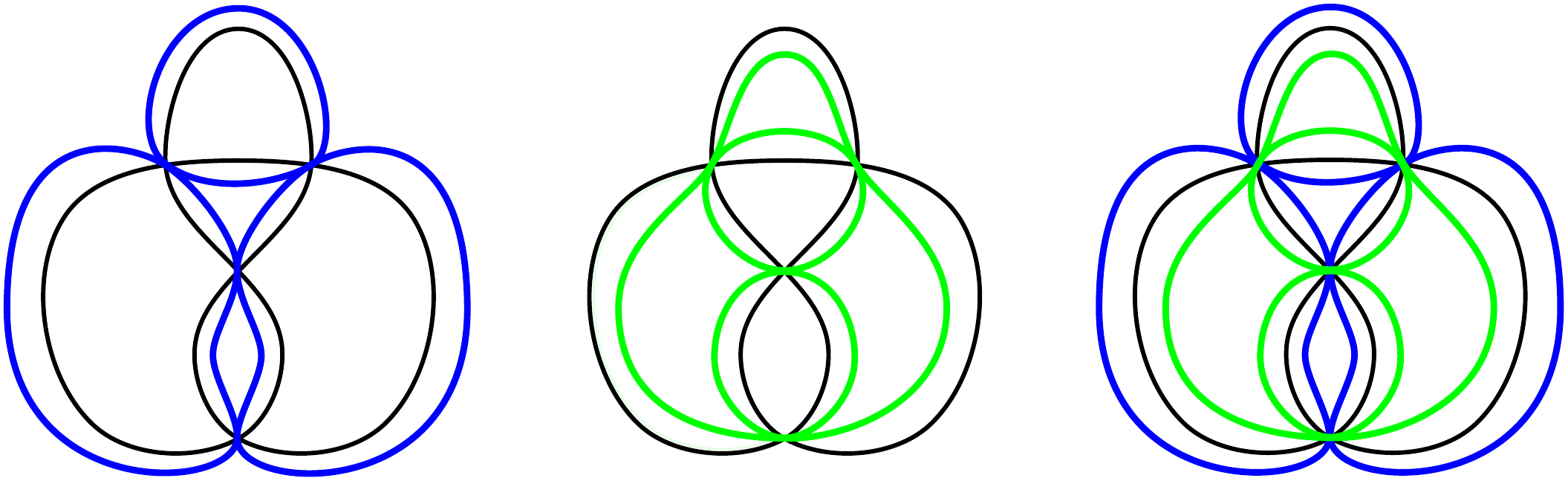}
    \put(35,-13){(a)}
    \put(132,-13){(b)}
    \put(230,-13){(c)}
  \end{overpic}
  \vspace{4mm}
  \caption{Put a circle $C_R$ around each region $R$
  of the diagram $\Diagram$, so that $C_R$ bounds a disk containing $R$
  and passes though the vertices on $\partial R$.
  The figures (a) and (b) illustrates the circles $\{C_R\}$
  where $R$ runs over the black or white regions, respectively.
  By overlaying these two figures, we obtain the figure (c).
  According to \cite{Agol2001}, the figure (c) \lq\lq illustrates''
  the intersection pattern of the limit circles $\{\partial_{\infty}(\Sigma_{R}(\tPoly))\}_R$, where $R$ runs over the regions of $\Diagram$.}
  \label{fig:intersection-pattern}
\end{figure}

The following characterization of butterflies
is used repeatedly
in the proof of Theorem~\ref{Theorem1}.

\begin{figure}
  \centering
  \begin{overpic}[width=0.3\columnwidth]{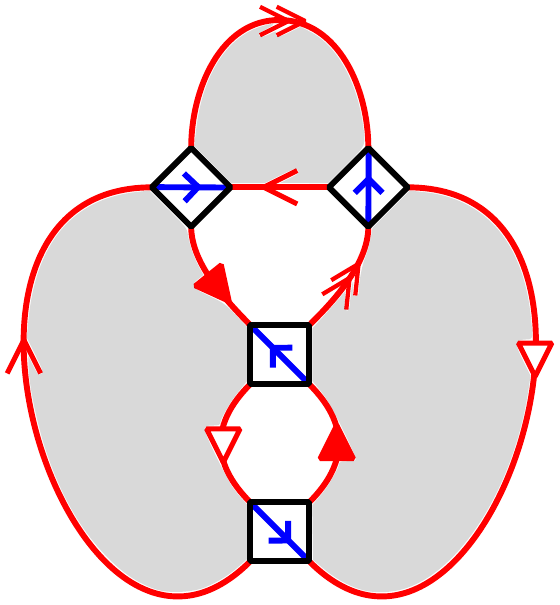}
  \end{overpic}
  \caption{The truncation $\tPoly_0$ of the checkerboard ideal polyhedron 
  $\tPoly=\tPoly_+$.
  The blue diagonal arcs in the squares project to meridians.}
  \label{fig:CB2}
\end{figure}

\begin{lemma}
\label{lem:charaterizing-butterfly}
Let $p\in \PFix(G)$ be a parabolic fixed point and
$\tPoly$ an ideal checkerboard polyhedron which has $p$ as an ideal vertex.
Let $c$ be the vertex of $\Diagram$ corresponding to $p$,
and let $\{R^-,R^+\}$ be a pair of regions
that contain $c$ and have the same color. 
Then, after replacing $R^{\pm}$ with $R^{\mp}$ if necessary, 
the pair $\{\Delta_{R^-}(\tPoly), \Delta_{R^+}(\tPoly)\}$ forms a butterfly $\BF(p)$ at $p$ (in the sense of Notation~$\ref{notation:butterfly}$).
Conversely, every butterfly is obtained in this way.
\end{lemma}

\begin{proof}
Consider the compact right-angled polyhedron $\tPoly_0$ obtained from $\tPoly$
through truncation along the peripheral planes $\{\Bplane_p\}$
(see Figure~\ref{fig:CB2}).
Then, for each ideal vertex $p$ of $\tPoly$,
the intersection $\tPoly\cap\Bplane_p$ forms a square in $\partial \tPoly_0$,
one of whose diagonals projects to a meridian (see Figure~\ref{fig:CB2}).
By using this fact, we can see that 
the meridian $\mu_p\in G$ maps 
the checkerboard hyperplane $\Sigma_{R^-}(\tPoly)$ to 
the checkerboard hyperplane $\Sigma_{R^+}(\tPoly)$,
if necessary after replacing $R^{\pm}$ with $R^{\mp}$.
We can further see that 
$\mu_p$ maps he ball pair $(\bHalf_{R^-}(\tPoly),\bHalf_{R^-}^{\mathfrak{c}}(\tPoly))$
to the ball pair $(\bHalf_{R^+}^{\mathfrak{c}}(\tPoly),\bHalf_{R^+}(\tPoly))$.
This implies that
$\{\Delta_{R^-}(\tPoly), \Delta_{R^+}(\tPoly)\}$ is a butterfly at $p$.

To see the converse,
let  $\BF(p)=\{\Delta_j^-, \Delta_{j+1}^+\}=\{\Delta_j^-(p), \Delta_{j+1}^+(p)\}$ be a butterfly,
where we use notations in Definition~\ref{def:butterfly}.
Consider the infinite strip in the peripheral plane (or the horosphere) 
$\Bplane_p\subset \pu^{-1}(\partial \Cubing)$
bounded by the lines 
$\ell_j(p)=\Sigma_j(p)\cap \Bplane_p$ and 
$\ell_{j+1}(p)=\Sigma_{j+1}(p)\cap \Bplane_p$
(see Figure~\ref{fig:Butterfly}).
Let $\tPoly$ be a checkerboard ideal polyhedron 
which has $p$ as an ideal vertex, such that
$\tPoly\cap\Bplane_p$ is a square contained in the strip.
Then there are regions $R^-$ and $R^+$ of $\Diagram$
containing the vertex $c$ of $\Diagram$ corresponding to the ideal vertex $p$
of $\tPoly$,
such that $\Sigma_j(p)=\Sigma_{R^-}(\tPoly)$
and $\Sigma_{j+1}(p)=\Sigma_{R^+}(\tPoly)$.
Then we see $\BF(p)=\{\Delta_j^-, \Delta_{j+1}^+\}=
\{\Delta_{R^-}(\tPoly), \Delta_{R^+}(\tPoly)\}$.
\end{proof}

At the end of this subsection,
we prove the following elementary lemma
concerning prime alternating diagrams
of hyperbolic alternating links,
which is used in Section~\ref{sec:proof-maintheorem}.

%\begin{figure}
%  \centering
%  \begin{tabular}{c}
%  \begin{subfigure}{0.17\columnwidth}
%    \begin{overpic}[width=\columnwidth]{graph_a.pdf}
%      \put(27,9){$R_w$}
%    \end{overpic}
%    \subcaption{}
%    \label{fig:graph_a}
%  \end{subfigure}
%  \hspace{8mm}
%  \begin{subfigure}{0.17\columnwidth}
%    \begin{overpic}[width=\columnwidth]{graph_b.pdf}
%      \put(27,9){$R_w$}
%    \end{overpic}
%    \subcaption{}
%    \label{fig:graph_b}
%  \end{subfigure}
%  \hspace{8mm}
%  \begin{subfigure}{0.17\columnwidth}
%    \begin{overpic}[width=\columnwidth]{graph_c.pdf}
%      \put(37,19){$R_w$}
%    \end{overpic}
%    \subcaption{}
%    \label{fig:graph_c}
%  \end{subfigure}
%  \end{tabular}
%  \caption{The plane graph $\Graph$ dual to the black regions.
%  The complementary region of $\Graph$ labeled $R_w$
%  determines the desired white region $R_w$.
%  }
%  \label{fig:graph}
%\end{figure}

\begin{lemma}
\label{lem:link-diagram}
Let $\Diagram$ be a prime alternating diagram of a hyperbolic 
alternating link $L\subset S^3$.
Then the following hold.
\begin{enumerate}
\item[{\rm (1)}]
$\Diagram$ has at least three black regions.
\item[{\rm (2)}]
Suppose $\Diagram$ has precisely three black regions.
Then there is a white region $R_w$,
such that $R_w$ is a bigon
and the black regions adjacent to $R_w$ are distinct
(to be precise, the black regions that contain one of the two edges of $R_w$
are distinct). 
\item[{\rm (3)}]
Suppose $\Diagram$ has precisely four black regions.
Then there is a white region $R_w$,
such that either
(a) $R_w$ is a bigon and the black regions 
adjacent to $R_w$ are distinct,
or (b) $R_w$ is a $3$-gon and the black regions
adjacent to $R_w$ are all distinct.
\end{enumerate}
Parallel statements also hold when black and white are interchanged.
\end{lemma}

\begin{figure}
  \centering
  \begin{tabular}{c}
  \begin{subfigure}{0.17\columnwidth}
    \begin{overpic}[width=\columnwidth]{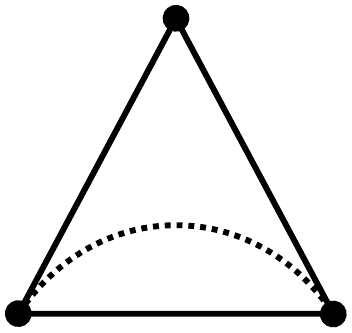}
      \put(27,9){$R_w$}
    \end{overpic}
    \subcaption{}
    \label{fig:graph_a}
  \end{subfigure}
  \hspace{8mm}
  \begin{subfigure}{0.17\columnwidth}
    \begin{overpic}[width=\columnwidth]{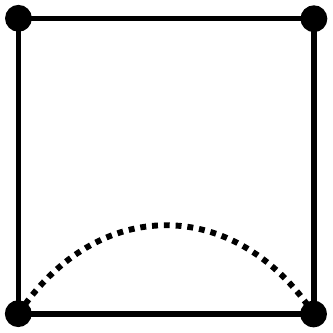}
      \put(27,9){$R_w$}
    \end{overpic}
    \subcaption{}
    \label{fig:graph_b}
  \end{subfigure}
  \hspace{8mm}
  \begin{subfigure}{0.17\columnwidth}
    \begin{overpic}[width=\columnwidth]{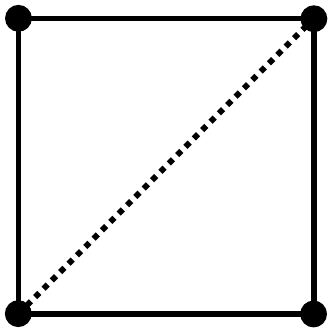}
      \put(37,19){$R_w$}
    \end{overpic}
    \subcaption{}
    \label{fig:graph_c}
  \end{subfigure}
  \end{tabular}
  \caption{The plane graph $\Graph$ dual to the black regions.
  The complementary region of $\Graph$ labeled $R_w$
  determines the desired white region $R_w$.
  }
  \label{fig:graph}
\end{figure}

\begin{proof}
Let $\Graph$ be the plane graph 
whose vertices are the black regions and 
whose edges correspond to the crossings.
Observe that $\Graph$ is connected and has no loop edge nor a cut edge,
because the diagram $\Diagram$ is connected and prime.

(1)
By using the above observation, we see that $\Diagram$ has at least two black regions.
If $\Diagram$ has only two black regions, then  
$L$ is the $(2,\pm n)$-torus link, where $n$ is the number of the edges of $\Graph$,
a contradiction.
Hence $\Diagram$ has at least three black regions.

(2)
Suppose $\Diagram$ has precisely three black regions.
Then, by using the observation above,
we see
that $\Graph$ has a $3$-cycle.
If $\Graph$ is equal to the $3$-cycle then $L$ is the $(2,\pm 3)$-torus knot, a contradiction.
Hence, there is an additional edge and so $\Graph$ has multiple edges.
Then we see that the white region, $R_w$, determined by an innermost pair
of multiple edges satisfies the desired condition
(see Figure~\ref{fig:graph_a}).

(3)
Suppose $\Diagram$ has precisely four black regions.
Then, as in (2),
we see that $\Graph$ has a $4$-cycle.
Since $L$ is hyperbolic, $\Graph$ is strictly bigger than the
$4$-cycle. 
Thus we see that there is a complementary region of $\Graph$
that is either a bigon or a triangle.
Then the white region, $R_w$, determined by 
a complementary bigon or triangle
satisfies the desired condition (a) or (b), accordingly
(see Figure~\ref{fig:graph} (b,c)).
\end{proof}

\section{Proof of Theorem~\ref{Theorem1} and~\ref{Theorem1-g}}
\label{sec:proof-maintheorem}

In this section, we first prove Theorem~\ref{Theorem1-g}
and then prove Theorem~\ref{Theorem1}.

\begin{proof}[Proof of Theorem~\ref{Theorem1-g}]
Let $L$, $\Diagram$, $\{\mu_1,\mu_2\}$ and $\gamma$ be as in the setting of 
the theorem, and 
let $\{p_1,p_2\}$ be the pair
of parabolic fixed points corresponding to $\{\mu_1,\mu_2\}$
(cf.~Lemma~\ref{lem:meridian-parabolic2}(1)).
We regard $\gamma$ living in the non-positively curved cubed complex 
$\Cubing\subset \XCubing$.
Then there is a lift $\tilde \gamma$ of $\gamma$
in the universal cover 
$\tCubing \subset \tXCubing$
which joins the peripheral planes
$\Bplane_{p_1}$ and $\Bplane_{p_2}$
centered at $p_1$ and $p_2$, respectively  
(cf.~Lemma~\ref{lem:meridian-parabolic2}(2)).
We may assume $\tilde\gamma$ satisfies the following conditions.
\begin{enumerate}
\item[(A1)]
$\tilde\gamma$ is an arc properly embedded in $\tCubing \subset \tXCubing$
that is disjoint from $\tilde\crossing=\tSb\cap\tSw$
and transversal to $\tSbw=\tSb\cup\tSw$.
Moreover, for $i=1,2$, the endpoint $x_i:= \partial \tilde\gamma\cap \Bplane_{p_i}$ 
is disjoint from the family of lines $\tSbw\cap \Bplane_{p_i}$
(see Figure~\ref{fig:Butterfly}(a)).
\item[(A2)]
The cardinality $\iota(\tilde\gamma)$ of 
$\tilde\gamma\cap \tSbw=\tilde\gamma\cap (\tSbw\setminus \tilde\crossing)$ is minimal among all arcs properly embedded
in $\tCubing$ joining the boundary components
$\Bplane_{p_1}$ and $\Bplane_{p_2}$ of $\tCubing$
and satisfying the condition (A1).
\end{enumerate}
We orient $\tilde\gamma$ so that
$x_1\in \Bplane_{p_1}$ and 
$x_2\in \Bplane_{p_2}$ are
the initial point and the terminal point, respectively.

In the remainder of the paper, we use the following terminology.
For a connected topological space $Y$ and its 
connected  subspaces $Y_1$, $Y_2$ and $Z$,
we say that $Z$ 
{\it separates} $Y_1$ and $Y_2$ (in $Y$),
if $Y_1$ and $Y_2$ are contained in distinct components of 
$Y\setminus Z$.
We say that $Z$ {\it weakly separates} 
$Y_1$ and $Y_2$ (in $Y$),
if $Y_1$ and $Y_2$ are contained in the closures of distinct components of 
$Y\setminus Z$.

\medskip

Case I. $\iota(\tilde\gamma)>0$.
Throughout the treatment of this case, 
geodesics are those with respect to the CAT(0) metric
of the cubed complex $\tXCubing$.  

\begin{lemma}
\label{lem:replacing}
Any checkerboard hyperplane intersects $\tilde\gamma$ in at most one point.
\end{lemma}

\begin{proof}
Assume that there is a checkerboard hyperplane $\Sigma$ 
which intersects $\tilde \gamma$
in more than one points.
Pick two successive intersection points $z_1$ and $z_2$ of
$\tilde\gamma$ with $\Sigma$,
and let $\tilde\gamma_{0}$ be the subarc of $\tilde\gamma$ 
bounded by $z_1$ and $z_2$.
Since $\Sigma$ is convex (Proposition~\ref{prop:cubing2X}),
the geodesic segment $[z_1,z_2]$ is
contained in $\Sigma$.

\begin{claim}
\label{claim:pair-intersection}
If a checkerboard hyperplane $\Sigma'$
different from $\Sigma$
intersects $[z_1,z_2]$, then 
(i) $\Sigma'\cap [z_1,z_2]$ consists of a single transversal intersection point in $(z_1,z_2)$
and (ii) $\Sigma'\cap \interior\tilde\gamma_{0}\ne \emptyset$.
\end{claim}

\begin{proof}
Let $\Sigma'\ne\Sigma$ be a checkerboard hyperplane 
which intersects $[z_1,z_2]$.
Then $\ell:=\Sigma\cap\Sigma' \supset [z_1,z_2]\cap\Sigma'\ne \emptyset$,
and so $\ell$ is a geodesic line 
(cf.~Proposition~\ref{pro:connected-intersection}(1))
which intersects $[z_1,z_2]$.
Since $z_1,z_2 \not\in \ell$ by the condition (A1), 
$\Sigma'\cap [z_1,z_2]=\ell\cap [z_1,z_2]$ is a singleton
$\{y\}$ for some $y\in (z_1,z_2)$
by Corollary~\ref{cor:no-branching}(1).
By Corollary~\ref{cor:no-branching}(2),
the two components of $[z_1,z_2]\setminus \{y\}$
is contained in distinct components of $\Sigma\setminus\ell$.
Hence the condition (i) holds.
This also implies that $\Sigma'$ separates 
the endpoints $z_1$ and $z_2$ of $\tilde\gamma_{0}$.
Hence (ii) also holds.
\end{proof}

Let $\tilde\gamma'$ be an arc obtained from $\tilde\gamma$ by replacing 
$\tilde\gamma_{0}$ with $[z_1,z_2]$
and then pushing (a neighborhood in the resulting arc of) $[z_1,z_2]$
off $\Sigma$,
by using a regular neighborhood of $\Sigma$.
Then $\tilde\gamma'$ is properly homotopic to $\tilde\gamma$,
and we may assume $\tilde\gamma'$ satisfies the condition (A1).
Moreover, Claim~\ref{claim:pair-intersection} implies that 
$\iota(\tilde\gamma')\le \iota(\tilde\gamma)-2$, a contradiction.
\end{proof}

\begin{figure}
  \centering
  \begin{overpic}[width=0.4\columnwidth]{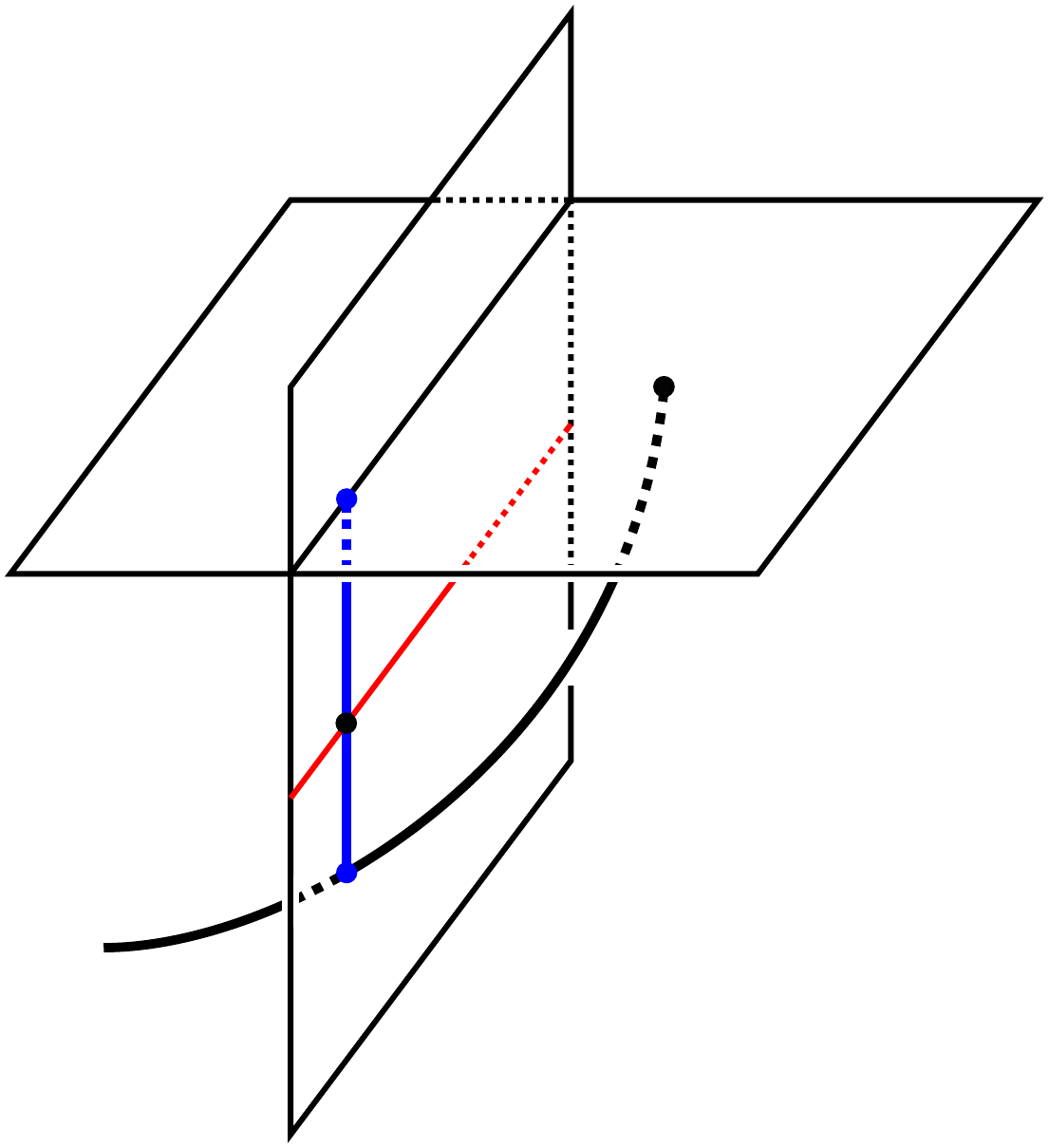}
    \put(7,28){$\tilde{\gamma}$}
    \put(-8,46){\textcolor{red}{$\ell = \Sigma \cap \Sigma^\prime$}}
    \put(56,59){$y$}
    \put(50,32){\textcolor{blue}{$z$}}
    \put(47,108){\small \textcolor{blue}{$x_1^\prime$}}
    \put(90,70){$\tilde{\gamma}_0$}
    \put(89,153){$\Sigma$}
    \put(104,113){$x_1$}
    \put(123,130){$H_{p_1}$}
  \end{overpic}
  \caption{If a checkerboard hyperplane $\Sigma^\prime \neq \Sigma$ intersects $[x_1^\prime,z]$, then it separates $H_{p_1}$ and $z$, and hence intersects $\tilde{\gamma}_0$.}
  \label{fig:near-boundary-plane}
\end{figure}
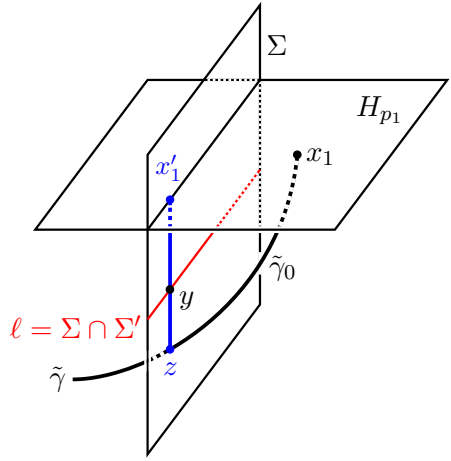

We now prove a key lemma for the treatment of Case 1.

\begin{lemma}
\label{lem:sliding}
Any checkerboard hyperplane which intersects $\tilde\gamma$ 
separates $\Bplane_{p_1}$ and $\Bplane_{p_2}$ in $\tXCubing$.
\end{lemma}

\begin{proof}
Let $\Sigma$ be a checkerboard hyperplane which intersects $\tilde\gamma$.
By Lemma~\ref{lem:replacing} and the condition (A1),
$\Sigma\cap \tilde\gamma$ consists of a single transversal intersection point
$z\in \interior\tilde\gamma$.
Thus we have only to show that $\Sigma$ is disjoint from
$\Bplane_{p_1}$ and $\Bplane_{p_2}$. 
Suppose to the contrary that $\Sigma$
intersects one of $\Bplane_{p_1}$ and $\Bplane_{p_2}$, say, $\Bplane_{p_1}$
(see Figure~\ref{fig:near-boundary-plane}).
Let $\tilde\gamma_0$ be the subarc of $\tilde\gamma$ bounded by the 
initial point $x_1\in \Bplane_{p_1}$ of $\tilde\gamma$ and the intersection point 
$z\in \Sigma\cap \tilde\gamma$.
Let $x_1'\in \Sigma\cap \Bplane_{p_1}$ be the projection, 
in the CAT(0) space $\Sigma$, of $z$ to the geodesic line 
$\Sigma\cap \Bplane_{p_1}$.
Since $\Sigma$ intersects $\Bplane_{p_1}$ orthogonally
(Proposition~\ref{pro:connected-intersection}(2)),
we see that the geodesic segment $[x_1',z]$ 
intersects $\Bplane_{p_1}$ orthogonally.
Thus $x_1'$ is the projection, 
in the CAT(0) space $\tXCubing$, of $z$ to $\Bplane_{p_1}$
by Lemma~\ref{lem:GR2b}.

\begin{claim}
\label{claim:pair-intersection2}
If a checkerboard hyperplane $\Sigma'$
different from $\Sigma$
intersects $[x_1',z]$, then 
(i) $\Sigma'\cap [x_1',z]$ consists of a single transversal intersection point in $(x_1',z)$
and (ii) $\Sigma'\cap \interior\tilde\gamma_{0}\ne \emptyset$.
\end{claim}

\begin{proof}
Let $\Sigma'\ne \Sigma$ be a checkerboard hyperplane 
which intersects $[x_1',z]$.
Then, as in the proof of Claim~\ref{claim:pair-intersection},
$\ell:=\Sigma\cap \Sigma'$ is a geodesic line
which intersects $[x_1',z]$.
Since $z\notin\ell\cap\tilde\gamma$ 
by the condition (A1),
$\ell\cap [x_1',z]$ is 
a singleton $\{y\}$ 
for some $y\in [x_1',z)$ 
by Corollary~\ref{cor:no-branching}(1).
If $y=x_1'$, then 
$\{x_1'\}=H_{p_1}\cap\Sigma\cap\Sigma'$
and so
$[x_1',z]\subset \pi_{H_{p_1}}^{-1}(x_1')=\ell$ 
by Corollary~\ref{cor:no-branching}(3),
a contradiction to the fact that 
$z\notin\ell\cap\tilde\gamma$.
Thus $\ell\cap [x_1',z]$ is a singleton $\{y\}$ 
for some $y\in (x_1',z)$.
So, by Corollary~\ref{cor:no-branching}(2), we obtain the conclusion (i).
This also implies that 
$x_1'$ and $z$ belong to distinct components of $\Sigma\setminus\ell$,
and hence $\Sigma'$ separates $x_1'\in \Bplane_{p_1}$ and $z$.
Moreover, $\Sigma'$ is disjoint from $\Bplane_{p_1}$ as shown below.
Suppose to the contrary that $\Sigma'\cap\Bplane_{p_1}\ne \emptyset$.
Then $\Sigma'$ intersects $\Bplane_{p_1}$ orthogonally
(Proposition~\ref{pro:connected-intersection}(2)),
and we see by the argument preceding Claim~\ref{claim:pair-intersection2}
that the projection, $y_1$, of $y$,
 in the CAT(0) space $\Sigma'$, to $\Sigma'\cap \Bplane_{p_1}$
is equal to the projection of $y$ in
the CAT(0) space $\tXCubing$ to $\Bplane_{p_1}$,
which is equal to $x_1'$.
Hence $x_1'=y_1$ belongs to $\Sigma'$,
and therefore $x_1'\in \Sigma'\cap\Sigma=\ell$, 
a contradiction to the fact that 
$[x_1',z]\cap \Sigma'=\{y\}\subset (x_1',z)$.
Hence $\Sigma'$ is disjoint from $\Bplane_{p_1}$ as desired.
Since $\Sigma'$ separates $x_1'\in \Bplane_{p_1}$ and $z$,
this implies that 
$\Sigma'$ separates $\Bplane_{p_1}$ and $z$.
Since $\tilde\gamma_{0}$ joins the point $x_1\in \Bplane_{p_1}$ and $z$,
$\tilde\gamma_{0}$ must intersect $\Sigma'$.
Thus the conclusion (ii) holds.
\end{proof}

Let $\tilde\gamma'$ be an arc obtained from $\tilde\gamma$ by replacing 
$\tilde\gamma_{0}$ with $[x_1',z]$
and then pushing (a neighborhood in the resulting arc of) $[x_1',z]$
off $\Sigma$.
Then $\tilde\gamma'$ is properly homotopic to $\tilde\gamma$,
and we may assume $\tilde\gamma'$ satisfies the condition (A1).
Moreover, Claim~\ref{claim:pair-intersection2} implies that 
$\iota(\tilde\gamma')\le \iota(\tilde\gamma)-1$, a contradiction.
\end{proof}

%\begin{figure}
%  \centering
%  \begin{overpic}[width=0.4\columnwidth]{near-boundary-plane.pdf}
%    \put(7,28){$\tilde{\gamma}$}
%    \put(-8,46){\textcolor{red}{$\ell = \Sigma \cap \Sigma^\prime$}}
%    \put(56,59){$y$}
%    \put(50,32){\textcolor{blue}{$z$}}
%    \put(47,108){\small \textcolor{blue}{$x_1^\prime$}}
%    \put(90,70){$\tilde{\gamma}_0$}
%    \put(89,153){$\Sigma$}
%    \put(104,113){$x_1$}
%    \put(123,130){$H_{p_1}$}
%  \end{overpic}
%  \caption{If a checkerboard hyperplane $\Sigma^\prime \neq \Sigma$ intersects $[x_1^\prime,z]$, then it separates $H_{p_1}$ and $z$, and hence intersects $\tilde{\gamma}_0$.}
%  \label{fig:near-boundary-plane}
%\end{figure}

Let $y_1$ be the first intersection point of $\tilde\gamma$ 
with $\tSbw$,
and $\tPoly_1$ the checkerboard ideal polyhedron 
that contains the subarc of $\tilde\gamma$
bounded by $x_1$ and $y_1$. 
Similarly, let $y_2$ be the last intersection point of 
$\tilde\gamma$ 
with $\tSbw$,
and $\tPoly_2$ the checkerboard ideal polyhedron 
that contains the subarc of $\tilde\gamma$
bounded by $x_2$ and $y_2$
(see Figure~\ref{fig:nonzero-intersection}(a)).
(If $\iota(\tilde\gamma)=1$
then $y_1=y_2$ but $\tPoly_1\ne \tPoly_2$.) 
For $i=1,2$,
recall the isomorphism
$\hPoly_i\cong (B^3,\Diagram)$, 
and 
let $c_i$ be the vertex of $\Diagram$ corresponding to 
the ideal vertex $p_i$ of $\hPoly_i$,
and let $R_i$ be the region of $\Diagram$ 
such that $\Sigma_{R_i}=\Sigma_{R_i}(\tPoly_i)$
contains $y_i$
(Definition and Notation~\ref{def:checkerboard-polyhedron}).
Note that the region $R_i$ does not contain the vertex $c_i$
by Lemma~\ref{lem:sliding}.

For simplicity, we assume that $R_1$ is a black region.
For $i=1,2$,
let $R_i^{\pm}$ be the black regions of $\Diagram$
that contain the crossing $c_i$, and
consider the disks $\Delta_{R_i^{\pm}}:=\Delta_{R_i^{\pm}}(\tPoly_i)$ in $\CCC$
(see Notation~\ref{notation:idealpolyhedron2}
and Figure~\ref{fig:disjoint-halfspace}(a)).
Then $\BF(p_i):=\{\Delta_{R_i^{-}}, \Delta_{R_i^{+}}\}$ forms a butterfly at $p_i$
by Lemma~\ref{lem:charaterizing-butterfly}
(after replacing $R_i^{\pm}$ with $R_i^{\mp}$ if necessary).
Set $\Sigma_{R_i^{\pm}}:=\Sigma_{R_i^{\pm}}(\tPoly_i)$ and
$\Half_{R_i^{\pm}}:=\Half_{R_i^{\pm}}(\tPoly_i)$
(see Figure~\ref{fig:nonzero-intersection}(a)).
Then we have the following lemma.

\begin{figure}
  \centering
  \begin{tabular}{c}
  \begin{subfigure}{0.35\columnwidth}
    \begin{overpic}[width=\columnwidth]{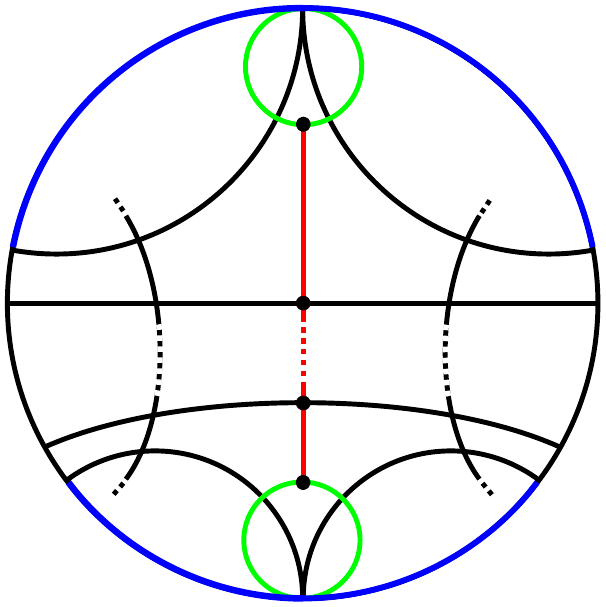}
      \put(-3,120){\textcolor{blue}{$\Delta_{R_1^-}$}}
      \put(119.5,120){\textcolor{blue}{$\Delta_{R_1^+}$}}
      \put(16,5){\textcolor{blue}{$\Delta_{R_2^-}$}}
      \put(109,5){\textcolor{blue}{$\Delta_{R_2^+}$}}
      \put(27,107){$\Half_{R_1^-}$}
      \put(94,107){$\Half_{R_1^+}$}
      \put(33,21.5){$\Half_{R_2^-}$}
      \put(88,21.5){$\Half_{R_2^+}$}
      \put(49,78){$\tPoly_1$}
      \put(49,35){$\tPoly_2$}
      \put(67,144){$p_1$}
      \put(67,-7){$p_2$}
      \put(66,114.5){\scriptsize$x_1$}
      \put(66,21.5){\footnotesize$x_2$}
      \put(72.5,75){$y_1$}
      \put(72.5,40){$y_2$}
      \put(72.5,88){\textcolor{red}{$\tilde{\gamma}$}}
      \put(140,67){\footnotesize $\Sigma_{R_1}$}
      \put(132,31){\footnotesize $\Sigma_{R_2}$}
    \end{overpic}
    \hspace{-0.1mm}
    \subcaption{}
    \label{fig:nonzero-intersection_a}
  \end{subfigure}
  \hspace{15mm}
  \begin{subfigure}{0.35\columnwidth}
    \begin{overpic}[width=\columnwidth]{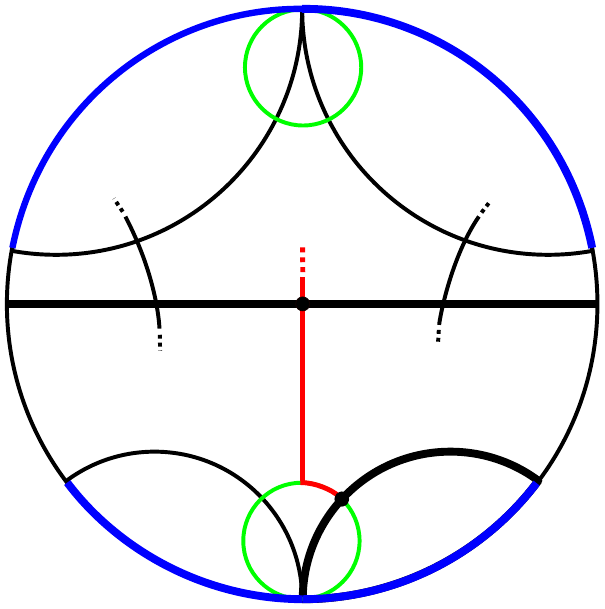}
      \put(-3,120){\textcolor{blue}{$\Delta_{R_1^-}$}}
      \put(119.5,120){\textcolor{blue}{$\Delta_{R_1^+}$}}
      \put(16,5){\textcolor{blue}{$\Delta_{R_2^-}$}}
      \put(109,5){\textcolor{blue}{$\Delta_{R_2^+}$}}
      \put(27,107){$\Half_{R_1^-}$}
      \put(94,107){$\Half_{R_1^+}$}
      \put(35,45){$\Half_{1,2}$}
      \put(98,23){$\Half_{2}$}
      \put(49,78){$\tPoly_1$}
      \put(67,144){$p_1$}
      \put(67,-7){$p_2$}
      \put(66,21.5){\footnotesize$x_2$}
      \put(72.5,75){$y_1$}
      \put(83,22){$z_2$}
      \put(72.5,45){\textcolor{red}{$\tilde{\gamma}_{1,2}$}}
      \put(140,67){\footnotesize $\Sigma_{R_1}$}
%      \put(132,31){$\Sigma_2$}
    \end{overpic}
    \hspace{-0.1mm}
    \subcaption{}
    \label{fig:nonzero-intersection_b}
  \end{subfigure}
  \end{tabular}
  \caption{(a) The checkerboard hyperplane $\Sigma_{R_1}(\tPoly_1)$ separates 
  $\Half_{R_1^-} \cup \Half_{R_1^+}$ and 
  $\Half_{R_2^-} \cup \Half_{R_2^+}$.
  Note that $\Half_{R_1}=\Half_{R_1}(\tPoly_1)$ is the region in
  $\tXCubing=\HH^3$ \lq\lq below'' $\Sigma_{R_1}$.
  (b) The arc $\tilde\gamma_{1,2}$,
  that is the union of the arc $\delta\subset \Bplane_{p_2}$ and 
  the subarc of $\tilde\gamma$ bounded by $y_1$ and $x_2$,
  intersects $\Sigma_{R_1}$ and $\Sigma_{R_2^{\epsilon}}$
  only at the endpoints. Here $\epsilon =+$.
  }
  \label{fig:nonzero-intersection}
\end{figure}

\begin{lemma} 
\label{lem:separating-butterfly} 
{\rm (1)} $\Half_{R_2^{-}}\cup \Half_{R_2^{+}}
\subset \Half_{R_1}$,
where $\Half_{R_1}=\Half_{R_1}(\tPoly_i)$

{\rm (2)} $\Half_{R_1^{-}}\cup \Half_{R_1^{+}}$ and 
$\Half_{R_2^{-}}\cup \Half_{R_2^{+}}$ are disjoint.

{\rm (3)} 
$\Delta_{R_2^{-}}\cup \Delta_{R_2^{+}}\subset \Delta_{R_1}$,
where $\Delta_{R_1}=\Delta_{R_1}(\tPoly_1)$.

{\rm (4)} $|\BF(p_1)|=\Delta_{R_1^{-}}\cup \Delta_{R_1^{+}}$ and $|\BF(p_2)|=\Delta_{R_2^{-}}\cup \Delta_{R_2^{+}}$ have disjoint interiors.
\end{lemma}

\begin{proof}
(1)
For each $\epsilon\in \{-,+\}$,
$\Sigma_{R_1}$ is distinct
from $\Sigma_{R_2^{\epsilon}}$, because
$\Sigma_{R_2^{\epsilon}}\cap\Bplane_{p_2}\ne \emptyset$
whereas 
$\Sigma_{R_1}\cap\Bplane_{p_2}=\emptyset$ by Lemma~\ref{lem:sliding}.
This implies that $\Sigma_{R_1}$ is disjoint from $\Sigma_{R_2^{\epsilon}}$
(because they are distinct components of $\pu^{-1}(\Sb)$).
By Lemma~\ref{lemma:three-components},
the disjoint union $\Sigma_{R_1} \sqcup \Sigma_{R_2^{\epsilon}}$
divides $\tXCubing$ into 
three closed convex subspaces 
$\Half_1$, $\Half_{1,2}$ and $\Half_2$,
such that 
$\Half_1\cap \Half_{1,2}=\Sigma_{R_1}$, 
$\Half_{1,2}\cap \Half_2=\Sigma_{R_2^{\epsilon}}$ and  
$\Half_1\cap \Half_2=\emptyset$.
Let $\delta$ be an arc in the square 
$\Bplane_{p_2}\cap \tPoly_2$ which joins $x_2$ with a point $z_2$ in 
$\Bplane_{p_2}\cap \Sigma_{R_2^{\epsilon}}$
(cf.~Figure~\ref{fig:Butterfly}(a)),
and let $\tilde\gamma_{1,2}$ be the union of $\delta$
and the subarc of $\tilde\gamma$ bounded by $y_1$ and $x_2$
(see Figure~\ref{fig:nonzero-intersection}(b), where $\epsilon$ 
is assumed to be $+$).
Then $\tilde\gamma_{1,2}$ is an arc in $\tXCubing$ 
joining $y_1$ and $z_2$,
such that $\tilde\gamma_{1,2}\cap \Sigma_{R_1}=\{y_1\}$
and $\tilde\gamma_{1,2}\cap \Sigma_{R_2^{\epsilon}}=\{z_2\}$.
Hence $\tilde\gamma_{1,2}$ is contained in $\Half_{1,2}$.
This implies $\tPoly_2\subset \Half_{1,2}$,
because $\interior\tPoly_2\cap(\Sigma_{R_1}\cup\Sigma_{R_2^{\epsilon}})
=\emptyset$ and 
$\interior\tPoly_2 \cap \interior\tilde\gamma_{1,2}\ne\emptyset$.
Hence we have $\Half_2=\Half_{R_2^{\epsilon}}$.
On the other hand, we have $\tPoly_1\subset \Half_{1}$,
because $\tilde\gamma_{1,2}\subset \Half_{1,2}$ and
$\tilde\gamma$ intersects $\Sigma_{R_1}$ transversely at $y_1$;
so
$\Half_1=\Half_{R_1}^{\mathfrak{c}}:=\Half_{R_1}^{\mathfrak{c}}(\tPoly_1)$.
Hence
$\Half_{R_1}^{\mathfrak{c}}\cap \Half_{R_2^{\epsilon}}=\Half_1\cap \Half_2=\emptyset$,
and therefore $\Half_{R_2^{\epsilon}}\subset \Half_{R_1}$.

(2)
Since the black region $R_1$ does not contain $c_1$, 
it is distinct from the black regions $R_1^{\pm}$.
Hence
$\Half_{R_1^{\pm}}$  are disjoint from 
$\Half_{R_1}$ 
by Proposition~\ref{prop:disjoint-half-space}.
Since $\Half_{R_2^{-}}\cup \Half_{R_2^{+}}
\subset \Half_{R_1}$ by (1),
this implies that $\Half_{R_1^{\pm}}$ are disjoint from
$\Half_{R_2^{-}}\cup \Half_{R_2^{+}}$.

(3) By (1), we have
$\Delta_{R_2^{-}}\cup \Delta_{R_2^{+}}
=(\bHalf_{R_2^{-}}\cap \CCC)\cup (\bHalf_{R_2^{+}}\cap \CCC)
\subset 
\bHalf_{R_1}\cap \CCC
= \Delta_{R_1}$.

(4) This follows from (2) and 
Lemma~\ref{lemma:Ideal-Halfspace}.
\end{proof}

\begin{lemma}
\label{lem:non-empty-openset}
The open set $O:=\CCC\setminus (|\BF(p_1)|\cup |\BF(p_2)|)$
is non-empty.
\end{lemma}

\begin{proof}
Suppose first that $\Diagram$ has more than $3$ black regions.
Pick a black region $R_b$ of $\Diagram$ different from
$R_1$ and $R_1^{\pm}$.
Then the interior of the disk $\Delta_{R_b}:=\Delta_{R_b}(\tPoly_1)$ 
is disjoint from the disks $\Delta_{R_1}$ and $\Delta_{R_1^{\pm}}$
by Proposition~\ref{prop:disjoint-open-disk}.
Since 
$\Delta_{R_2^{-}}\cup \Delta_{R_2^{+}}\subset \Delta_{R_1}$
by Lemma~\ref{lem:separating-butterfly}(3),
this implies that the open disk $\interior\Delta_{R_b}$
is disjoint from 
$\Delta_{R_1}\cup \Delta_{R_1^{-}}\cup \Delta_{R_1^{+}}
\supset \Delta_{R_1^{-}}\cup \Delta_{R_1^{+}}\cup \Delta_{R_2^{-}}\cup \Delta_{R_2^{+}}
=|\BF(p_1)|\cup |\BF(p_2)|$
(see Figure~\ref{fig:gap-region_a}).
Hence $\interior\Delta_{R_b}\subset O$ and therefore
$O$ is non-empty, as desired.

Suppose next that $\Diagram$ has at most $3$ black regions.
Then, by Lemma~\ref{lem:link-diagram}(2),
$\Diagram$ has precisely three black regions, 
$\{R_{b,j}\}_{1\le j\le 3}=\{R_1, R_1^{-}, R_1^{+}\}$
and a white bigon $R_w$.
We may assume $R_{b,1}$ is not adjacent to $R_w$.
Let $\tPoly_1'$ be the checkerboard ideal polyhedron
such that $\tPoly_1\cap\tPoly_1'$ is the common face 
corresponding to $R_w$. 
Set $\Delta_{R_{b,j}}:=\Delta_{R_{b,j}}(\tPoly_1)$ and
$\Delta_{R_{b,j}}':=\Delta_{R_{b,j}}(\tPoly_1')$
($1\le j\le 3$).
(See Figure~\ref{fig:gap-region_b}.)

\begin{figure}
  \centering
  \begin{tabular}{c}
    \begin{subfigure}{0.35\columnwidth}
      \begin{overpic}[width=\columnwidth]{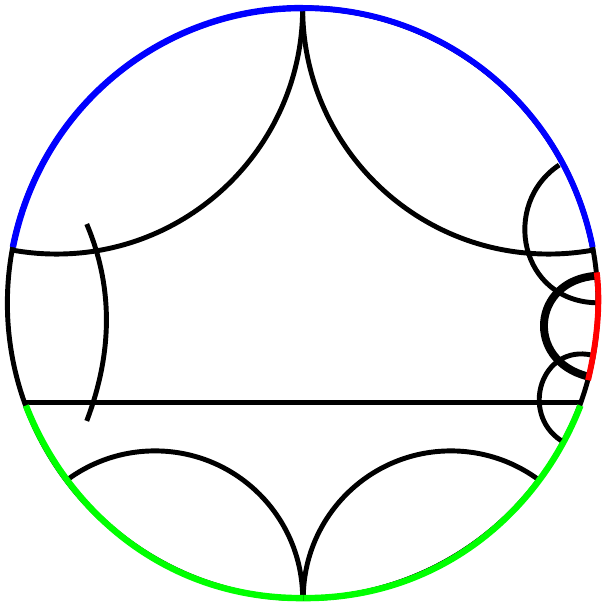}
        \put(-8,43.5){$\Sigma_1$}
        \put(-3,120){\textcolor{blue}{$\Delta_{R_1^-}$}}
        \put(119.5,120){\textcolor{blue}{$\Delta_{R_1^+}$}}
        \put(65,70){$\tPoly_1$}
        \put(139,61){\textcolor{red}{$\Delta_{R_{b}}$}}
        \put(29,0){\textcolor{green}{$\Delta_{R_1}$}}
        \put(67,144){$p_1$}
        \put(67,-7){$p_2$}
      \end{overpic}
      \hspace{-0.1mm}
      \subcaption{}
      \label{fig:gap-region_a}
    \end{subfigure}
    \hspace{12mm}
    \begin{subfigure}{0.35\columnwidth}
      \begin{overpic}[width=\columnwidth]{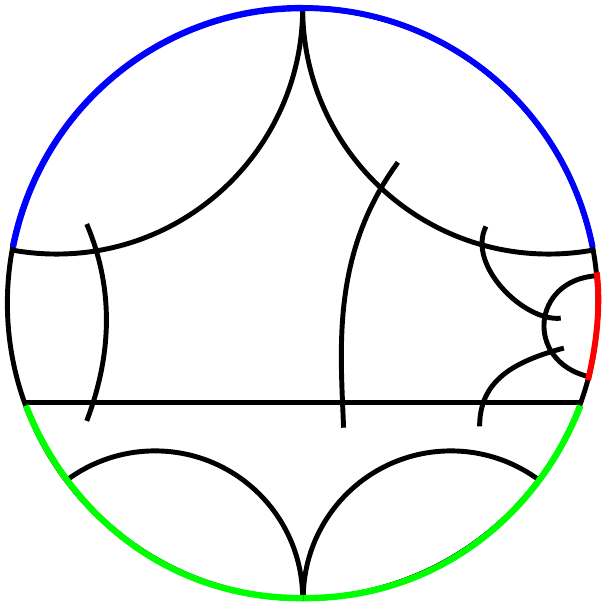}
        \put(-8,43.5){$\Sigma_1$}
        \put(-3,120){\textcolor{blue}{$\Delta_{R_1^-}$}}
        \put(119.5,120){\textcolor{blue}{$\Delta_{R_1^+}$}}
        \put(52,66){$\tPoly_1$}
        \put(92,63){$\tPoly^\prime_1$}
        \put(92,104){$\Sigma_{R_w}$}
        \put(139,61){\textcolor{red}{$\Delta^\prime_{R_{b,1}}$}}
        \put(29,0){\textcolor{green}{$\Delta_{R_1}$}}
        \put(67,144){$p_1$}
        \put(67,-7){$p_2$}
      \end{overpic}
      \hspace{-0.1mm}
      \subcaption{}
      \label{fig:gap-region_b}
    \end{subfigure}
  \end{tabular}
  \caption{
  (a) If $\Diagram$ has more than $3$ black regions,
  then, for a black region $R_b$ distinct from $R_1$ and $R_1^{\pm}$,
  the open disk $\interior\Delta_{R_b}$ is disjoint from $|\BF(p_1)|\cup |\BF(p_2)|$.
  (b) If $\Diagram$ has precisely $3$ black regions,
  then, for the black region $R_{b,1}$ that is not adjacent to 
  the white bigon $R_w$, the open disk 
  $\interior\Delta_{R_{b,1}}'$, where $\Delta_{R_{b,1}}'=\Delta_{R_{b,1}}(\tPoly_1')$,  
  is disjoint from
  $|\BF(p_1)|\cup |\BF(p_2)|$.
  }
  \label{fig:gap-region}
\end{figure}
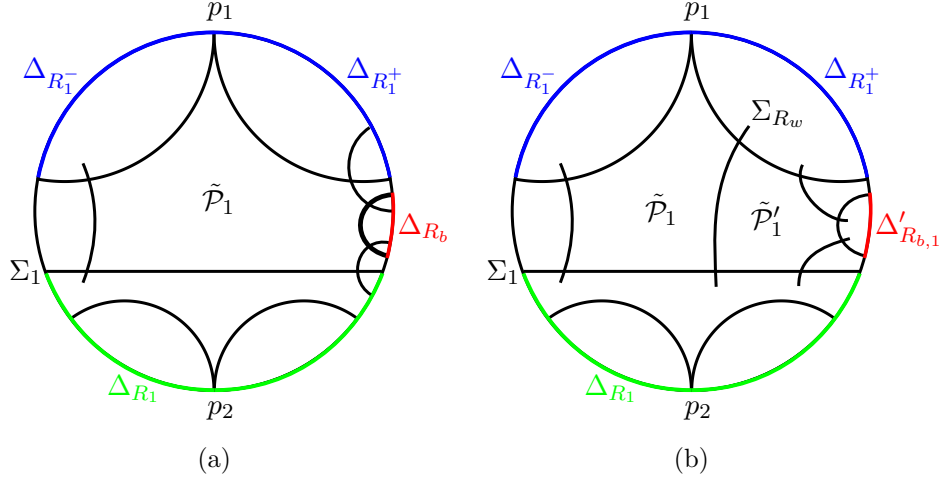

\begin{claim}
\label{claim:disjoint-disk}
The interior of the disk $\Delta_{R_{b,1}}'$
is disjoint from $\cup_{j=1}^3 \Delta_{R_{b,j}}=\Delta_{R_1}\cup \Delta_{R_1^{-}}\cup \Delta_{R_1^{+}}$.
\end{claim}
\begin{proof}
By Lemma~\ref{lem:common-plane}
and by Notation~\ref{notation:idealpolyhedron2}(3),
we see
$\Delta_{R_{b,2}}'=\Delta_{R_{b,3}}$
and 
$\Delta_{R_{b,3}}'=\Delta_{R_{b,2}}$.
Hence, by Proposition~\ref{prop:disjoint-open-disk},
$\interior\Delta_{R_{b,1}}'$ is
disjoint from $\Delta_{R_{b,3}}'\cup \Delta_{R_{b,2}}'=\Delta_{R_{b,2}}\cup \Delta_{R_{b,3}}$.
Moreover, $\interior\Delta_{R_{b,1}}'$
is also disjoint from $\Delta_{R_{b,1}}$, as explained below. 
Since $R_{b,1}$ and $R_w$ are not adjacent, 
Proposition~\ref{prop:disjoint-half-space} implies that
$\Half_{R_{b,1}}(\tPoly_1) \cap \Half_{R_w}(\tPoly_1)=\emptyset$.
Hence 
$\Half_{R_{b,1}}(\tPoly_1)\subset 
\tXCubing \setminus \Half_{R_w}(\tPoly_1)
=\interior\Half_{R_w}(\tPoly_1')$.
Similarly, $\Half_{R_{b,1}}(\tPoly_1')\subset \interior\Half_{R_w}(\tPoly_1)$.
Since $\interior\Half_{R_w}(\tPoly_1)$ and 
$\interior\Half_{R_w}(\tPoly_1')=\interior\Half_{R_w}^{\mathfrak{c}}(\tPoly_1)$
are disjoint, $\Half_{R_{b,1}}(\tPoly_1)$ and $\Half_{R_{b,1}}(\tPoly_1')$
are disjoint.
By Lemma~\ref{lemma:Ideal-Halfspace},
this implies that $\Delta_{R_{b,1}}$ and $\Delta_{R_{b,1}}'$ have disjoint
interiors,
and hence $\interior\Delta_{R_{b,1}}'$
is disjoint from $\Delta_{R_{b,1}}$.
\end{proof}

Since 
$\Delta_{R_2^{-}}\cup \Delta_{R_2^{+}}\subset \Delta_{R_1}$
by Lemma~\ref{lem:separating-butterfly}(3),
Claim~\ref{claim:disjoint-disk} implies that
the open disk $\interior\Delta_{R_{b,1}}'$
is disjoint from $\Delta_{R_1^{-}}\cup \Delta_{R_1^{+}}\cup \Delta_{R_2^{-}}\cup \Delta_{R_2^{+}} = |\BF(p_1)|\cup |\BF(p_2)|$.
Hence $\interior\Delta_{R_{b,1}}'\subset O$ and therefore
$O$ is non-empty, as desired. 
\end{proof}

Thus we have proved that
the pair of butterflies $\BF(p_1)$ and $\BF(p_2)$
satisfies the conditions in 
Proposition~\ref{prop:ping-pong}.
Hence $\{\mu_1,\mu_2\}$ 
generates a rank $2$ free Kleinian group which is geometrically finite.
This completes the proof of Theorem~\ref{Theorem1-g}
in Case I where $\iota(\tilde\gamma)>0$.

\medskip

Case II. $\iota(\tilde\gamma)=0$.
In this case, the proper arc
$\tilde\gamma\subset \tCubing$ is contained in 
$\tPoly\cap \tCubing$ 
for some ideal checkerboard polyhedron $\tPoly$.
Recall the isomorphism $\hat\varphi:(B^3,\Diagram) \to \hPoly$,
where $\hPoly$ is the closure of $\tPoly$ in $\HHH^3$
(Section~\ref{sec:Butterflies-checkerboard}).
We identify $\hPoly$ with $(B^3,\Diagram)$
through the isomorphism.
Let $c_i$ be the vertex of $\Diagram$ corresponding to the ideal vertex 
$p_i$ of $\hPoly$ ($i=1,2$).
Then the equivalence class of the meridian pair $\{\mu_1,\mu_2\}$
is determined by the pair $\{c_1, c_2\}$.
Let $\hat\gamma$ be an arc in $\partial B^3$ joining $c_1$ and $c_2$,
such that $\hat\gamma$ intersects the vertex set of $\Diagram$
only at their endpoints and that $\interior\hat\gamma$ is transversal to $\Diagram$.
Then the proper homotopy class of $\tilde\gamma$ 
is represented by $\hat\gamma$.
We assume that the cardinality $\omega(\hat\gamma)$
of $\interior\hat\gamma \cap \Diagram$ 
is minimized. 

\smallskip

Subcase II-1.
$\omega(\hat\gamma)>0$.
For $i=1,2$, let $R_{i}^{-}$ and $R_{i}^{+}$ be the black regions
that contain the vertex $c_i$.
Then $\{\Delta_{R_{i}^{-}}, \Delta_{R_{i}^{+}}\}$ forms a butterfly $\BF(p_i)$ at $p_i$
by Lemma~\ref{lem:charaterizing-butterfly}
(see Figure~\ref{fig:omega-positive}).

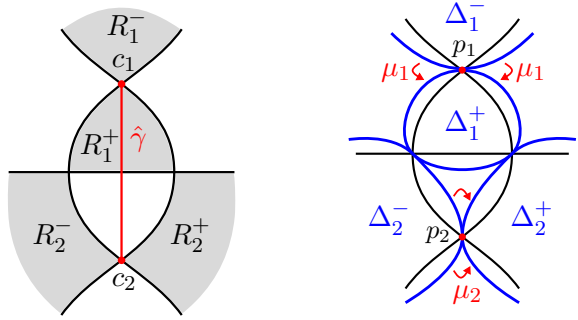
\begin{figure}
  \centering
  \begin{tabular}{c}
    \begin{subfigure}{0.22\columnwidth}
      \begin{overpic}[width=\columnwidth]{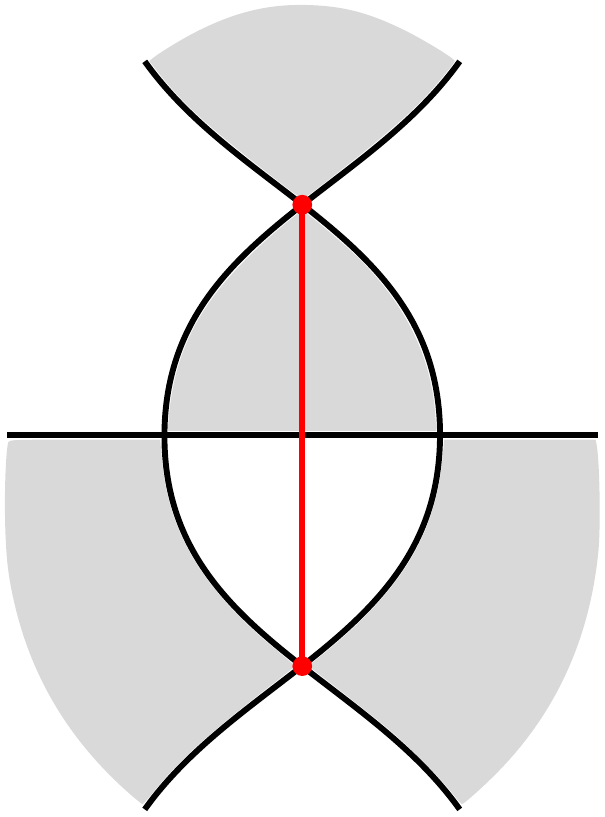}
        \put(28,62){$R_1^+$}
        \put(38,107){$R_1^-$}
        \put(62,28){$R_2^+$}
        \put(10,28){$R_2^-$}
        \put(40,95){$c_1$}
        \put(40,12){$c_2$}
        \put(47,65){\textcolor{red}{$\hat\gamma$}}
      \end{overpic}
      \label{fig:omega-positive_a}
    \end{subfigure}
    \hspace{12mm}
    \begin{subfigure}{0.22\columnwidth}
      \begin{overpic}[width=\columnwidth]{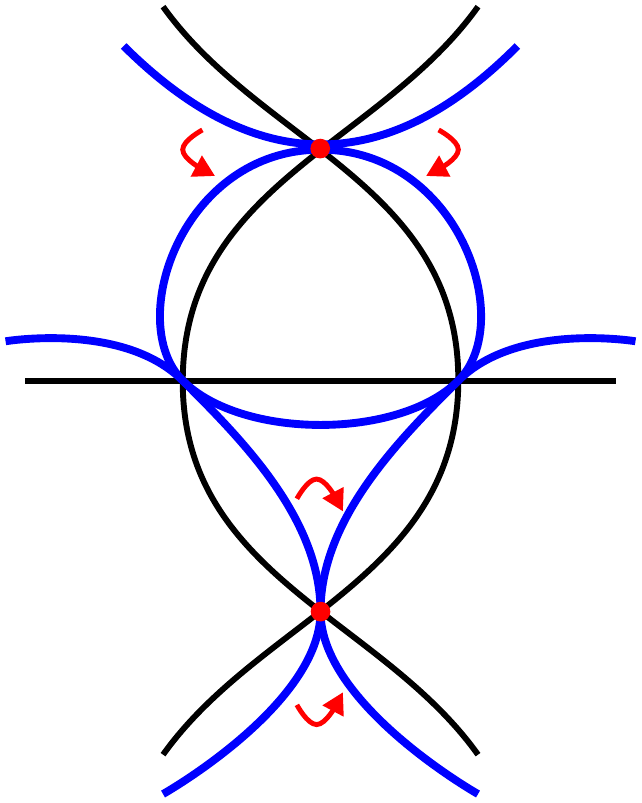}
        \put(37,66){\textcolor{blue}{$\Delta_1^+$}}
        \put(37,107){\textcolor{blue}{$\Delta_1^-$}}
        \put(62,30){\textcolor{blue}{$\Delta_2^+$}}
        \put(8,30){\textcolor{blue}{$\Delta_2^-$}}
        \put(13,87){\textcolor{red}{$\mu_1$}}
        \put(64,87){\textcolor{red}{$\mu_1$}}
        \put(40,3){\textcolor{red}{$\mu_2$}}
        \put(40,96){\small$p_1$}
        \put(30,25.5){\small$p_2$}
      \end{overpic}
      \label{fig:omega-positive_b}
    \end{subfigure}
  \end{tabular}
  \caption{The butterflies $\BF(p_1)$ and $\BF(p_2)$
  and the actions of the meridians $\mu_1$ and $\mu_2$,
  in the case $\omega(\hat\gamma)>0$.
  Here, we employ the model picture of the limit circles
  described in Remark~\ref{rem:disjoint-open-disk} and Figure~\ref{fig:intersection-pattern}.
  }
  \label{fig:omega-positive}
\end{figure}

\begin{claim}
\label{claim:distict-region}
The four black regions $R_1^-$, $R_1^+$, $R_2^-$ and $R_2^+$
are distinct.
\end{claim} 

\begin{proof}
Suppose to the contrary that there is an overlap among the four regions.
Since $R_i^-\ne R_i^+$ for $i=1,2$,
we have $R_1^{\epsilon_1}=R_2^{\epsilon_2}$ for some $\epsilon_1, \epsilon_2
\in \{-, +\}$.
Then the vertices $c_1$ and $c_2$ are contained in the single region
$R_1^{\epsilon_1}=R_2^{\epsilon_2}$. 
Thus the two vertices are joined by an arc in the region,
a contradiction to the assumption $\omega(\hat\gamma)>0$.
\end{proof}

By Claim~\ref{claim:distict-region} and 
Proposition~\ref{prop:disjoint-open-disk},
the butterflies $\BF(p_1)$ and $\BF(p_2)$ have disjoint interiors.
Moreover, the following lemma holds.

\begin{lemma}
\label{lem:non-empty-openset2}
The open set $O:=\CCC\setminus (|\BF(p_1)|\cup |\BF(p_2)|)$
is non-empty.
\end{lemma}

\begin{proof}
The proof of this lemma is parallel to that of 
Lemma~\ref{lem:non-empty-openset}.
If $\Diagram$ has more than four black regions, 
then 
a black region $R_b$ different from $R_1^{\pm}$ and $R_2^{\pm}$
gives a non-empty open disk $\interior \Delta_{R_b}$
disjoint from $|\BF(p_1)|\cup |\BF(p_2)|$
by Proposition~\ref{prop:disjoint-open-disk}.
So, we may assume $\Diagram$ has precisely four black regions.
Then, by Lemma~\ref{lem:link-diagram}(3),
there is a white region $R_w$ which is either a bigon or a $3$-gon.
In either case, 
there is a black region, say $R_{b,1}$,
that is not adjacent to $R_w$.
Let $\tPoly_1'$ be the checkerboard ideal polyhedron
such that $\tPoly_1\cap\tPoly_1'$ is the common face 
corresponding to $R_w$. 
Then, as in the proof of Claim~\ref{claim:disjoint-disk},
we see that the open disk $\interior\Delta_{R_{b,1}}(\tPoly_1')$
is disjoint from $|\BF(p_1)|\cup |\BF(p_2)|$.
\end{proof}

Thus 
the pair of butterflies $\BF(p_1)$ and $\BF(p_2)$
satisfies the conditions in 
Proposition~\ref{prop:ping-pong}.
Hence $\{\mu_1,\mu_2\}$ 
generates a rank $2$ free Kleinian group which is geometrically finite.
This completes the proof of Theorem~\ref{Theorem1-g}
in the case where $\iota(\tilde\gamma)=0$ and $\omega(\hat\gamma)>0$.

\medskip

Subcase II-2.
$\omega(\hat\gamma)=0$.
In this case, there is a region $R$ of $\Diagram$
that contains $\hat\gamma$ and the vertices $c_1$ and $c_2$.
Recall that $\gamma$ is not properly homotopic to a crossing arc 
by the assumption of the theorem.
This implies that 
$\hat\gamma$ is not homotopic relative to the endpoints to an edge of $R$
(i.e., the vertices $c_1$ and $c_2$ are not adjacent in $\partial R$),
because, for any edge $e$ of $\Diagram$, the composition
$\pu\circ\varphi:\PPoly(\Diagram) \to \pu(\tPoly)\subset X$
maps the ideal edge $\check e$ to an open crossing arc
(cf.~Proposition~\ref{prop:checkerboard-polyhedron}(3) and 
Figure~\ref{fig:CB1}).

For simplicity, assume that $R$ is a white region.
For $i=1,2$, let $R_{i}^{\pm}$ be the black regions
that contain the crossing $c_i$.
Then $\{\Delta_{R_{i}^{-}}, \Delta_{R_{i}^{+}}\}$ forms a butterfly $\BF(p_i)$ at $p_i$
by Lemma~\ref{lem:charaterizing-butterfly}
(see Figure~\ref{fig:omega-zero-noncrossing}).

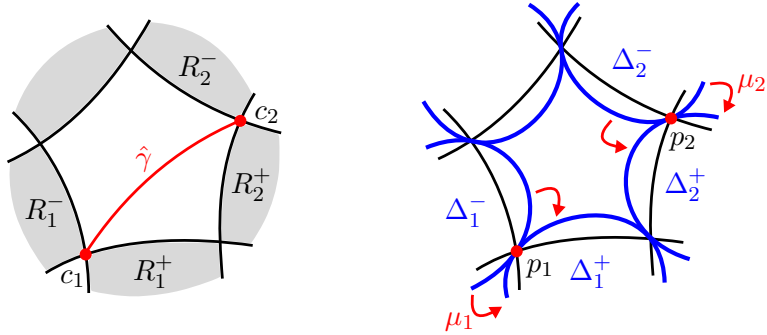
\begin{figure}
  \centering
  \begin{overpic}[width=0.7\columnwidth]{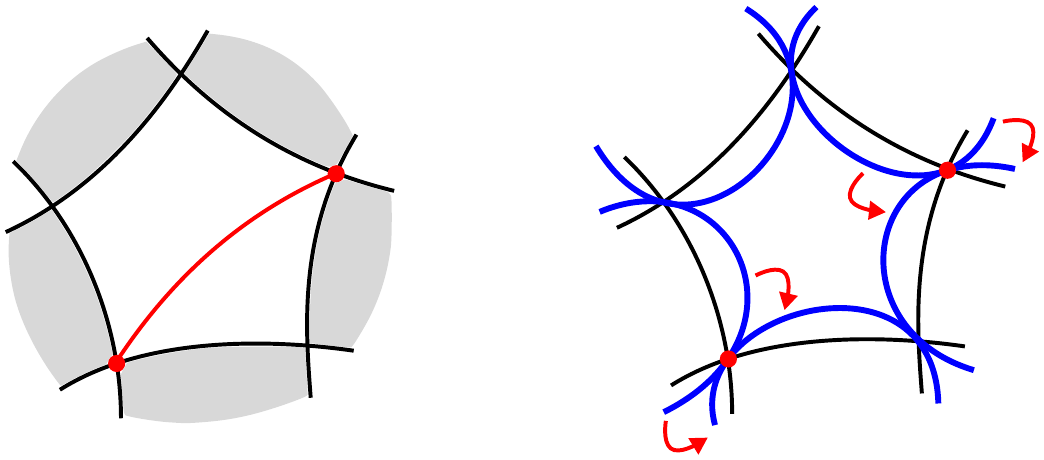}
    \put(21.5,15){$c_1$}
    \put(49,16){$R_1^+$}
    \put(8,39){$R_1^-$}
    \put(95,77){$c_2$}
    \put(86,49){$R_2^+$}
    \put(66,93){$R_2^-$}
    \put(50,59){\textcolor{red}{$\hat\gamma$}}
    \put(167,-1){\textcolor{red}{$\mu_1$}}
    \put(277.5,89){\textcolor{red}{$\mu_2$}}
    \put(214,16){\textcolor{blue}{$\Delta_1^+$}}
    \put(167,39){\textcolor{blue}{$\Delta_1^-$}}
    \put(250,49){\textcolor{blue}{$\Delta_2^+$}}
    \put(230,96){\textcolor{blue}{$\Delta_2^-$}}
    \put(198,19){$p_1$}
    \put(252,67){$p_2$}
  \end{overpic}
  \caption{The butterflies $\BF(p_1)$ and $\BF(p_2)$
  and the actions of the meridians $\mu_1$ and $\mu_2$,
  in the case $\omega(\hat\gamma)=0$ and $\gamma$ is not a crossing arc.
  Here, we employ the model picture of the limit circles
  described in Remark~\ref{rem:disjoint-open-disk} and Figure~\ref{fig:intersection-pattern}.}
  \label{fig:omega-zero-noncrossing}
\end{figure}

\begin{claim}
\label{claim:distict-region2}
The four black regions $R_1^-$, $R_1^+$, $R_2^-$ and $R_2^+$
are distinct.
\end{claim} 

\begin{proof}
Suppose to the contrary that there is an overlap among the $4$ regions.
Then as in the proof of Claim~\ref{claim:distict-region},
we have $R_1^{\epsilon_1}=R_2^{\epsilon_2}$ for some $\epsilon_1, \epsilon_2
\in \{-, +\}$.
Since $c_1$ and $c_2$ are not adjacent in $\partial R$,
the edges $e_i:=R\cap R_i^{\epsilon_i}$ ($i=1,2$) are distinct.
Thus we can find a simple loop $C$ in 
the union of the white region $R$ and 
the black region $R_1^{\epsilon_1}=R_2^{\epsilon_2}$,
which intersects $\Diagram$ transversely
in precisely two points, 
one in  $\interior e_1$
and the other in $\interior e_2$.
Since $e_1\ne e_2$, both disks bounded by $C$
contains a vertex of $\Diagram$. 
This contradicts the primeness of the diagram $\Diagram$.
\end{proof}

The proof of Lemma~\ref{lem:non-empty-openset2}
works in the current setting,
and so, we see that the open set
$O=\CCC\setminus(|\BF(p_1)|\cup |\BF(p_2)|)$
is non-empty.
Thus the pair of the butterflies $\BF(p_1)$
and $\BF(p_2)$
satisfies the conditions in Proposition~\ref{prop:ping-pong}.
Hence $\{\mu_1,\mu_2\}$ 
generates a rank $2$ free Kleinian group which is geometrically finite.

This completes the proof of Theorem~\ref{Theorem1-g}.
\end{proof}

\begin{proof}[Proof of  Theorem~\ref{Theorem1}]
Let $L\subset S^3$ be a hyperbolic $2$-bridge link, 
$\gamma$ an essential proper path in the link exterior $M$,
$\{\mu_1,\mu_2\}$ a non-commuting meridian pair in 
the link group $G$ represented by $\gamma$,
and $\{p_1,p_2\}$ the corresponding pair
of parabolic fixed points.
Assume that $\gamma$ is not 
properly homotopic to the upper or lower tunnel of $L$.
We show that $\{\mu_1,\mu_2\}$ 
generates a rank $2$ free Kleinian group which is geometrically finite.

If necessary by taking the mirror image of $L$,
we may assume that
$L$ admits the 
prime alternating diagram
$\Diagram$ in Figure~\ref{fig:2bridge-diagram},
where $(a_1, a_2, \dots, a_n)$ is a sequence of
positive integers with $n \ge 2$, $a_1 \ge 2$ and $a_n \ge 2$.
$\Diagram$ consists of $n$ twist regions $A_1, A_2, \cdots, A_n$,
where $A_i$ consists of $a_i$ right-hand or left-hand half-twists
according to whether $i$ is odd or even.
By Theorem~\ref{Theorem1-g}, we have only to treat the case where
$\gamma$ is a crossing arc with respect to 
the diagram $\Diagram$.
Let $A_i$ be the twist region that contains the crossing
corresponding to the crossing arc $\gamma$.
If $i=1$ or $n$, then $\gamma$ is isotopic to the upper or lower tunnel accordingly.
So, $2 \le i \le n-1$.

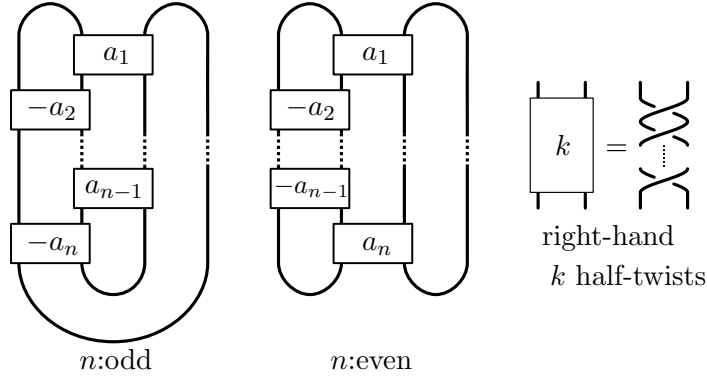
\begin{figure}
  \centering
  \begin{overpic}[width=0.65\columnwidth]{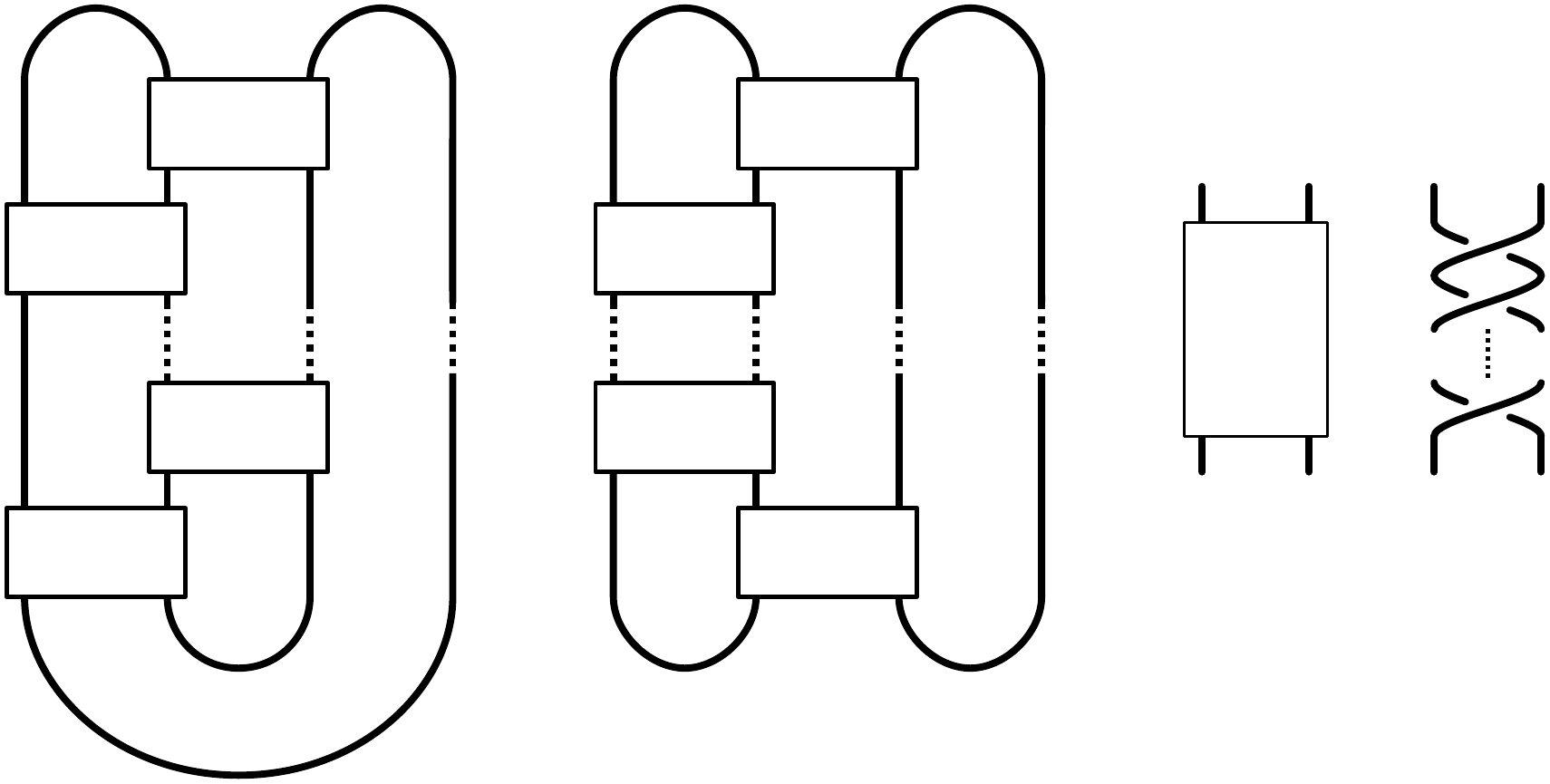}
    \put(36,108){$a_1$}
    \put(7,87){$-a_2$}
    \put(29,57.5){$a_{n-1}$}
    \put(7,36){$-a_n$}

    \put(134,108){$a_1$}
    \put(105,87){$-a_2$}
    \put(100.5,57.5){\small$-a_{n-1}$}
    \put(134,36){$a_n$}

    \put(27,-9){$n$:odd}
    \put(122,-9){$n$:even}
    \put(207,72){$k$}
    \put(226,72){$=$}
    \put(202,38){right-hand}
    \put(205,23){$k$ half-twists}
  \end{overpic}
  \vspace{2mm}
  \caption{The standard prime alternating diagram $\Diagram$ of a hyperbolic $2$-bridge link $L$}
  \label{fig:2bridge-diagram}
\end{figure}

Suppose $i$ is odd.
Apply the flype to $\Diagram$
as illustrated in Figure~\ref{fig:flype1},
and let $\Diagram'$ be the resulting prime alternating diagram.
Then the image of 
the crossing arc $\gamma$ 
by the flype is an arc $\gamma'$ contained in a region $R'$ of $\Diagram'$,
such that 
the corresponding arc $\hat\gamma'$ in the polyhedron 
$(B^3,\Diagram')$
joins crossings $c_1'$ and $c_2'$ of $R'$ which are not adjacent 
in $\partial R'$.
Hence, we can apply the arguments in Subcase II-2
in the proof of Theorem~\ref{Theorem1-g}
to show that
$\{ \mu_1, \mu_2 \}$ generates a rank $2$ free group which is geometrically finite.

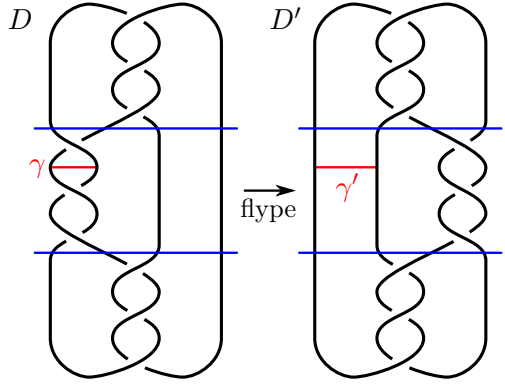
\begin{figure}
  \centering
  \begin{overpic}[width=0.45\columnwidth]{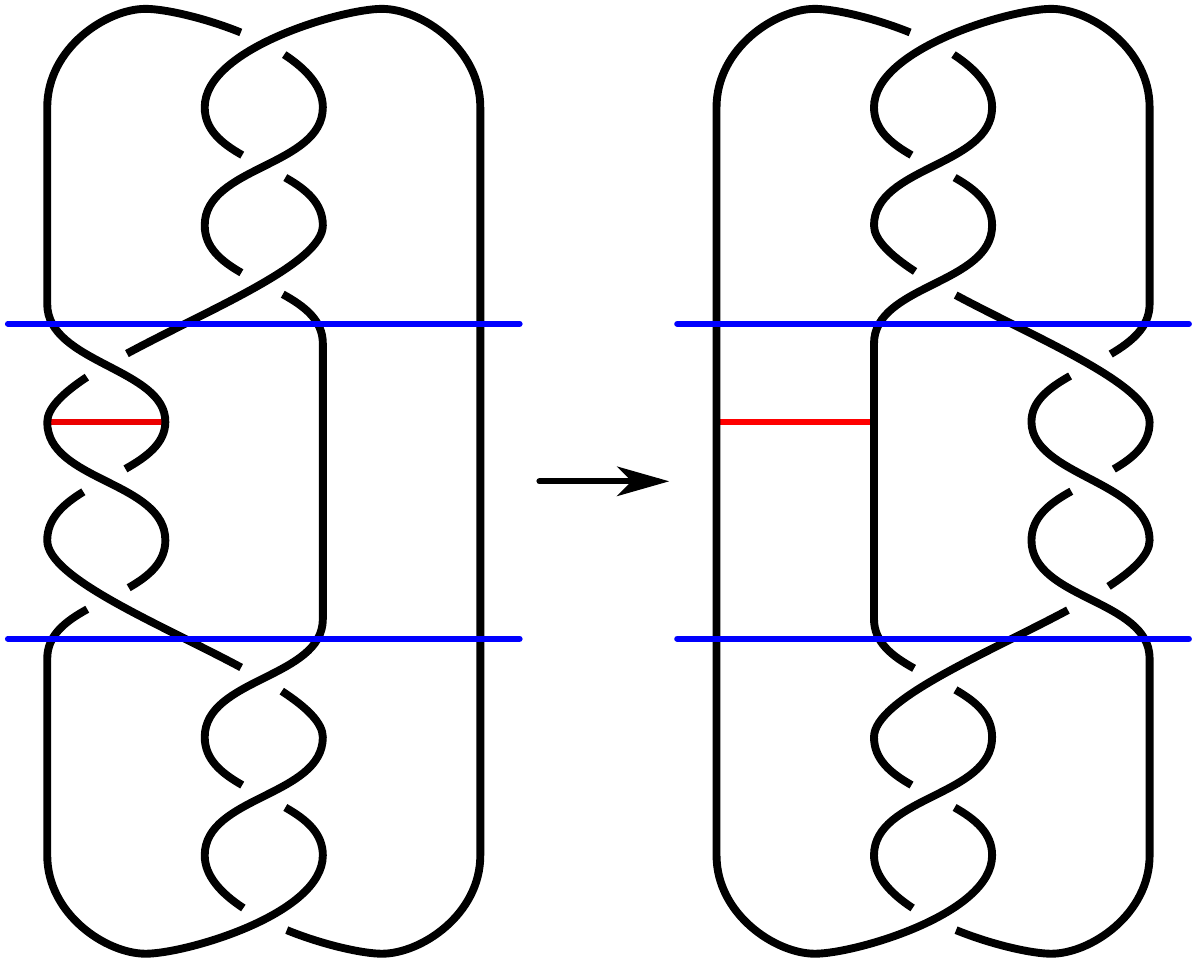}
    \put(-9,133){$\Diagram$}
    \put(90,133){$\Diagram'$}
    \put(-1,79){\color{red}$\gamma$}
    \put(115,70){\color{red}$\gamma'$}
    \put(79,62){\small flype}
  \end{overpic}
  \caption{The flype maps the crossing arc $\gamma$ in the diagram 
  $\Diagram$ to an arc which is not a crossing arc in the new diagram 
  $\Diagram'$.}
  \label{fig:flype1}
\end{figure}

Suppose $i$ is even.
Then we first modify $\Diagram$ by an ambient isotopy in $S^2$ (which is not an ambient isotopy in $\RR^2$) as in Figure~\ref{fig:flype2}, 
and then apply the flype as in Figure~\ref{fig:flype2}.
Then we can again apply the arguments in Subcase II-2
in the proof of Theorem~\ref{Theorem1-g}
and to obtain the same conclusion.

This completes the proof of Theorem~\ref{Theorem1}.
\end{proof}

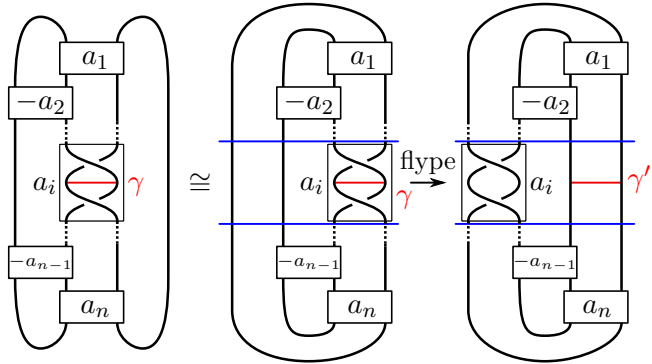
\begin{figure}
  \centering
  \begin{overpic}[width=0.6\columnwidth]{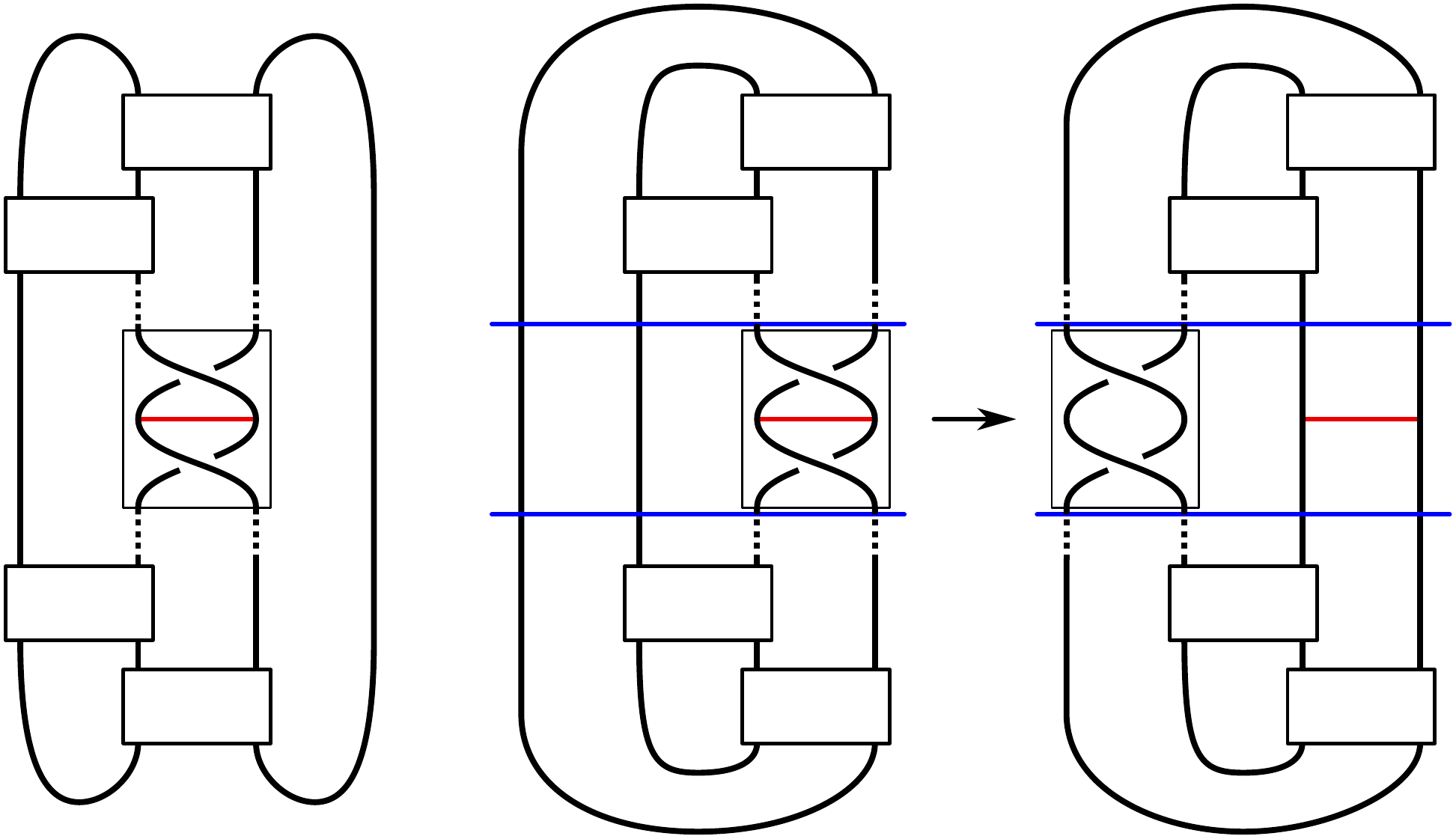}
    \put(28.5,113.5){$a_1$}
    \put(3.5,96.5){$-a_2$}
    \put(10,66){$a_i$}
    \put(1,37){\tiny$-a_{n-1}$}
    \put(28,19.5){$a_n$}
    \put(46,66){\textcolor{red}{$\gamma$}}
    \put(69,65){$\cong$}

    \put(130,113.5){$a_1$}
    \put(105,96.5){$-a_2$}
    \put(111,66){$a_i$}
    \put(102.5,37){\tiny$-a_{n-1}$}
    \put(129,19.5){$a_n$}
    \put(147.5,61){\textcolor{red}{$\gamma$}}
    \put(149,74){\small flype}

    \put(219,113.5){$a_1$}
    \put(194.5,96.5){$-a_2$}
    \put(198,66){$a_i$}
    \put(191.7,37){\tiny$-a_{n-1}$}
    \put(218,19.5){$a_n$}
    \put(235,67){\textcolor{red}{$\gamma'$}}
  \end{overpic}
  \caption{If $i$ is even,
then first modify $\Diagram$ by an ambient isotopy in $S^2$
into the middle diagram and then apply the flype.}
  \label{fig:flype2}
\end{figure}

\section{Rational links in the projective $3$-space and the proof of Theorem~\ref{thm:generalization}}
\label{sec:generalization}

In this section,
we first define the rational links in $P^3$ (Definition~\ref{def:rational-link-P})
and present their basic properties including classification and hyperbolization
(Propositions~\ref{prop:covering-link} and~\ref{prop:classifying-rational-P-link}).
Then we give a detailed description of Theorem~\ref{thm:generalization}(3)
in Remark~\ref{rem:statement-thm-general}, 
and prove the theorem.

We recall the definition of a rational tangle
following~\cite[Chapter 18]{Bonahon-Siebenmann}
and~\cite[Section 2]{AOPSY}.
Let $\Ball:=\{(x,y,z)\in \RR^3 \ | \ x^2+y^2+z^2\le2\}$ be the round $3$-ball
in $\RR^3\subset\RRR^3:=\RR^3\cup\{\infty\}$,
whose boundary contains the set $P^0$ consisting of 
the four marked points
\[
\pSW:=(-1,-1,0),\quad \pSE:=(1,-1,0), \quad
\pNE:=(1,1,0), \quad \pNW:=(-1,1,0).
\]
For $r\in \QQQ:=\QQ\cup\{\infty\}$, 
the {\it rational tangle of slope $r$} is 
the pair $(\Ball, t(r))$,
where $t(r)$ is
a pair of arcs properly embedded in $\Ball$
such that $t(r)\cap\partial\Ball=\partial t(r)= P^0$
as depicted in Figure~\ref{fig:rational-tangle}(b).
Here the \lq\lq pillowcase'' in the figure is 
the quotient space $(\RR^2,\ZZ^2)/\mathcal{J}$,
where $\mathcal{J}$ is the group of isometries of 
the Euclidean plane $\RR^2$ generated by the $\pi$-rotations around the points in $\ZZ^2$, 
and the pair of arcs on the pillowcase
is the image of the lines in $\RR^2$ of slope $r$ passing through 
points in $\ZZ^2$.
We can arrange $t(r)$ so that it is invariant by 
the $\pi$-rotations $h_x$, $h_y$ and $h_z=h_xh_y$
about the $x$-, $y$- and $z$-axis, respectively.

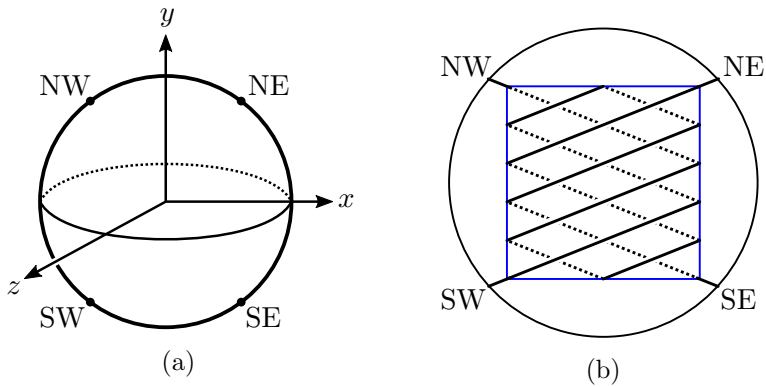
\begin{figure}
  \centering
  \begin{tabular}{c}
    \begin{subfigure}{0.3\columnwidth}
      \begin{overpic}[width=\columnwidth]{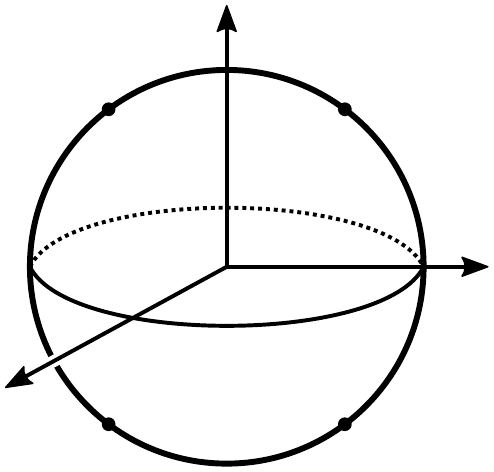}
        \put(120,47.5){$x$}
        \put(52.5,118){$y$}
        \put(-5.5,14.5){$z$}
        \put(7,90){NW}
        \put(86,90){NE}
        \put(7,2){SW}
        \put(85,2){SE}
      \end{overpic}
      \subcaption{}
      \label{fig:rational-tangle_a}
    \end{subfigure}
    \hspace{12mm}
    \begin{subfigure}{0.3\columnwidth}
      \begin{overpic}[width=\columnwidth]{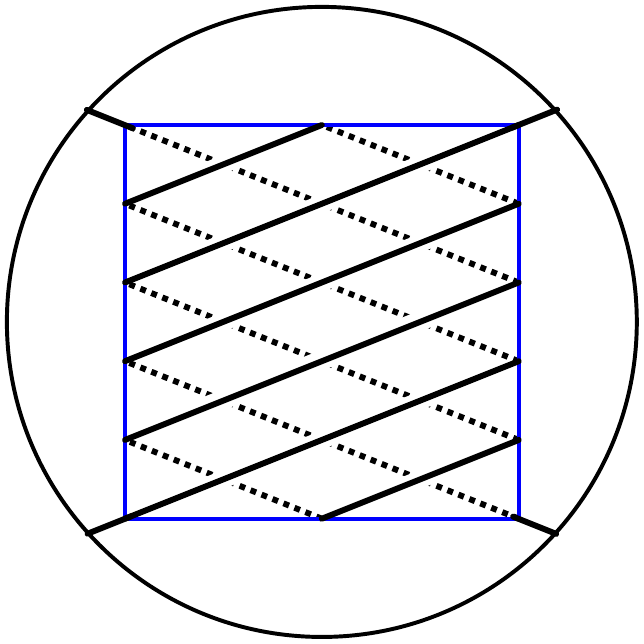}
        \put(-2,100){NW}
        \put(105,100){NE}
        \put(-2,12){SW}
        \put(104,12){SE}
      \end{overpic}
      \subcaption{}
      \label{fig:rational-tangle_b}
    \end{subfigure}
  \end{tabular}
  \caption{(a) The $3$-ball $\Ball$ with the set $P^0$ of the four marked points. 
  (b) The rational tangle $(\Ball, t(r))$ with $r=2/5$.
Note that the vertical axis is the $y$-axis, not the $z$-axis.
}
  \label{fig:rational-tangle}
\end{figure}

The {\it $2$-bridge link
$(S^3,K(r))$
of slope $r$}
is obtained by gluing (disjoint copies of)
$(\Ball, t(r))$ and $(-\Ball, t(\infty))$
via the identity map on $\partial\Ball$.
(Here $\Ball$ inherits the natural orientation of $\RRR^3$.)
Thus we may regard 
\[
K(r)= t(r)\cup \iota(t(\infty))\subset \Ball\cup \iota(\Ball) =\RRR^3,
\]
where $\iota$ is the inversion of $\RRR^3$
in $\partial\Ball$.
Let $\DD$ be the
{\it Farey tessellation},
that is,
the tessellation of the upper half-space 
$\HH^2$ by ideal triangles which are obtained
from the ideal triangle with the ideal vertices $0, 1,
\infty \in \QQQ$ by repeated reflection in the edges.
Let $\Aut(\DD)$ be the automorphism group of $\DD$ and
$\Aut^+(\DD)$ the orientation-preserving subgroup of $\Aut(\DD)$.
The following proposition reformulates
(i) the classification of $2$-bridge links 
established by Schubert~\cite{Schubert}
and (ii) 
the hyperbolization of alternating link complements
proved by Menasco~\cite[Corollary~2]{Menasco}
by using Thurston's uniformization theorem of Haken manifolds~
\cite{Thurston1982}, 
applied to $2$-bridge link complements.

\begin{proposition}
\label{prop:classifying-2bridge}
{\rm (1)}
For two rational numbers $r, r'\in\QQQ$,
there is a homeomorphism $\psi:S^3\to S^3$ 
such that $\psi(K(r))=K(r')$
if and only if 
there is an element $\xi\in \Aut(\DD)$
that maps $\{r, \infty\}$ to $\{r',\infty\}$.
Moreover, $\psi$ can be chosen to be orientation-preserving
if and only if 
either (i) $\xi$ is orientation-preserving and
$(\xi(r),\xi(\infty))=(r',\infty)$ or
(ii) $\xi$ is orientation-reversing and
$(\xi(r),\xi(\infty))=(\infty,r')$.

{\rm (2)}
$K(r)$ is hyperbolic if and only if
$d(\infty, r)\ge 3$,
where $d$ is the edge path distance
in the $1$-skeleton of $\DD$.
\end{proposition}

Now, we define the rational links in $P^3$
and state their basic properties.

\begin{definition}
\label{def:rational-link-P}
{\rm
For $r\in\QQQ$,
the {\it rational link of slope $r$ in the projective $3$-space $P^3$} 
is the pair 
$(P^3,K_P(r)):=(\Ball,t(r))/\sim$,
where $\sim$ identifies 
$x$ and $-x$ for every $x\in \partial \Ball$.
The inverse image $\tilde K_P(r)$
of $K_P(r)$ in the universal cover $S^3$ of $P^3$
is called the {\it covering link} of $K_P(r)$. 
}
\end{definition}

\begin{proposition}
\label{prop:covering-link}
The covering link of a rational link $K_P(r)$ in $P^3$
is equivalent to the $2$-bridge link $K(\tilde r)$
with $\tilde r=\eta_r(r)$,
where $\eta_r$ is an element
of $\Aut^+(\DD)$ such that $\eta_r(-r)=\infty$.
(In other words, $\tilde r$ is characterized by the property that
$(\tilde r,\infty)=(\eta_r(r),\eta_r(-r))$ for some
$\eta_r\in \Aut^+(\DD)$.) 
\end{proposition}

Here, we assume that $P^3$ inherits the natural orientation of
$\Ball\subset \RRR^3 \cong S^3$, and so the
covering projection $S^3\to P^3$ is orientation-preserving. 
Two links in an oriented $3$-manifold are  
{\it equivalent}
if there is an orientation-preserving homeomorphism of the ambient 
$3$-manifold that maps one to the other.

\begin{proof}[Proof of Proposition~\ref{prop:covering-link}]
Identify $S^3:=\{(z_1,z_2)\in \CC^2 \ | \ |z_1|^2+|z_2|^2=1\}$
with the spherical join 
$S^1_1*S^1_2$ 
of the circles $S^1_1:=S^3\cap (\CC\times 0)$
and $S^1_2:=S^3\cap (0\times \CC)$
(cf.~\cite[Definition I.5.13]{BH}).
Then we can identify $\RRR^3$
with $S^3$
so that the following conditions are satisfied 
(see Figure~\ref{fig:free-symmetry_a}).
\begin{enumerate}
\item
The great circle $\partial\Ball\cap \{y=0\}$ is identified with $S^1_1$,
and the compactified $y$-axis is identified with $S^1_2$.
Moreover $\Ball$ is identified with the spherical join $S^1_1*J_2$,
where $J_2:=\{(0,z_2)\in S^1_2 \ | \ -\pi/2 \le \arg(z_2) \le \pi/2\}$.
\item
The $\pi$-rotations $h_x$, $h_y$, $h_z$ of $\RRR^3$,
respectively, 
are identified with 
the involutions on $S^3$ defined by
\[
\quad h_x(z_1,z_2)=(\bar z_1, \bar z_2), \quad 
h_y(z_1,z_2)=(-z_1, z_2), \quad
h_z(z_1,z_2)=(-\bar z_1, \bar z_2).\]
\item
Let $f$ be the generator of the covering transformation group
of the covering $S^3\to P^3$, given by $f(z_1,z_2)=(-z_1,-z_2)$.
Then $f$ viewed on $\RRR^3$ is the composition
of the antipodal map $(x,y,z)\mapsto (-x,-y,-z)$
and the inversion $\iota$ in $\partial\Ball$.
\end{enumerate}
Then 
the the covering link $\tilde K_P(r)\subset S^3$ of $K_P(r)\subset P^3$
is given by $\tilde K_P(r)=t(r)\cup f(t(r))\subset \Ball \cup  f(\Ball)=S^3$, and it 
is invariant by the action of the subgroup
$\langle h_x,h_y,f\rangle \cong (\ZZ/2\ZZ)^3$ of
$\Isom^+(S^3)$.
Note that $f(t(r))=fh_z(t(r))$,
where $fh_z$, which is given by $fh_z(z_1,z_2)=(\bar z_1, -\bar z_2)$,
is the $\pi$-rotation of $S^3=S^1_1*S^1_2$ whose axis
is the spherical join $S^0_1*iS^0_2$,
where $S^0_1=\{(\pm 1,0)\}$ and $iS^0_2=\{(0,\pm i)\}$.
The axis of $fh_z$ viewed in $\RRR^3$ is the great circle
$\partial\Ball\cap \{z=0\}$,
which passes through the set $P^0$. 
Hence the action of $fh_z$ on $(S^3,\tilde K_P(r))$
is conjugate to the involution illustrated in
Figure~\ref{fig:free-symmetry_b},
where $\tilde K_P(r)$ is represented as the \lq\lq sum'' of the two rational tangles of slope $r$.
Note that the right rational tangle in the figure corresponds to
the image of $(\Ball, t(-r))$ by the inversion $\iota$.
So, we have $(S^3,\tilde K_P(r))\cong 
(\Ball, t(r))\cup \iota(\Ball, t(-r))$.

Now, let $\eta_r\in \Aut^+(\DD)$ and $\tilde r\in\QQQ$ be such that  
$(\tilde r,\infty)=(\eta_r(r),\eta_r(-r))$.
Recall the isomorphism $\Aut^+(\DD) \cong \SL(2,\ZZ)$,
and let $A\in \SL(2,\ZZ)$ be the matrix corresponding to $\eta_r$.
Then the linear map $A:\RR^2\to \RR^2$ maps the lines of slope $r$ (resp.~$-r$)
to the lines of slope $\tilde r$ (resp.~$\infty$).
Thus $A$ induces an orientation-preserving auto-homeomorphism of the pillowcase
$(\RR^2,\ZZ^2)/\mathcal{J}$ which maps 
the pair of proper arcs 
of \lq\lq slope'' $r$ (resp.~$-r$) to the pair of proper arcs 
of slope $\tilde r$ (resp.~$\infty$).
This homeomorphism induces 
an orientation-preserving auto-homeomorphism 
of $(\partial \Ball, P^0)$ via the natural
identification $(\partial \Ball, P^0)\cong(\RR^2,\ZZ^2)/\mathcal{J}$.
By using the fact that $t(s)\subset\Ball$ is boundary parallel 
for every $s\in\QQQ$,
we can
extend the homeomorphism to an orientation-preserving homeomorphism 
from $(S^3,\tilde K_P(r))\cong (\Ball, t(r))\cup \iota(\Ball, t(-r))$ 
to 
$(S^3,K(\tilde r))=(\Ball, t(\tilde r))\cup \iota(\Ball, t(\infty))$.
\end{proof}

\begin{figure}
  \centering
  \begin{tabular}{c}
    \begin{subfigure}{0.35\columnwidth}
      \begin{overpic}[width=\columnwidth]{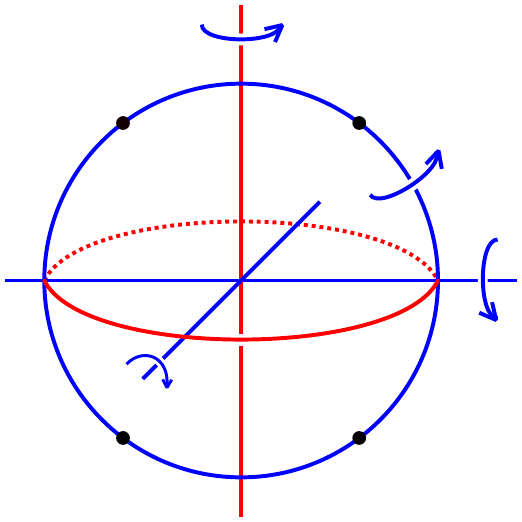}
        \put(12.5,108){NW}
        \put(99,108){NE}
        \put(12.5,14){SW}
        \put(99,14){SE}
        \put(77,128){\textcolor{blue}{$h_y$}}
        \put(134,49){\textcolor{blue}{$h_x$}}
        \put(40,27){\textcolor{blue}{$h_z$}}
        \put(120,95){\textcolor{blue}{$f h_z$}}
        \put(80,39){\textcolor{red}{$S^1_1$}}
        \put(60,143){\textcolor{red}{$S^1_2$}}
      \end{overpic}
      \vspace{2mm}
      \subcaption{}
      \label{fig:free-symmetry_a}
    \end{subfigure}
    \hspace{7mm}
    \begin{subfigure}{0.4\columnwidth}
      \begin{overpic}[width=\columnwidth]{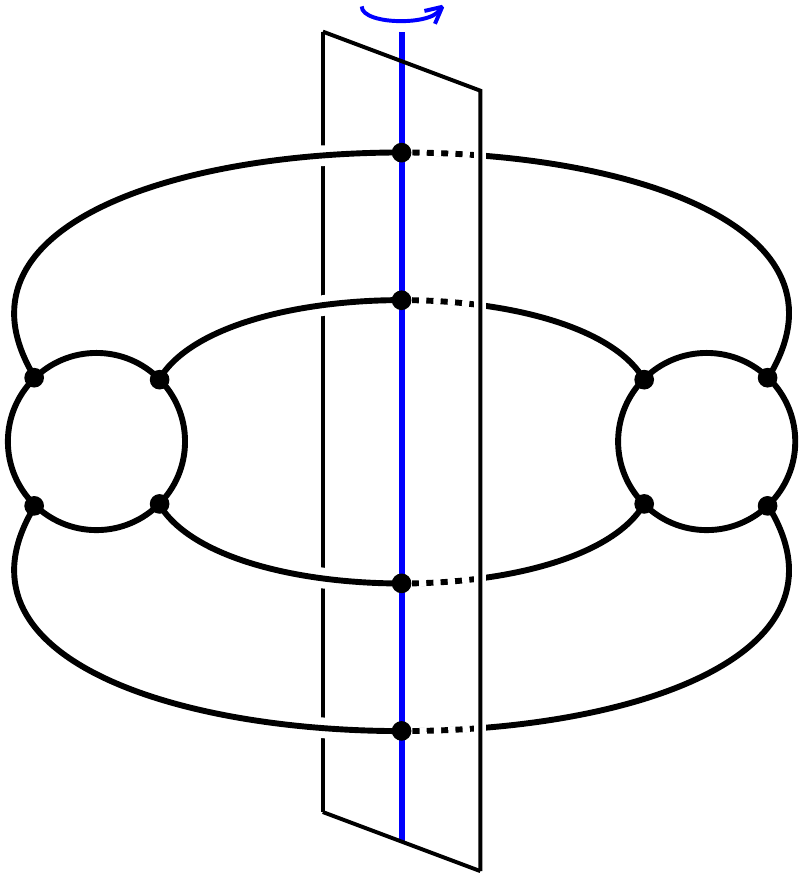}
        \put(65,136){\tiny NW}
        \put(67,107){\tiny NE}
        \put(68,51){\tiny SE}
        \put(66,22){\tiny SW}
        \put(11,84){$t(r)$}
        \put(133,84){$t(r)$}
        \put(20,148){$B^3$}
        \put(110,148){$f(B^3) = f h_z(B^3)$}
        \put(89,170){\textcolor{blue}{$f h_z$}}
      \end{overpic}
      \subcaption{}
      \label{fig:free-symmetry_b}
    \end{subfigure}
  \end{tabular}
  \caption{In (a), 
  the axis of the $\pi$-rotation $fh_z$
  is the great circle
  $\partial\Ball\cap \{z=0\}$,
  and it passes through the four marked points.
  In (b),
  $\partial B^3$ is the central vertical plane and the axis 
  of $fh_z$ is the vertical line.
  The free involution $f$ is the composition of 
  the $\pi$-rotations $fh_z$ and $h_z$,
  where $\Fix(fh_z)\cup \Fix(h_z)$ forms a Hopf link.
  }
  \label{fig:free-symmetry}
\end{figure}
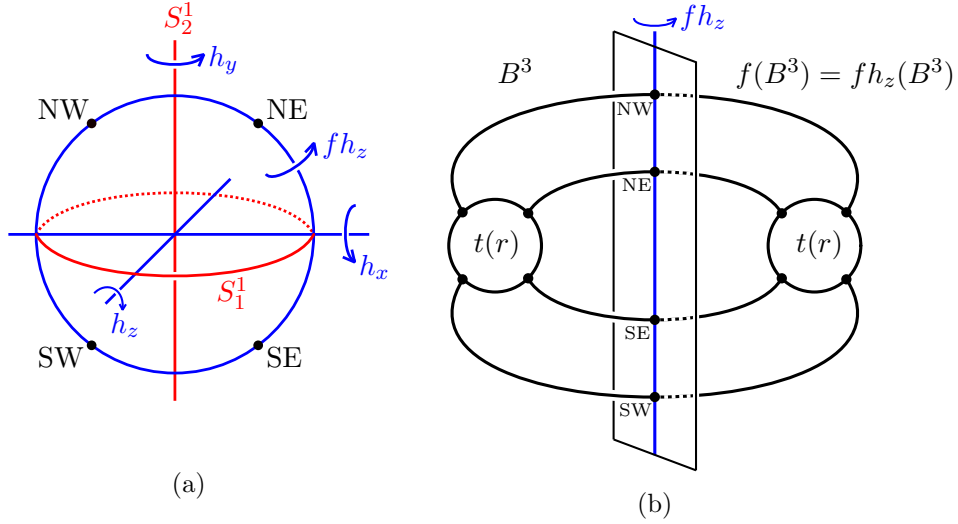

\begin{remark}
\label{rem:continued-fraction}
{\rm
By using~
\cite[Proof of Lemma II.3.3(3) and Figure II.3.4]{Sakuma-Weeks},
we obtain the following expression of $\tilde r$.
Consider a continued fraction expansion
\begin{center}\begin{picture}(230,70)
\put(0,48){$\displaystyle{
r=a_0+[a_1,a_2, \cdots,a_{n}] =
a_0+\cfrac{1}{a_1+
\cfrac{1}{ \raisebox{-5pt}[0pt][0pt]{$a_2 \, + \, $}
\raisebox{-10pt}[0pt][0pt]{$\, \ddots \ $}
\raisebox{-12pt}[0pt][0pt]{$+ \, \cfrac{1}{a_{n}}$}
}} \ .}$}
\end{picture}\end{center}
Then 
\[
\tilde r =
\begin{cases}
(-1)^{n-1} [a_n,\cdots, a_{1}, 2 a_0, a_1, \cdots, a_n]
& \mbox{if $a_0\ne 0$},
\\
(-1)^{n-1} [a_n,\cdots, a_{2}, 2a_{1}, a_2,\cdots, a_n]
& \mbox{if $a_0= 0$.}
\end{cases}
\]
Moreover, if $\tilde r=\tilde q/\tilde p$ with $\gcd(\tilde p,\tilde q)=1$
then $\tilde q^2 \equiv 1 \pmod{2\tilde p}$.
}
\end{remark}

Propositions~\ref{prop:classifying-2bridge} and~\ref{prop:covering-link} imply the following proposition for rational links in $P^3$.

\begin{proposition}
\label{prop:classifying-rational-P-link}
{\rm(1)}
$K_P(r)$ is trivial (i.e., it bounds a disk in $P^3$)
if and only if $r=0$ or $\infty$.

{\rm(2)}
For $r, r'\in \QQQ$,
there is a homeomorphism $\psi:P^3\to P^3$
such that $\psi(K_P(r))=K_P(r')$
if and only if $r'=\pm r$ or $\pm 1/r$.
Moreover, $\psi$ can be chosen to be orientation-preserving
if and only if $r'=r$ or $-1/r$.

{\rm(3)}
$K_P(r)$ is hyperbolic
if and only if 
$\min(d(0,r), d(\infty,r)) \ge 2$,
equivalently, 
$r\not\in \ZZ \cup\{\infty\}\cup \{1/p \ | \ p\in \ZZ\setminus \{0\}\}$.
\end{proposition}

\begin{proof}
(1)
Recall that $t(r)$ is boundary parallel in $\Ball$,
namely, there is a pair of disjoint disks $\Delta$ in $\Ball$,
such that $t(r)\subset \partial \Delta$
and $\closure(\partial\Delta\setminus t(r))=\Delta\cap\partial\Ball$.
If $r=0$ or $\infty$,
then the antipodal map interchanges the components of
$\closure(\partial\Delta\setminus t(r))$,
and so $\Delta$ descends to a disk in $P^3$ bounded by $K_P(r)$.
Hence $K_P(r)$ is trivial if $r=0$ or $\infty$.
Conversely, suppose that $K_P(r)$ is trivial.
Then its covering link $K(\tilde r)$ is the $2$-component trivial link,
and so $\tilde r=\infty$.
This implies $r=0$ or $\infty$ by Proposition~\ref{prop:covering-link}.

(2) 
If $r'=-1/r$,
then $(\Ball, t(r'))$ is obtained from 
$(\Ball, t(r))$ by $\pi/2$-rotation about the $z$-axis.
Since its restriction to $\partial\Ball$ is commutative
with the antipodal map,
it induces
an orientation-preserving homeomorphism $\psi:P^3\to P^3$
such that $\psi(K_P(r))=K_P(r')$.
Similarly, if $r'=-r$,
then $(\Ball, t(r'))$ is obtained from 
$(\Ball, t(r))$ by the reflection in the $xy$-plane.
Since its restriction to $\partial\Ball$ is commutative
with the antipodal map,
it induces  
an orientation-reversing homeomorphism $\psi:P^3\to P^3$
such that $\psi(K_P(r))=K_P(r')$.
The if part of (2) follows from these two observations.

Next, we prove the only if part of (2).
By (1), we may assume none of $r$ and $r'$ is equal to $0$ or $\infty$.
Then the following hold.
\begin{enumerate}
\item[(a)]
Let $\nu_0\in \Aut(\DD)$ be the reflection of $\DD$
in the Farey edge $\overline{0\infty}$,
i.e., $\nu_0$ is the element of $\Aut(\DD)$ such that $\nu_0(x)=-x$
for every $x\in\QQQ$.
Then, for any $r\in \QQQ\setminus\{0,\infty\}$, $\nu_0$ is the unique reflection of $\DD$
that interchanges $r$ and $-r$.
\item[(b)]
If $\xi\in \Aut(\DD)$ is commutative with $\nu_0$,
then the action of $\xi$ on $\QQQ$ is given by
$\xi(x)= x$, $-x$, $1/x$ or $-1/x$.
Here $\xi$ is orientation-preserving
if and only if $\xi(x)= x$ or $-1/x$.
\end{enumerate}
The observation (a) implies that,
for any $r\in \QQQ\setminus\{0,\infty\}$, 
if $\eta_r$ is an element of $\Aut^+(\DD)$ 
such that $(\eta_r(r),\eta_r(-r))=(\tilde r,\infty)$,
then 
$\nu_r:=\eta_r\nu_0\eta_r^{-1}$
is the unique reflection of $\DD$
that interchanges $\tilde r$ and $\infty$.

Now suppose that 
there is a homeomorphism $\psi:P^3\to P^3$
such that $\psi(K_P(r))=K_P(r')$,
where $r, r'\in \QQQ\setminus\{0,\infty\}$. 
Then $\psi$ lifts to a homeomorphism 
$\tilde\psi:S^3\to S^3$
which maps the covering link $K(\tilde r)$ of $K_P(r)$
to the covering link $K(\tilde r')$ of $K_P(r')$.
By Proposition~\ref{prop:classifying-2bridge}(1),
there is an automorphism $\xi\in \Aut(\DD)$
which maps $\{\tilde r, \infty\}$ to $\{\tilde r', \infty\}$.
By the uniqueness of the reflections $\nu_r$ and $\nu_{r'}$,
we have $\nu_{r'}=\xi\nu_r\xi^{-1}$.
Again, by the uniqueness of the reflection $\nu_0$,
this in turn implies that the conjugation of $\nu_0$ by
$\xi_0:=\eta_{r'}^{-1}\xi\eta_r$ is $\nu_0$,
i.e., $\nu_0$ and $\xi_0$ are commutative.
Hence, by the observation (b),
the action of $\xi_0$ on $\QQQ$ is given by
$\xi_0(x)= x$, $-x$, $1/x$ or $-1/x$.
On the other hand, $r'=\eta_{r'}^{-1}(\tilde r')$
is equal to either 
$\eta_{r'}^{-1}(\xi(\tilde r))=
\eta_{r'}^{-1}(\xi(\eta_r(r)))=\xi_0(r)$ or
$\eta_{r'}^{-1}(\xi(\infty))=
\eta_{r'}^{-1}(\xi(\eta_r(-r)))=\xi_0(-r)$.
Since $\xi_0$ is equal to one of the four transformations
in the above, we see that
$r'$ is equal to $\pm r$ or $\pm 1/r$ as desired.
This completes the proof of the first assertion of (2).
The second assertion of (2) can be proved by refining the above 
arguments by using the second assertion of 
Proposition~\ref{prop:classifying-2bridge}(1).

(3)
Since $K_P(r)$ is hyperbolic if and only if $K(\tilde r)$ is hyperbolic,
Proposition~\ref{prop:classifying-2bridge}(2) implies that
$K_P(r)$ is hyperbolic 
if and only if $d(\infty, \tilde r)\ge 3$.
On the other hand,
since the Farey edge $\overline{0\infty}$ separates
$-r$ and $r$,
we have 
\[ d(\infty, \tilde r)=
d(-r,r)=2\min(d(\infty, r),d(0, r)).\]
Hence $K_P(r)$ is hyperbolic if and only if
$\min(d(\infty, r),d(0, r))\ge 2$.
It is obvious that the latter condition is 
equivalent to the condition 
$r\not\in \ZZ \cup\{\infty\}\cup \{1/p \ | \ p\in \ZZ\setminus \{0\}\}$.
\end{proof}

By Proposition~\ref{prop:classifying-rational-P-link},
we have the following description of 
the statement (3) of Theorem~\ref{thm:generalization}.

\begin{remark}
\label{rem:statement-thm-general}
{\rm
In the setting of 
Theorem~\ref{thm:generalization}(3), the following hold.
$X=\HH^3/G$ is the complement
of a hyperbolic rational link $K_P(r)$ in $P^3$
for some $r\in\QQQ\setminus (\ZZ \cup\{\infty\}\cup 
\{1/p \ | \ p\in \ZZ\setminus \{0\}\})$,
$\Gamma=\langle \mu_1, \mu_2 \rangle$ is an
index $2$ subgroup of $G$,
and $\HH^3/\Gamma$ is the complement of
the $2$-bridge link $K(\tilde r)$,
where $\tilde r$ is characterized by the property that
$(\tilde r,\infty)=(\eta(r),\eta(-r))$ for some
$\eta\in \Aut^+(\DD)$.
In the group $\Gamma=\pi_1(S^3\setminus K(\tilde r))$,
$\{ \mu_1, \mu_2 \}$ is equivalent to 
the upper or lower meridian pair of the
$2$-bridge link $K(\tilde r)$.
In the group $G= \pi_1(P^3\setminus K_P(r))$,
$\{ \mu_1, \mu_2 \}$ is a meridian pair of 
the rational link $K_P(r)$,
such that $G/\langle \mu_1,\mu_2\rangle \cong \pi_1(P^3)\cong \ZZ/2\ZZ$.
}
\end{remark}

The following proposition, 
obtained by using 
the result of Millichap-Worden~\cite[Corollary 1.2]{Millichap-Worden}
on the commensurable classes of hyperbolic $2$-bridge links, 
plays a key role in the proof of Theorem~\ref{thm:generalization}.

\begin{proposition}
\label{prop:free-peridic-2bridge}
If the complement of a hyperbolic $2$-bridge link $K(\tilde r)$
non-trivially covers an orientable, complete hyperbolic manifold $X$,
then $X$ is the complement of a hyperbolic rational link 
$K_P(r)$ in $P^3$,
and $K(\tilde r)$ is the covering link of $K_P(r)$.
Thus the covering is a double covering, and 
$\tilde r$ is characterized by the property that
$(\tilde r,\infty)=(\eta(r),\eta(-r))$ for some
$\eta\in \Aut^+(\DD)$.
Moreover, the image of the upper and lower meridian pairs 
of the link group of $K(\tilde r)$ in $\pi_1(X)$
are meridian pairs of $K_P(r)$.
\end{proposition}

\begin{proof}
By~\cite[Corollary 1.2]{Millichap-Worden},
a hyperbolic $2$-bridge link complement
covers a hyperbolic manifold $X$ non-trivially,
then it is a regular covering. 
The isometry group of hyperbolic $2$-bridge link complements
are calculated by~\cite[Proposition 4.1]{ALSS}
(cf. ~\cite[Theorem 4.1]{Sakuma1990}).
As suggested by Boileau-Weidmann~\cite[Lemma 15]{Boileau-Weidmann},
the calculation implies that 
(i) the complement of the hyperbolic $2$-bridge link 
$K(\tilde r)$ with $\tilde r=\tilde q/\tilde p$ 
admits an orientation-preserving free isometry
if and only if $\tilde q^2 \equiv 1 \pmod{2\tilde p}$ and 
(ii) any such hyperbolic $2$-bridge link complement admits a unique
orientation-preserving free isometry.
In fact, the orientation-preserving isometry group 
$\Isom^+(S^3\setminus K(\tilde r))$ for 
such a $2$-bridge link $K(\tilde r)$
is isomorphic to $(\ZZ/2\ZZ)^3$.
Moreover, it extends to the $(\ZZ/2\ZZ)^3$-action of $(S^3,K(\tilde r))$
generated by $\{h_x, h_y, f\}$ as illustrate
in Figure~\ref{fig:free-symmetry};
we can easily check that $f$ is the unique element 
which acts on the link complement (and also on $S^3$) freely.
(See Bonahon-Siebenmann~\cite[Chapter 18]{Bonahon-Siebenmann}
for nice description of link symmetries as rigid motions of $S^3$.)
This fact together with Proposition~\ref{prop:covering-link} implies the first assertion.
The last assertion is obvious.
\end{proof}

\begin{proof}[Proof of Theorem~\ref{thm:generalization}]
Let $X = \HH^3/G$ and $\{ \mu_1, \mu_2 \}$ be as in Theorem~\ref{thm:generalization},
and let $\Gamma=\langle \mu_1, \mu_2\rangle$ be the subgroup of $G$
generated by $\{ \mu_1, \mu_2 \}$.
Then, since $\Gamma<G$ is torsion-free,
Theorem~\ref{thm:Agol} implies that
$\Gamma$ is either a rank $2$ free group 
or a hyperbolic $2$-bridge link group.
In the former case, the conclusion (1) holds.
In the latter case, 
$X=\HH^3/G$ is covered by
the hyperbolic $2$-bridge link complement $\HH^3/\Gamma$.
Hence, by
Proposition~\ref{prop:free-peridic-2bridge},
either (i) $\Gamma=G$ and the conclusion (2) holds by 
Theorem~\ref{Theorem1-0}
(or Theorem~\ref{thm:Agol}) or
(ii) $\Gamma$ is a proper subgroup of $G$ and the conclusion (3) holds.
This completes the first assertion of 
Theorem~\ref{thm:generalization}.

In order to prove the second assertion,
assume that 
$X=\HH^3/G$ has finite volume
and $\Gamma$ is a rank $2$ free group.
Suppose to the contrary that $\Gamma$ is geometrically infinite.
Since the codomain $X$ of the covering $p : \hat{X}= \HH^3/\Gamma \to X = \HH^3/G$ has finite volume 
  and since $\hat{X}$ is tame by
  the tameness theorem (\cite{Agol2004, Bowditch, Calegari-Gabai, Soma}),  
  the covering theorem of Canary~\cite{Canary}
  implies that $X$ has a finite cover $X'$ which fibers over the circle, 
such that the cover $X_S$ of $X'$ associated to a fiber subgroup
satisfies one of the following conditions.
\begin{enumerate}
\item[(a)] 
$\hat{X} = X_S$.
\item[(b)]
$\hat{X}$ is a twisted $I$-bundle which is doubly covered by $X_S$.
\end{enumerate}

  Suppose first that $\hat{X} = X_S$.
Then there is a fuchsian group $\Gamma_0$ of co-finite volume,
such that 
(i) the hyperbolic surface $\HH^2/\Gamma_0$ is homeomorphic to
the fiber surface $S$ of the bundle $X'$ over $S^1$,
and (ii) there is an isomorphism $\rho:\Gamma_0 \to \Gamma$ which is 
strictly type-preserving, i.e., for $g \in \Gamma_0 < \Isom^+(\HH^2)$, $\rho(g)$ is parabolic if and only if $g$ is parabolic.
Since $\Gamma$ is generated by two parabolic elements, $S$ must be a thrice-punctured sphere.
This contradicts the assumption that $S$ is a fiber surface of $X'$, because a 
thrice-punctured sphere does not admit a pseudo-Anosov homeomorphism.

Suppose next that $\hat{X}$ is a twisted $I$-bundle which is doubly covered by $X_S$.
Then there is a non-orientable
hyperbolic surface 
$F = \HH^2/\Gamma_0$,
where $\pi_1(F) \cong \Gamma_0 < \Isom(\HH^2) < \Isom^+(\HH^3)$, and 
a strictly type-preserving isomorphism $\rho: \Gamma_0 \to \Gamma$.
(Here $F$ is homeomorphic to the base space of the twisted $I$-bundle $\hat X$.)
This contradicts the fact that there is no non-orientable surface whose fundamental group is generated by peripheral elements.
Hence $\Gamma$ is geometrically finite.
\end{proof}

\bibstyle{plain}

\end{document}